\documentclass[reqno,11pt]{amsart}
\usepackage{xcolor}
\usepackage[top=2.0cm,bottom=2.0cm,left=3cm,right=3cm]{geometry}
\usepackage{amsthm,amsmath,amssymb,dsfont}
\usepackage{mathrsfs,amsfonts,functan,extarrows,mathtools}

\usepackage[colorlinks,
linkcolor=red,
citecolor=blue
]{hyperref}

\usepackage{xcolor}
\usepackage{cite}
\usepackage{indentfirst, latexsym, amssymb, enumerate,amsmath,graphicx}
\usepackage{float}
\usepackage{relsize}
\usepackage{marginnote}
\usepackage{stmaryrd}
\usepackage{esint}
\usepackage{graphicx}
\usepackage{bm}
\usepackage{caption}
\usepackage{subfigure}
\usepackage{cite}
\usepackage{color}
\usepackage{graphicx}
\usepackage{subfigure}
\usepackage{float}
\usepackage{paralist}
\usepackage{indentfirst}
\usepackage{cite}
\usepackage{mathrsfs}
\usepackage{amsfonts}
\allowdisplaybreaks %[4]
\theoremstyle{plain}
\newtheorem{thm}{Theorem}[section]
\newtheorem{cor}[thm]{Corollary}
\newtheorem{lem}[thm]{Lemma}
\newtheorem{prop}[thm]{Proposition}
\newtheorem{rem}{Remark}[section]

\newtheorem{defn}{Definition}[section]

\numberwithin{equation}{section}

\DeclareMathOperator{\dive}{div}

\usepackage{appendix}
\usepackage{xcolor}

\DeclareMathOperator{\A}{\mathcal{A}}

\DeclareMathOperator{\K}{\mathcal{K}}

\newcommand{\dv}{{\rm div\,}}

\newcommand{\eps}{{\epsilon}}

\begin{document}

\title[low mach number limit for the NSK system] {Low Mach number limit for the Navier--Stokes--Korteweg equations with a stationary force}

\author[J. Ni]{Jinkai Ni$^*$}  \thanks{$^*$\! Corresponding author}
\address[JKN]{School  of Mathematics, Nanjing University, Nanjing 
 210093, P. R. China}
\email{jinkaini123@gmail.com}

\author[L. Wang] {Luqi Wang}
\address[LQW]{School of Mathematics, Nanjing University, Nanjing
 210093, P. R. China}
\email{wangluqi@nju.edu.cn}

\author[Y. Zhang]{Yichi Zhang}
\address[YCZ]{The Institute of Mathematical Sciences, The Chinese University of Hong Kong, Hong Kong, China}
\email{ZhangYichi@link.cuhk.edu.hk}

\author[Z. Zhang]{Zhipeng Zhang}   
\address[ZPZ]{School of Mathematical Sciences, Ocean University of China, Qingdao
 266100, P. R. China}
\email{zhangzp@ouc.edu.cn}

\begin{abstract}
In this paper, we investigate the low Mach number limit for the three-dimensional compressible Navier--Stokes--Korteweg equations in the whole space under a small stationary external force. We first construct a family of small stationary solutions uniformly with respect to the Mach number $\eps$ and prove that both the stationary density fluctuation and the compressible component of the stationary velocity are of order $\eps^2$. 
For ill-prepared non-stationary perturbations around these stationary solutions, we establish the existence and uniqueness of global strong solution by combining uniform high-order energy estimates with a low-frequency Besov estimate and a Kawashima-type compensating functional.
The main difficulty is that Korteweg tensor not only changes the elliptic structure of the stationary problem, but also modifies the dispersive mechanism of the acoustic modes. 
In Korteweg-symmetric variables, the associated spectral projections are uniformly bounded zero-order Fourier multipliers, while the acoustic-capillary phase is wave-like at low frequencies and Schr\"odinger-like at high frequencies. Since the source terms generated by the stationary coefficients are generally not integrable in time, we decompose the Duhamel source according to its time-integrability and frequency behavior. Dyadic dispersive estimates, high-frequency damping estimates, and maximal regularity for the heat equation yield the global-in-time convergence rate $\eps^{\min\{1/r,\,1/2-1/p\}}$ in the mixed Besov norms \(L^r(0,\infty;\dot B^s_{p,1})\). As a consequence, Besov embeddings also yield quantitative convergence in the mixed Lebesgue norms \(L^r(0,\infty;L^p)\).
\end{abstract}


\makeatletter
\@namedef{subjclassname@2020}{\textup{2020} Mathematics Subject Classification}
\makeatother


\maketitle

%\tableofcontents

%-----------------section one-------------------------------------------------------------
%\renewcommand{\theequation}{\thesection.\arabic{equation}} 
\setcounter{equation}{0}
 \indent \allowdisplaybreaks

\section{Introduction and main results}

\subsection{Introduction}

In this paper, we investigate the following three-dimensional compressible Navier--Stokes--Korteweg (NSK) equations with a stationary external force:
\begin{equation}\label{NNSK-intro}
\begin{cases}
\partial_t \rho_\epsilon +\dive(\rho_\eps u_\eps) = 0, \\
\partial_t (\rho_\eps u_\eps)+\dive (\rho_\eps u_\eps \otimes u_\eps) + \dfrac{1}{\eps^2}\nabla p(\rho_\eps) - \mu \Delta u_\eps - \nu \nabla \dive u_\eps = \dfrac{\kappa}{\eps^2} \rho_\eps \nabla \Delta \rho_\eps + \rho_\eps F,\\
(\rho_\epsilon,u_\epsilon)|_{t=0}=(\rho_{\epsilon,0},u_{\epsilon,0}),\quad \displaystyle\lim_{|x|\to\infty}(\rho_\epsilon,u_\epsilon)=(\rho_{\infty},0),
\end{cases}
\end{equation}
in $(t,x)\in(0,+\infty)\times \mathbb{R}^3$,  
where  the unknowns $\rho_\eps(t,x)$ and $u_\eps(t,x)$ denote the density and velocity of the fluid, respectively, while $\rho_{\epsilon,0}$ and $u_{\epsilon,0}$ are the given initial data. Here, $\rho_\infty>0$ is the constant far-field density.
The barotropic pressure law obeys $p\in C^\infty$ and $p'(\rho_\infty)>0$. 
The vector field $F=F(x)$ is a prescribed stationary external force, not necessarily a potential field.
The parameter $0<\eps\leq1$ is the Mach number defined as
the ratio of the reference velocity to the reference sound speed in the fluid.
The viscosity coefficients satisfy $\mu>0$ and $\nu>0$. The capillary coefficient $\kappa$ is a positive constant. 
In the scaling adopted in \eqref{NNSK-intro}, the Korteweg force is of order $\eps^{-2}$, which is different from previous studies on the low Mach number limit of the compressible NSK equations, and we will elaborate on this below.

As an important extension of the compressible Navier--Stokes equations, 
the compressible NSK equations incorporate capillary and diffuse-interface effects into the modeling of compressible viscous fluids via the Korteweg tensor in the momentum equation, 
thus allowing for the description of the dynamics of a compressible fluid with internal capillarity. 
The general Korteweg tensor is written as
\begin{align}
\mathbb{K} = \left( {\rho_\epsilon \dv(\kappa(\rho_\epsilon )\nabla \rho_\epsilon ) + \left( {\kappa(\rho_\epsilon ) - \rho\kappa'(\rho_\epsilon )} \right){{\left| {\nabla \rho_\epsilon } \right|}^2}} /2\right)\mathbb{I}_3 -  {\kappa(\rho_\epsilon )\nabla \rho_\epsilon  \otimes \nabla \rho_\epsilon },\nonumber
\end{align}
where $\mathbb{I}_3$ denotes the third-order identity matrix.
In particular, for a constant capillarity coefficient, the Korteweg tensor reduces, up to the pressure-gradient convention, to the third-order term $\kappa\rho_\epsilon\nabla\Delta\rho_\epsilon$ appearing in \eqref{NNSK-intro}.
Its origin goes back to Korteweg's constitutive law for fluids with density-gradient forces \cite{Ko-1901}.  A modern version was proposed by Dunn
and Serrin \cite{Du-Se-1985}, where they developed a thermodynamically consistent continuum formulation.
In addition, this model also permits describing some phase transition phenomena \cite{An-Mc-Wh-1998}. 
Besides its physical role in the description of liquid-vapor interfaces and phase transitions, the Korteweg tensor  also changes the analytic structure of the compressible modes by adding one derivative to the density energy and a genuinely dispersive high-frequency component.

Due to its physical importance and mathematical challenges, the compressible NSK equations have attracted considerable attention.
For the well-posedness theory, 
the pioneering works of Hattori and Li \cite{Ha-Li-1994,Ha-Li-1996} established global smooth solutions near equilibrium in multi-dimensional space. Danchin and Desjardins  \cite{Da-De-2001} introduced a critical-space approach for compressible Korteweg fluids, while Kotschote \cite{Ko-2008} developed a maximal-regularity theory for strong solutions. In critical and near-critical Besov frameworks, global well-posedness, nonlocal-to-local capillarity limits, and large-data local theory were studied in \cite{Ch-Ko-2019,Ch-Ha-2011,Ch-2014}. Gevrey regularity and decay were obtained in \cite{Ch-Da-Xu-2021}. The long-time behavior generated by the capillary dissipation has been investigated by Tan and Zhang 
\cite{Ta-Zh-2014} and by Kawashima, Shibata and Xu \cite{Ka-Sh-Xu-2021}. In the presence of external forces or nontrivial reference states, significant contributions include, but are not limited to, the global existence and optimal decay obtained by Li \cite{Li-2012}, the potential-force analysis of Wang and Wang \cite{Wa-Wa-2015}, the nonlinear stability of stationary solutions to the full NSK equations by Chen and Zhao \cite{Ch-Zh-2014}, and the construction and stability of time-periodic solutions by Tsuda \cite{Ts-2016}. 
These works show that the Korteweg tensor supplies additional density regularity, but they also reveal that stationary coefficients and low frequencies require a treatment different from that used around a constant equilibrium.

The low Mach number limit (incompressible limit) of compressible fluid dynamics models, such as the compressible Euler equations and Navier--Stokes equations, is another important and challenging mathematical problem.
The first result  can be traced back to 
Klainerman and Majda \cite{Kl-Ma-1981,Kl-Ma-1982}, in which they proved the incompressible limit of the isentropic Euler equations to the incompressible Euler equations for local smooth solutions with well-prepared initial data. 
Using the dispersive mechanisms of acoustic waves,
Ukai \cite{Uk-1986} verified the low Mach number limit for the ill-prepared initial data in the whole space, see also \cite{Is-1987}.
In the viscous setting, Lions and Masmoudi \cite{Li-Ma-1998}, Hoff \cite{Ho-1998}, and Desjardins and Grenier \cite{De-Gr-1999} proved the low Mach number limit for global weak solutions of the isentropic Navier--Stokes equations, respectively.
Danchin \cite{Da-2002} established global convergence for ill-prepared data in scaling-critical spaces. Related progress for full systems, singular thermodynamic limits, $L^p$-critical data, large solutions, and refined convergence estimates can be found in \cite{Al-2006,Fe-No-2009,Da-He-2016,Da-Mu-2017,Fu-2024} and references therein. The common feature of these works is that the acoustic component must be separated from the solenoidal component and controlled by a combination of parabolic smoothing and dispersive estimates.

As the capillarity approximation of the compressible Navier--Stokes equations \cite{BYZ-2014},  
the low Mach number limit of the compressible NSK equations has also attracted attention. 
Nevertheless, the appearance of the Korteweg tensor gives rise to an extra singular term, and progress 
has been made for some special cases.
Li and Yong \cite{Li-Yo-2016} studied the low Mach number limit for the case of well-prepared initial data and $\kappa$ being of  order $O(\epsilon)$
in the three-dimensional whole space and periodic domain. Subsequently, Ju and Xu \cite{JX-2022} improved some results in \cite{Li-Yo-2016} by relaxing the restrictions on the initial data. 
Fujii and Li \cite{Fu-Li-2025} established the low Mach limit for large solutions of the two-dimensional compressible NSK equations with ill-prepared initial data and $\kappa$ being of order $O(\epsilon^2)$  
in critical Fourier--Besov spaces. 
As for the low Mach number limit of the full NSK equations, we refer to \cite{HLY-2024, LY-2026, SL-2019}.
The results mentioned above are all concerned with perturbations of constant states or settings without external forces. In contrast, the present problem involves a combination of the low Mach number limit, a stationary external force, and perturbations around a nontrivial stationary flow in the three-dimensional whole space. 
In particular, 
the capillarity coefficient is independent on Mach number, which makes us face stronger singularity than \cite{Fu-Li-2025, JX-2022, Li-Yo-2016}.

Stationary external force introduces a second and essentially independent difficulty.  In scaling-invariant or weak Besov spaces, stationary incompressible flows and their stability were considered by Kaneko, Kozono and Shimizu \cite{Ka-Ko-Sh-2019}, Cunanan, Okabe and Tsutsui \cite{Cu-Ok-Ts-2022}, and Kozono and Shimizu \cite{Ko-Sh-2023}. For compressible flows, the existence and nonlinear stability of stationary flows were studied by Shibata and Tanaka \cite{Sh-Ta-2003,Sh-Ta-2007}.
Deguchi established the stability and sharp decay around stationary solutions in \cite{De-2024}.
A central point in these works is that a stationary velocity in three dimensions may have a spatial tail of order $|x|^{-1}$, in agreement with the far-field behavior of steady Navier--Stokes flows \cite{Ko-Sv-2011}. Such a profile is naturally accommodated by $\dot B^{1/2}_{2,\infty}$ but not, in general, by the stronger space $\dot B^{1/2}_{2,1}$. This motivates the mixed weak-Besov/high-Sobolev framework used below. Standard facts concerning the Littlewood--Paley analysis of these spaces can be found in \cite{BCD-Book-2011}.

Although the Korteweg tensor gives rise to challenges, it also plays a stabilizing role.
Deguchi \cite{De-2025} recently established the global-in-time low Mach number limit for the compressible Navier--Stokes equations with a small stationary force and ill-prepared initial data in the three-dimensional whole space. In the absence of capillary effect, the stationary density fluctuation is of order \(O(\eps^2)\) only at the weak low-frequency level, whereas its full high-order estimate is of order \(O(\eps)\). Moreover, both the compressible velocity component and the incompressible projection error are generally of order \(O(\eps)\). Our main results show that the Korteweg tensor leads to a different stationary asymptotic regime: the elliptic operator 
$p'(\rho_\infty)-\kappa\rho_\infty\Delta  $
provides additional density regularity and yields second-order convergence simultaneously for the stationary density, and both velocity errors in the low- and high-frequency norms considered here. In the corresponding proofs,  compared with the non-stationary problem of Navier--Stokes equations,
the distinction is structural.  
The capillary term cannot be treated as a lower-order perturbation of the non-capillary acoustic dynamics: after the ill-prepared rescaling, the acoustic-capillary phase is wave-like at low frequencies and Schr\"odinger-like at high frequencies. A Korteweg-symmetric formulation is therefore required to obtain spectral projections that are uniformly bounded with respect to \(\eps\). In addition, the linear terms generated by the stationary solution are not generally integrable in time and cannot be included in a standard \(L^1_t\) Duhamel source. To overcome these difficulties, we combine uniform energy estimates, weak low-frequency Besov control, dyadic acoustic-capillary dispersive estimates, and a decomposition of the Duhamel source into time-integrable and stationary-coefficient components.
Our main results are stated in the following subsection.

\subsection{Main results}

\subsubsection{Low Mach number limit for the stationary case}
Consider the stationary NSK equations with Mach number 
\(0<\epsilon \leq 1\)
\begin{equation}\label{LMSNKS-intro}
\begin{cases}
\text{div}(\rho_\epsilon^* u_\epsilon^*) = 0, \\
\text{div}(\rho_\epsilon^* u_\epsilon^* \otimes u_\epsilon^*) + \frac{\nabla p(\rho_\epsilon^*)}{\epsilon^2} - \mu \Delta u_\epsilon^* - \nu \nabla \text{div} u_\epsilon^* = \frac{\kappa \rho_\epsilon^* \nabla \Delta \rho_\epsilon^*}{\epsilon^2} + \rho_\epsilon^* F(x),\\
\lim_{|x|\rightarrow\infty}(\rho_\epsilon^*,u_\epsilon^*)=(\rho_\infty,0),
\end{cases}
\end{equation}
and the stationary incompressible Navier--Stokes equations
\begin{equation}\label{SNS-intro}
\begin{cases}
\rho_\infty \text{div}(u^* \otimes u^*) = \mu \Delta u^* - \nabla \Pi^* + \rho_\infty F(x), \\
\text{div} u^* = 0.
\end{cases}
\end{equation}
Define 
\begin{align*}
   \mathbb{Q}=\nabla\Delta^{-1}\text{div},\quad  \mathbb{P}=\text{Id}-\mathbb{Q}.
\end{align*}
Then, the result on the low Mach number limit of the problem \eqref{LMSNKS-intro} is  the following theorem.
\begin{thm}\label{LMSNSKThm-intro}
Assume that $k\ge 3$. There exists \(\delta_0 > 0\) such that if
\begin{align*}
\|F\|_{\dot B_{2,\infty}^{-\frac{3}{2}} \cap \dot H^k} \leq \delta_0,
\end{align*}
then the problem \eqref{LMSNKS-intro} with \(0<\epsilon \leq 1\) and problem \eqref{SNS-intro} admit unique solutions \((\rho_\epsilon^*, u_\epsilon^*)\) and $u^*$, respectively. Moreover, it holds that
\begin{equation}\label{siguepsest-intro}
\|u^*\|_{\dot{B}_{2,\infty}^{\frac{1}{2}}\cap \dot H^{k+2}}+\|u_\epsilon^*\|_{\dot B_{2,\infty}^{\frac{1}{2}} \cap \dot H^{k+2}} + \epsilon^{-2} \|\rho_\epsilon^* - \rho_\infty\|_{\dot B_{2,\infty}^{-\frac{1}{2}} \cap \dot H^{k+3}}  \leq C \|F\|_{\dot B_{2,\infty}^{-\frac{3}{2}} \cap \dot H^k},
\end{equation}
and
\begin{equation}\label{machest-intro}
\|(\mathbb{Q} u_\epsilon^*, \mathbb{P} u_\epsilon^* - u^*)\|_{\dot B_{2,\infty}^{\frac{1}{2}} \cap \dot H^{k+2}}\le C\epsilon^2 \delta_0^2,
\end{equation}
where \(C\) is a constant independent of \(\epsilon\).
\end{thm}

\begin{rem}
Compared with the result in
\cite{De-2025} on the stationary compressible Navier--Stokes equations, the Korteweg tensor changes the elliptic structure of the density equation through the operator
\begin{align*}
p'(\rho_\infty)-\kappa\rho_\infty\Delta.    
\end{align*}
The inverse of this operator gains two derivatives and, together with the
divergence structure of the stationary equations, leads to 
\begin{align*}
 \eps^{-2}(\rho_\eps^*-\rho_\infty)
\in
\dot B^{-1/2}_{2,\infty}\cap\dot H^{k+3}.   
\end{align*}
Moreover, the compressible velocity and the incompressible
projection error satisfy
\begin{align*}
\mathbb Q u_\eps^*=O(\eps^2),
\qquad
\mathbb P u_\eps^*-u^*=O(\eps^2)   
\end{align*}
in the weak Besov and high-order Sobolev norms appearing in
\eqref{machest-intro}, while for the non-capillary stationary result of \cite{De-2025}, the corresponding velocity errors are only of order \(O(\eps)\).
\end{rem}

\begin{rem}
The proof of Theorem~\ref{LMSNSKThm-intro} is divided into two steps.
Theorem~\ref{SNSthm} first provides the unique stationary
incompressible solution \(u^*\), which serves as the limiting profile.
Section~\ref{Sec:SNSK} then constructs the stationary NSK solution
\((\rho_\eps^*,u_\eps^*)\) uniformly in \(\eps\) and proves the
estimates \eqref{siguepsest-intro}--\eqref{machest-intro}.
\end{rem}

\subsubsection{Low Mach number limit for the non-stationary case}
Let
\begin{align}\label{def-Psi-Phi-intro}
    \Psi(\rho):=\rho^{-1},\quad \Phi(\rho):=\frac{p'(\rho)}{\rho}, \quad \gamma_0:=\Phi(\rho_\infty)>0.
\end{align}
Formally, letting $\epsilon\rightarrow 0$ in \eqref{NNSK-intro}, the limiting system is the following non-stationary incompressible Navier--Stokes equations
\begin{equation}\label{NNS-intro}
\begin{cases}
\rho_\infty (\partial_t u+\text{div}(u \otimes u)) = \mu \Delta u - \nabla \Pi + \rho_\infty F(x), \\
\text{div} u = 0,\\
u|_{t=0}=\mathbb{P} u_0
\end{cases}
\end{equation}
Let $(\rho_\eps^*, u_\eps^*)$ be the unique solution of the stationary NSK equations \eqref{LMSNKS-intro} obtained in Theorem \ref{LMSNSKThm-intro}, satisfying $\rho_\eps^*-\rho_\infty=O(\eps^2)$. 
Define the perturbation
    \begin{equation}\label{Sigweps-intro}
        \sigma_\eps:=\eps^{-1}(\rho_\eps-\rho_\eps^*), \quad w_\eps:=u_\eps-u_\eps^*.
    \end{equation}
\begin{rem}
    For the non-stationary case with ill-prepared initial data, the compressible Helmholtz projection component $\mathbb{Q}(u_{\eps,0}-u_\eps^*)$ is not necessarily of order $O(\eps)$.
\end{rem}
The Korteweg operator
\begin{align*}
    \mathcal K:=\gamma_0-\kappa\Delta
\end{align*}
satisfies
\begin{align*}
\widehat{\mathcal K z}(\xi)=(\gamma_0+\kappa|\xi|^2)\hat{z}(\xi),
\end{align*}
thus we have
\begin{equation}\label{Khalfest-intro}
    \|\mathcal K^{\frac{1}{2}}z\|_{\dot H^m}^2=\gamma_0\|z\|_{\dot H^m}^2+\kappa\|\nabla z\|_{\dot H^m}^2,
\end{equation}
and
\begin{align*}    
\|\mathcal K^{\frac{1}{2}}z\|_{\dot B^s_{p,r}}\sim \|z\|_{\dot B^s_{p,r}}+\|z\|_{\dot B^{s+1}_{p,r}}.
\end{align*}
Define
\begin{align*}
    \mathcal{E}_N(t):=\|(\mathcal K^{\frac{1}{2}}\sigma_\eps,w_\eps)(t)\|_{\dot B^{1/2}_{2,\infty}\cap H^N},\quad
    \mathcal{D}_N(t):=\|(\nabla\mathcal K^{\frac{1}{2}}\sigma_\eps,\nabla w_\eps)(t)\|_{ H^N}.
\end{align*}
The following theorem presents the low Mach number limit of the problem \eqref{NNSK-intro}.
\begin{thm}\label{NNSKThm-intro}
Assume that $N\ge 5$. There exist \(\delta_0 > 0\) and \(\eps_0 > 0\)  
such that if
\begin{align*}
0<\eps\leq \eps_0, \quad \|F\|_{\dot B_{2,\infty}^{-\frac{3}{2}} \cap \dot H^{N+1}} \leq \delta_0,
\end{align*}
and the initial data satisfy
\begin{align}\label{initialcondition-intro}
    \|(\mathcal K^{\frac{1}{2}}\sigma_{\eps,0},w_{\eps,0})\|_{\dot B^{1/2}_{2,\infty}\cap H^N}
    +\|\mathbb P u_{\eps,0}-u^* \|_{ \dot B^{1/2}_{2,\infty}} \leq \delta_0
\end{align}
where $u^*$ is the unique solution of the stationary Navier--Stokes equations \eqref{SNS-intro}, and
\begin{equation*}
    \sigma_{\eps,0}:=\eps^{-1}(\rho_{\eps,0}-\rho_\eps^*),\quad 
    w_{\eps,0}:=u_{\eps,0}-u_\eps^*,
\end{equation*}
then the problems \eqref{NNSK-intro} and \eqref{NNS-intro} admit unique solutions \((\rho_\eps, u_\eps)\) and \(u\), respectively,
and the following estimates hold
\begin{equation}\label{MainEng-intro}
    \sup_{t\geq 0}\mathcal{E}_N(t)^2+\int_0^\infty \mathcal{D}_N(t)^2 dt\leq C\delta_0^2,
\end{equation}
{here and in \eqref{NNSKest-intro}, the constant $C$ is independent of $\epsilon$}, and for $1/2<s<N-2$, 
\begin{align}\label{1.12.1}
    \|(\mathcal K^{\frac{1}{2}}\sigma_\eps,w_\eps)(t)\|_{\dot B^{s}_{2,1}}\leq C_s (1+t)^{-\frac{s-1/2}{2}}\delta_0, \quad t\geq 0.
\end{align}
Moreover,  if $2<p<6$, $2<r<\infty$ and 
$(r+4)/2r<s<3/p$,
then
\begin{equation}
\label{NNSKest-intro}\|\sigma_\eps\|_{L^r(0,\infty;\dot B^s_{p,1})} + \|\mathbb{Q}w_\eps\|_{L^r(0,\infty;\dot B^s_{p,1})}+
\|\mathbb{P}w_\eps-\tilde u\|_{L^r(0,\infty;\dot B^s_{p,1})}
\leq C \eps^{\beta(p,r)}\delta_0,
\end{equation}
where
\begin{align*}
\tilde u:=u-u^*,\quad   \beta(p,r):=\min\Big\{\frac{1}{r},\frac{1}{2}-\frac{1}{p}\Big\}>0.
\end{align*}
\end{thm}

\begin{rem}
The regularity condition \(N\geq5\) is technical and is not expected to be optimal. In the non-capillary problem, Deguchi \cite{De-2025} closed the argument in
\(\dot B^{1/2}_{2,\infty}\cap\dot H^4\). In the present NSK system, the Korteweg tensor induces three spatial derivatives of the density, while the symmetric variable \(\mathcal K^{1/2}\sigma_\eps\) carries one additional density derivative and the operator
\(\mathcal K^{1/2}\operatorname{div}\) is of order two at high frequencies. Moreover, the convergence argument requires the choice
$s\in\left(\frac52,N-2\right)$ 
to obtain the time-integrable \(L^\infty\) bound and to control the high-frequency acoustic-capillary source terms  (see Lemma \ref{Lem:TimeIntegrability}).  The assumption \(N\geq5\) guarantees that this interval is nonempty and provides the derivative margin needed in the low-high frequency interpolation.
\end{rem}

By the continuous embeddings $\dot B^{3(\frac{1}{q}-\frac{1}{p})}_{q,1}\hookrightarrow\dot B^0_{p,1}\hookrightarrow L^p$ with $1 < q \leq p < \infty$ (see \cite[Proposition 2.39]{BCD-Book-2011}), we obtain the following $L^p$-form convergence result.

\begin{cor}\label{NNSKLpCor-intro}
Under the assumptions of Theorem~\ref{NNSKThm-intro}, let
\(3<p<\infty\) and \(2<r<\infty\) satisfy $\frac2r<1-\frac3p$. 
Then, for any
\begin{align*}
0<\beta<\beta_*(p,r),
\quad \text{with}\quad
\beta_*(p,r)
:=
\min\left\{
\frac1r,\,
\frac13-\frac1p-\frac{2}{3r}
\right\},    
\end{align*}
there exists a constant \(C_\beta>0\), independent of
\(\eps\in(0,\eps_0]\), such that
\begin{align*}
 \eps^{-1}
\|\rho_\eps-\rho_\eps^*\|_{L^r(0,\infty;L^p)}
+\|\mathbb Q(u_\eps-u_\eps^*)\|_{L^r(0,\infty;L^p)}+\left\|
\mathbb P(u_\eps-u_\eps^*)-(u-u^*)
\right\|_{L^r(0,\infty;L^p)}
\leq
C_\beta\eps^\beta\delta_0.
\end{align*}
\end{cor}

\subsection{Proof strategies and main difficulties}

We explain the key ideas of the proof and point out the differences from the arguments in \cite{De-2025} on  the compressible Navier--Stokes equations  that arise due to the appearance of the Korteweg tensor.

\medskip
\noindent\textit{1. Uniform construction of the stationary solutions.}
We write
\begin{align*}
 \rho_\eps^*=\rho_\infty+\eps^2\sigma_\eps^*   
\end{align*}
and use the density equation together with the Helmholtz decomposition of the stationary velocity. The density equation is governed by the elliptic operator (see \eqref{NJKG3.2}):
\begin{align*}
\mathbf B:=p'(\rho_\infty)-\kappa\rho_\infty\Delta.    
\end{align*}
Its inverse is a uniformly bounded Fourier multiplier on the Besov spaces and gains two derivatives in the homogeneous Sobolev scale. After rewriting the stationary system as a fixed-point problem, the Besov product estimates and the high-order Sobolev estimates yield a contraction uniformly in $\eps$. Then the factor $\eps^2$ in the stationary density equation gives $\mathbb Q u_\eps^*=O(\eps^2)$, while the comparison between the projected momentum equations for $\mathbb P u_\eps^*$ and $u^*$ gives the second estimate in \eqref{machest-intro}.

\medskip
\noindent\textit{2. Uniform nonlinear energy estimates around a non-constant profile.}
The perturbation system contains singular acoustic--capillary terms
of order \(\eps^{-1}\), variable viscous coefficients, and linear
terms with coefficients depending on
\((\rho_\eps^*,u_\eps^*)\). The singular terms are skew-symmetric after pairing the density equation with $\K\sigma_\eps$ and the velocity equation with $\rho_\infty w_\eps$. A Kawashima-type cross term recovers the dissipation of $\nabla\K^{1/2}\sigma_\eps$. At high regularity, the remaining stationary and nonlinear terms are absorbed by the smallness of the external force and of the perturbation. At low frequency, a dyadic Besov estimate propagates the $\dot B^{1/2}_{2,\infty}$ norm. Combining both levels gives the uniform bound \eqref{MainEng-intro}. Interpolation between the low-frequency Besov control and the high-order dissipation yields the decay estimate in $\dot B^s_{2,1}$.

The weak Besov component is essential here. 
As pointed out in the theoretical analysis of the stationary Navier--Stokes equations in \cite{De-2024,De-2025}, the stationary velocity may possess a slow spatial tail and is not naturally expected to belong to the stronger critical space $\dot B^{1/2}_{2,1}$. The mixed framework $\dot B^{1/2}_{2,\infty}\cap H^N$ contains the stationary profile and at the same time supplies enough derivatives to control the nonlinear terms arising from the Korteweg tensor.

\medskip
\noindent\textit{3. Acoustic-capillary spectral analysis.}
Let $d_\eps:=\Lambda^{-1}\dive\mathbb Qw_\eps$. The compressible subsystem becomes symmetric in
\[
V_\eps:=\bigl(\K^{1/2}\sigma_\eps,\sqrt{\rho_\infty}\,d_\eps\bigr)^{\mathsf T}.
\]
This change of variables is decisive: the spectral projectors of the linearized acoustic-capillary operator become uniformly bounded zero-order Fourier multipliers.  {The two conjugate oscillatory eigenvalues have imaginary parts $\pm\Omega_\eps(|\xi|)$, where the frequency satisfies the asymptotic scaling}
\[
\Omega_\eps(|\xi|)
\sim
\begin{cases}
|\xi|/\eps,& |\xi|\ll1,\\
|\xi|^2/\eps,& |\xi|\gg1.
\end{cases}
\]
Consequently, the low-frequency evolution is wave-like, whereas the high-frequency evolution is Schr\"odinger-like. We derive dyadic kernel bounds in the two frequency regimes and combine them with the parabolic damping. The low-frequency estimate produces the gain $\eps^{1/r}$, while the high-frequency stationary-phase estimate produces $\eps^{1/2-1/p}$.

\medskip
\noindent\textit{4. Structural decomposition of the Duhamel source.}
The perturbation around a nontrivial stationary flow generates source terms with stationary coefficients. Such terms are not necessarily integrable in time and therefore cannot be fully handled by a standard inhomogeneous Strichartz estimate with an $L^1_t$ source. We separate the acoustic source into a time-integrable part and a stationary-coefficient part. The former is estimated by the decay and energy bounds. The low-frequency stationary part carries an explicit factor of $\eps$, while the high-frequency stationary part is treated by a damped $L^r_t$ estimate that exploits the heat factor without requiring time integrability of the coefficient. This decomposition closes the acoustic estimate with the rate in \eqref{NNSKest-intro}.

\medskip
\noindent\textit{5. Convergence of the incompressible component.}
After applying the Helmholtz projection, the pressure and capillary gradient terms disappear. The difference
\[
z_\eps:=\mathbb Pw_\eps-(u-u^*)
\]
satisfies a forced heat equation. Its source is decomposed into the terms containing either the already controlled acoustic component, the stationary error $u_\eps^*-u^*=O(\eps^2)$, or a small coefficient multiplying $z_\eps$. The low-high compatible Besov product estimates together with the maximal regularity of the heat equation then yield the same convergence rate for $z_\eps$ as for the acoustic component.

\subsection{Outline of this paper}
The rest of the paper is organized as follows. Section \ref{Sec:Pre} introduces the notation, the homogeneous Besov spaces, the product and composition estimates, the elliptic Fourier-multiplier estimate used for the Korteweg operator, and the result on the stationary incompressible Navier--Stokes equations required later. Section \ref{Sec:SNSK} constructs the solutions of the stationary NSK equations uniformly in the Mach number and proves their second-order convergence toward the stationary incompressible flow. Section \ref{Sec:NSNSK} treats the non-stationary problem. It derives the perturbation system, establishes uniform a priori estimates, proves global existence and decay, analyzes the acoustic-capillary semigroup, and completes the convergence estimates between the compressible and incompressible components.

\section{Preliminary}\label{Sec:Pre}

\subsection{Notations}
Throughout this paper, $c$ and \(C\) represent generic positive constants whose value may change from one line to line, while the constants $c_\alpha$ and $C_\alpha$ indicate dependence on the parameter $\alpha$.
The notation \(A\lesssim_{\alpha} B\) indicates that there is a constant \(C\) which depends on \(\alpha\) such that \(A \leq C B\). We write $A\sim_{\alpha} B$ if $A\lesssim_{\alpha} B$ and $B\lesssim_{\alpha} A$. Unless otherwise stated, all implicit constants are independent of the time variable \(t\) and the Mach number \(\eps\).
Let \(X\) be a Banach space. For \(f,g\in X\), we define
$\|(f,g)\|_{X}:=\|f\|_{X}+\|g\|_{X}$.
Given \(T>0\) and \(1\leq p\leq\infty\), \(L^{p}(0,T;X)\) denotes the usual Bochner space of \(X\)-valued measurable functions on \((0,T)\). For convenience, its norm is abbreviated by $\|f\|_{L_T^{p}X}
:=
\|f\|_{L^{p}(0,T;X)}$.
The Fourier transform and inverse Fourier transform of \(f\) are denoted by $\widehat{f}:=\mathcal{F}f$ and $\check{f}:=\mathcal{F}^{-1}f$, respectively.
 
\subsection{Littlewood--Paley decomposition}
We briefly introduce the homogeneous Littlewood--Paley decomposition used throughout the paper, see \cite[Chapters 2--3]{BCD-Book-2011} for a systematic presentation. Choose a radial function $
\chi\in C_c^\infty(\mathbb R^d)$,
which is non-increasing with respect to \(|\xi|\), satisfies
\begin{align*}
0\leq \chi\leq 1,\qquad
\chi(\xi)=1\quad\text{for}\quad|\xi|\leq \frac34,    
\end{align*}
and vanishes whenever \(|\xi|\geq \frac43\). We then set
$\varphi(\xi):=\chi\Big(\frac{\xi}{2}\Big)-\chi(\xi)$.
By construction,
\begin{align*}
 \operatorname{supp}\varphi
\subset
\bigg\{\xi\in\mathbb R^d:
\frac34\leq |\xi|\leq \frac83\bigg\},   
\end{align*}
and the family \(\{\varphi(2^{-k}\cdot)\}_{k\in\mathbb Z}\) forms a homogeneous partition of unity away from the origin:
\begin{align*}
\sum_{k\in\mathbb Z}\varphi(2^{-k}\xi)=1,
\qquad \xi\neq 0.    
\end{align*}
For each \(k\in\mathbb Z\), the \(k\)-th homogeneous dyadic localization of a tempered distribution \(f\) is defined by
$\dot\Delta_k f
:=
\mathcal F^{-1}
\left(
\varphi(2^{-k}\xi)\widehat f(\xi)
\right)$.
By setting \(h=\mathcal F^{-1}\varphi\), one has $
\dot\Delta_k f
=
2^{kd}h(2^k\cdot)\ast f$.
Let \(\mathcal P\) be the space of polynomials on \(\mathbb R^d\), and denote by $\mathcal S'_h:=\mathcal S'(\mathbb R^d)/\mathcal P$,
the space of tempered distributions modulo polynomials. For every
\(f\in\mathcal S'_h\), the homogeneous decomposition takes the form $f=\sum\limits_{k\in\mathbb Z}\dot\Delta_k f$
in \(\mathcal S'_h\). Moreover, the Fourier supports of distinct dyadic pieces are almost disjoint, and hence $\dot\Delta_k\dot\Delta_\ell f=0$ whenever
$|k-\ell|\geq 2$.
For a distribution \(z\), we decompose it into low- and high-frequency parts:
\begin{equation*}
z^\ell:=\sum_{j\leq0}\dot\Delta_jz,
\qquad z^h:=z-z^\ell.
\end{equation*}

Based on the frequency localization, we now introduce the following homogeneous Besov spaces:
\begin{defn}
For $ s \in \mathbb{R} $ and $ 1 \leq p, r \leq \infty $, the homogeneous Besov spaces $ \dot{B}_{p,r}^s$ are defined as 
\begin{align*}
\dot B_{p,r}^s:=\big\{ f\in \mathcal{S}^\prime_h \,\big{|}\, \|f\|_{\dot{B}_{p,r}^s}:=\big\|\{2^{ks}\|\dot\Delta_k f\|_{L^p}\}_{k\in\mathbb Z}\big\|_{l^r}<\infty\big\}. 
\end{align*}
\end{defn}

\subsection{Analytic tools}
Next, we state several fundamental properties of Besov spaces, along with product estimates.
\begin{prop}\label{prop2.1}{\rm(\!\!\!\cite{BCD-Book-2011})}
Let $0<r<R$, $1\leq p\leq q\leq \infty$ and $m\in \mathbb{N}$. Define the ball $\mathcal{B}=\{\xi\in\mathbb{R}^{3}~| ~|\xi|\leq  R\}$ and the annulus $\mathcal{C}=\{\xi\in\mathbb{R}^{3}~|~  r\leq |\xi|\leq R\}$ . For any $ g\in L^p$ and $\lambda_1>0$, it holds that
\begin{equation*} \left\{
\begin{aligned}
{\rm{supp}}\, \mathcal{F}(g) \subset& \lambda_1 \mathcal{B} \Rightarrow \|D^{m}g\|_{L^q}\lesssim\lambda_1^{m+d(\frac{1}{p}-\frac{1}q{})}\|g\|_{L^p}, \nonumber\\
{\rm{supp}}\, \mathcal{F}(g) \subset& \lambda_1 \mathcal{C} \Rightarrow \lambda_1^{m}\|g\|_{L^{p}}\lesssim\|D^{m}g\|_{L^{p}}\lesssim \lambda_1^{m}\|g\|_{L^{p}}.
\end{aligned}\right.
\end{equation*}
 \end{prop}

By virtue of Proposition \ref{prop2.1}, the following properties of Besov spaces can be   established.

\begin{prop}{\rm(\!\!\cite[Chapter 2]{BCD-Book-2011}, \cite[Section 2]{De-2025} and \cite{Ka-Ko-Sh-2019})}\label{BesovP}
Let \( s \in \mathbb{R} \) and \( 1 \leq p,r \leq \infty \).
\begin{enumerate}
    \item[(i)]  For any \( k \geq 0 \),
\begin{align*}
  \|\nabla^k g\|_{\dot B_{p,r}^{s}} \sim \|g\|_{\dot B_{p,r}^{s+k}}.    
\end{align*}
    \item[(ii)]  Let \( p' \) be the conjugate exponent of \( p \) and let \( r' \) be the conjugate exponent of \( r \). Then, we have the following duality estimates:
\begin{align*}
 \langle g,h \rangle \lesssim \|g\|_{\dot B_{p,r}^{s}} \|h\|_{\dot  B_{p',r'}^{-s}} \quad \text{and} \quad \|g\|_{\dot B_{p,r}^{s}} \lesssim \sup_{\psi} \langle g,\psi \rangle,    
\end{align*}   
    where the supremum is taken over the Schwartz functions \( \psi \) with \( \|\psi\|_{\dot B_{p',r'}^{-s}} \leq 1 \) and \( 0 \notin \text{supp} \mathcal{F}\psi \).

    \item[(iii)] Let \( s_1 < s_2 \) satisfy \( s = (1-\theta)s_1 + \theta s_2 \) for some \( 0 < \theta < 1 \). Then, the interpolation inequality
\begin{align*}
    \|g\|_{\dot B_{p,r}^{s}} \lesssim_{\theta,s_1,s_2} \|g\|_{\dot B_{p,\infty}^{s_1}}^{1-\theta} \|g\|_{\dot B_{p,\infty}^{s_2}}^{\theta} 
\end{align*}
    holds.
    
    \item[(iv)] Let \( s_1 \geq s_2 \), \( 1 \leq p_1 \leq p_2 \leq \infty \), and \( 1 \leq q \leq \infty \). If \( s_1 - 3/p_1 = s_2 - 3/p_2 \), then we have the continuous embedding
$    \dot{B}_{p_1,q}^{s_1} \hookrightarrow \dot{B}_{p_2,q}^{s_2}$.
    \item[(v)] Let \( \Phi \in C^\infty(\mathbb{R}^3) \) and \( g,h \in \dot B_{2,r}^{s} \cap \dot B_{2,1}^{3/2} \) with \( -3/2 \leq s < 3/2 \) or \( s = 3/2 \), \( r = 1 \). Then, we have
\begin{align*}
  \|\Phi(g) - \Phi(h)\|_{\dot B_{2,r}^{s}} \lesssim_\Phi \big (1 + \|(g,h)\|_{\dot B_{2,1}^{\frac{3}{2}}}\big )\|g - h\|_{\dot B_{2,r}^{s}}.    
\end{align*}
\item[(vi)] Let \( s_1,s_2 \in \mathbb{R} \) satisfy \( s_1,s_2 < 3/2 \) and \( s_1 + s_2 > 0 \). Let \( 1 \leq r_1,r_2 \leq \infty \) satisfy \( 1/r_1 + 1/r_2 = 1/r \). Then, we have
\begin{align*}
 \|gh\|_{\dot B_{2,r}^{s_1+s_2-\frac{3}{2}}} \lesssim_{s_1,s_2} \|g\|_{\dot B_{2,r_1}^{s_1}} \|h\|_{\dot B_{2,r_2}^{s_2}}.    
\end{align*}
    In the cases \( s_1 \leq 3/2 \), \( s_2 < 3/2 \) with \( s_1 + s_2 \geq 0 \), we have
\begin{align*}
\|gh\|_{\dot B_{2,\infty}^{s_1+s_2-\frac{3}{2}}} \lesssim \|g\|_{\dot B_{2,1}^{s_1}} \|h\|_{\dot B_{2,\infty}^{s_2}}.    
\end{align*}
    \item[(vii)] Let \(1 \leq p,q \leq \infty\), \(s > 0\), \(\alpha > 0\), and \(\beta > 0\). Moreover, assume that \(1 \leq p_1,p_2,\tilde{p}_1,\tilde{p}_2 \leq \infty\) satisfy \(1/p = 1/p_1 + 1/p_2 = 1/\tilde{p}_1 + 1/\tilde{p}_2\). If \(g \in \dot{B}_{p_1,q}^{s+\alpha} \cap \dot{B}_{\tilde{p}_1,\infty}^{-\beta}\) and \(h \in \dot{B}_{p_2,\infty}^{-\alpha} \cap \dot{B}_{\tilde{p}_2,q}^{s+\beta}\), then we have \(gh \in \dot{B}_{p,q}^s\) with the estimate
\begin{align*}
    \|gh\|_{\dot{B}_{p,q}^s} \leq C\big(\|g\|_{\dot{B}_{p_1,q}^{s+\alpha}} \|h\|_{\dot{B}_{p_2,\infty}^{-\alpha}} + \|g\|_{\dot{B}_{\tilde{p}_1,\infty}^{-\beta}} \|h\|_{\dot{B}_{\tilde{p}_2,q}^{s+\beta}}\big),      
\end{align*}
    where \(C = C(p,p_1,p_2,\tilde{p}_1,\tilde{p}_2,q,s,\alpha,\beta)\).
\end{enumerate}
\end{prop}

To handle the elliptic equations arising in the stationary analysis, we define the operator $\mathbf{B}$ as follows:
\begin{align} 
\mathbf{B}:=\zeta_1-\zeta_2\Delta, \label{2.1.1}   
\end{align}
where $\zeta_1>0$ and $\zeta_2>0$ are  two constants.
The following standard Fourier-multiplier estimate will be used repeatedly.

\begin{prop}\label{Prop2.3}
Let \(1 < p < \infty\), \(1 \leq r \leq \infty\), \(s \in \mathbb{R}\), and \(k > s+2\). Assume \(f \in \dot B_{p,r}^s \cap \dot H^k\), then the problem \(\mathbf{B} u = f\) admits a unique solution \(u \in \dot B_{p,r}^s \cap \dot H^{k+2}\) satisfying
\begin{equation}\label{Best}
\|u\|_{\dot B_{p,r}^s \cap \dot H^{k+2}} \leq C \|f\|_{\dot B_{p,r}^s \cap \dot H^k}.
\end{equation}
\end{prop}
\begin{proof}
Taking the Fourier transform of \(\mathbf{B}u=f\), we obtain
\begin{align*}
\widehat{u}(\xi)
=
m(\xi)\widehat{f}(\xi),  \quad\text{where}\quad  m(\xi):=\frac{1}{\zeta_1+\zeta_2|\xi|^2}.
\end{align*}
Since \(\zeta_1,\zeta_2>0\), the symbol \(m\) is smooth and satisfies the
Mikhlin multiplier condition
\begin{align*}
\sup_{\xi\neq0}
|\xi|^{|\alpha|}
\big|\partial_\xi^\alpha m(\xi)\big|
<\infty,    
\end{align*}
for any multi-index \(\alpha\). Hence, by the Mikhlin multiplier
theorem and the dyadic characterization of homogeneous Besov spaces, we derive
\begin{align*}
\|u\|_{\dot B^s_{p,r}}
\leq C\|f\|_{\dot B^s_{p,r}}.    
\end{align*}

Moreover, by direct calculation, it can be deduced that
\begin{align*}
 \|u\|_{\dot H^{k+2}}=
\left\|
\frac{|\xi|^{k+2}}{\zeta_1+\zeta_2|\xi|^2}
\widehat{f}(\xi)
\right\|_{L^2_\xi} 
= 
\left\|
\frac{|\xi|^2}{\zeta_1+\zeta_2|\xi|^2}
|\xi|^k\widehat{f}(\xi)
\right\|_{L^2_\xi}
\leq
\frac{1}{\zeta_2}\|f\|_{\dot H^k}.   
\end{align*}
Therefore, we arrive at
\begin{align*}
\|u\|_{\dot B^s_{p,r}\cap\dot H^{k+2}}
\leq
C\|f\|_{\dot B^s_{p,r}\cap\dot H^k}.    
\end{align*}

Finally, the strict positivity of \(\zeta_1+\zeta_2|\xi|^2\) implies uniqueness. The proof of Proposition \ref{Prop2.3} is complete.
\end{proof}

Next, we  introduce the commutator estimates that will be employed in the high-order energy argument.

\begin{lem}[{\!\!\cite[Appendix]{commutator1},\cite{commutator2}}]\label{LA.1}
Let $g$ and $h$ be two Schwarz functions. For $k\geq 1$, one has
\begin{align*}
\|\nabla^{k}(gh) \|_{L^r} \lesssim &\, \|g\|_{L^{r_1} }\|\nabla^{k}h\|_{L^{r_2} }+\|h\|_{L^{r_3} }\|\nabla^{k}g\|_{L^{r_4} },\\
\|\nabla^{k}(gh)-g\nabla^k h \|_{L^r} \lesssim  &\, \|\nabla g\|_{L^{r_1}}\|\nabla^{k-1}h\|_{L^{r_2}}+\|h\|_{L^{r_3}}\|\nabla^{k}g\|_{L^{r_4}},    
\end{align*}
where $1<r,r_2,r_4<\infty$ and $r_i(1\leq i\leq 4)$ satisfy 
\begin{align*}
\frac{1}{r_1}+\frac{1}{r_2}=\frac{1}{r_3}+\frac{1}{r_4}=\frac{1}{r}.   
\end{align*}
\end{lem}

Applying Lemma \ref{LA.1}, we establish the following estimate, which is also necessary to control the nonlinear source terms.

\begin{lem}\label{Lem:L6ab}
Let \(N \geq 4\) and $\delta>0$. Suppose that \(\tilde{a}\) is a vector field and \(\tilde{b}\) is a tensor field satisfying  
\begin{align*}
 \|\tilde{a}\|_{\dot{B}_{2,\infty}^{1/2}\cap H^{N+2}}
+\|\tilde{b}\|_{\dot{B}_{2,\infty}^{-1/2}\cap H^{N+1}}
\leq \delta.   
\end{align*}
Then, for any smooth vector field \(\tilde c\) with \(\nabla \tilde c\in H^N\), it holds that
\begin{equation}\label{L6ab1}
\sum_{|\alpha|\leq N}
\left|
\left\langle
\partial^\alpha \tilde c,\partial^\alpha(\tilde{a}\cdot\nabla \tilde c)
\right\rangle
\right|
\leq C\delta\|\nabla \tilde c\|_{H^N}^2,
\end{equation}
and
\begin{equation}\label{L6ab2}
\sum_{|\alpha|\leq N}
\left|
\left\langle
\partial^\alpha \tilde c,\partial^\alpha(\tilde{b}\tilde c)
\right\rangle
\right|
\leq C\delta\|\nabla \tilde c\|_{H^N}^2.
\end{equation}
\end{lem}

\begin{proof}
For \(|\alpha|\leq N\), we write
\begin{align*}
\partial^\alpha(\tilde{a}\cdot\nabla \tilde c)
=\tilde{a}\cdot\nabla\partial^\alpha \tilde c
+[\partial^\alpha,\tilde{a}\cdot\nabla]\tilde c.    
\end{align*}
Integration by parts and Proposition \ref{BesovP} give
\begin{align*}
\left|
\left\langle
\partial^\alpha \tilde c,\tilde{a}\cdot\nabla\partial^\alpha \tilde c
\right\rangle
\right|=
\frac12
\left|
\left\langle
\operatorname{div}\tilde{a},|\partial^\alpha \tilde c|^2
\right\rangle
\right| 
 \lesssim
\|\tilde{a}\|_{\dot B^{1/2}_{2,\infty}}
\|\nabla\partial^\alpha \tilde c\|_{L^2}^2.
\end{align*}
For \(1\leq|\alpha|\leq N\), Lemma~\ref{LA.1} and Sobolev embedding
yield
\begin{align*}
\|[\partial^\alpha,\tilde{a}\cdot\nabla]\tilde c\|_{L^2}
\lesssim
\|\nabla \tilde{a}\|_{L^\infty}
\|\nabla^{|\alpha|}\tilde c\|_{L^2}
+\|\nabla^{|\alpha|}\tilde{a}\|_{L^2}
\|\nabla \tilde c\|_{L^\infty}\lesssim
\delta\|\nabla \tilde c\|_{H^N}.
\end{align*}
Since the commutator vanishes for \(|\alpha|=0\), summing over
\(|\alpha|\leq N\) proves \eqref{L6ab1}.

Similarly, we have
\begin{align*}
\partial^\alpha(\tilde{b}\tilde c)=\tilde{b}\partial^\alpha \tilde c+[\partial^\alpha,\tilde{b}]\tilde c.    
\end{align*}
By Proposition~\ref{BesovP}, we have
\begin{align*}
|
\langle
\partial^\alpha \tilde c,\tilde{b}\partial^\alpha \tilde c
\rangle
|
\lesssim
\|\tilde{b}\|_{\dot B^{-1/2}_{2,\infty}}
\|\nabla\partial^\alpha \tilde c\|_{L^2}^2.    
\end{align*}
For \(1\leq|\alpha|\leq N\), Lemma~\ref{LA.1} and Sobolev inequality
give
\begin{align*}
\|[\partial^\alpha,\tilde{b}]\tilde c\|_{L^2}
\lesssim
\|\nabla \tilde{b}\|_{L^3}
\|\nabla^{|\alpha|-1}\tilde c\|_{L^6}
+\|\nabla^{|\alpha|}\tilde{b}\|_{L^2}\|\tilde c\|_{L^\infty}
\lesssim
\delta\|\nabla \tilde c\|_{H^N}.
\end{align*}
Summing over \(|\alpha|\leq N\) proves \eqref{L6ab2}.
\end{proof}

The nonlinear terms involve the products of vector-valued functions as well as the compositions with the elliptic viscous operator
\begin{align*}
 \A:=\mu \Delta + \nu \nabla \text{div}.
\end{align*}
To handle these terms, we establish low‑high compatible product estimates adapted to the intersection space
\begin{equation}\label{Eq:HybridX}
\mathcal X:=\dot B^{1/2}_{2,\infty}\cap\dot B^{3/2}_{2,1},
\end{equation}
which will be used later in Lemma \ref{Lem:ErrorSources}.
\begin{lem}[Low-high compatible product estimates]\label{Lem:HybridProduct}
Let \(2<p<6\) and \(1/2<s<3/p\).  For vector fields $\tilde a$, $\tilde b$, and a scalar field $\tilde c$, the following bounds hold:
\begin{align}
\|\tilde a\cdot\nabla \tilde b\|_{\dot B^{s-2}_{p,1}}
&\leq C\bigl(
\|\tilde a\|_{\dot B^s_{p,1}}\|\tilde b\|_{\mathcal X}
+\|\tilde a\|_{\mathcal X}\|\tilde b\|_{\dot B^s_{p,1}}\bigr),
\label{Eq:HybridTransport}\\
\|\tilde c\A \tilde b\|_{\dot B^{s-2}_{p,1}}
&\leq C\bigl(
\|\tilde c\|_{\mathcal X}\|\tilde b\|_{\dot B^s_{p,1}}
+\|\tilde c\|_{\dot B^s_{p,1}}\|\tilde b\|_{\mathcal X}\bigr).
\label{Eq:HybridViscosity}
\end{align}
\end{lem}

\begin{proof}
    By Bony’s decomposition, we write \(fg=T_fg+T_gf+R(f,g)\).  
    We first record some low-high frequency estimates. By Bernstein's inequality, for \(m=0,1,2\),
\begin{equation}\label{Eq:HybridLowHighLinfty}
\|\dot S_{j-1}\nabla^m f\|_{L^\infty}
\lesssim
\begin{cases}
2^{(m+1)j}\|f\|_{\dot B^{1/2}_{2,\infty}},
& j\leq0,\\[1mm]
2^{mj}\|f\|_{\dot B^{3/2}_{2,1}},
& j\geq1.
\end{cases}
\end{equation}
These bounds give the two paraproducts in \eqref{Eq:HybridTransport} after multiplication by $2^{j(s-2)}$.

By the Fourier-support properties of the remainder, there exists a fixed integer $N_0$ such that
    \begin{align*}
        \dot\Delta_jR(\tilde a,\nabla \tilde b)=\sum_{k\geq j-N_0}\dot\Delta_j\bigl(\dot\Delta_k \tilde a\cdot\nabla\widetilde{\dot\Delta}_k\tilde b\bigr),
    \end{align*}
 where $\widetilde{\dot\Delta}_k=\dot\Delta_{k-1}+\dot\Delta_k+\dot\Delta_{k+1}$.
For $k\geq j-N_0$, Bernstein's and H\"older's inequalities give
    \begin{align*}
        \|\dot\Delta_j(\dot\Delta_k \tilde a\,\nabla\widetilde{\dot\Delta}_k \tilde b)\|_{L^p}
        \leq C2^{3j/2}2^k\|\dot\Delta_k \tilde a\|_{L^p}\|\widetilde{\dot\Delta}_k \tilde b\|_{L^2}.
    \end{align*}
    Since $s>1/2$, we have
    \begin{align}\notag
        \|R(\tilde a,\nabla \tilde b)\|_{\dot B^{s-2}_{p,1}}
        &\leq C\sum_{k\in\mathbb Z}\sum_{j\leq k+N_0} 2^{j(s-2)}
        2^{3j/2}2^k\|\dot\Delta_k \tilde a\|_{L^p}\|\widetilde{\dot\Delta}_k \tilde b\|_{L^2}\\
        &\leq C \sum_{k\in\mathbb Z}2^{k(s+1/2)}\|\dot\Delta_k \tilde a\|_{L^p}\|\widetilde{\dot\Delta}_k \tilde b\|_{L^2}.
        \label{Eq:HybridRemainderTransport}
    \end{align}
    If $k\leq 0$, we have 
    \begin{align*}
    2^{k(s+1/2)}\|\dot\Delta_k \tilde a\|_{L^p}\|\widetilde{\dot\Delta}_k \tilde b\|_{L^2}\leq C2^{ks}\|\dot\Delta_k\tilde a\|_p\|\tilde b\|_{\dot B^{1/2}_{2,\infty}},
    \end{align*}
    which implies that the summation in $k\leq 0$ is controlled by $\|\tilde a\|_{\dot B^s_{p,1}}\|\tilde b\|_{\dot B^{1/2}_{2,\infty}}$.
    For $k\geq 1$, we have $2^{-k}\leq 1$, which implies that
    \begin{align*}
        \sum_{k\geq 0} 2^{k(s+1/2)}\|\dot\Delta_k \tilde a\|_{L^p}\|\widetilde{\dot\Delta}_k \tilde b\|_{L^2}
        &\leq C \sum_{k\geq 0} (2^{ks}\|\dot\Delta_k \tilde a\|_{L^p}) (2^{3k/2}\|\widetilde{\dot\Delta}_k \tilde b\|_{L^2})\\
        &\leq C \big(\sum_{k\geq 0} 2^{ks}\|\dot\Delta_k \tilde a\|_{L^p}\big) \big(\sum_{k\geq 0}2^{3k/2}\|\widetilde{\dot\Delta}_k \tilde b\|_{L^2}\big)\\
        &\leq C \|\tilde a\|_{\dot B^s_{p,1}}\|\tilde b\|_{\dot B^{3/2}_{2,1}}.
    \end{align*}
    Summing separately over \(k\leq0\) and \(k\geq1\) in
\eqref{Eq:HybridRemainderTransport}, we obtain
\begin{align*}
    \|R(\tilde a,\nabla \tilde b)\|_{\dot B^{s-2}_{p,1}}\leq C
\|\tilde a\|_{\dot B^s_{p,1}}\|\tilde b\|_{\mathcal X}.
\end{align*}
This proves \eqref{Eq:HybridTransport}.

We next prove \eqref{Eq:HybridViscosity}. Since $\A$ is a
second-order constant-coefficient operator, Bony's decomposition gives
\begin{align*}
    \tilde c\A \tilde b=T_{\tilde c}(\A \tilde b)+T_{\A \tilde b}\tilde c+R(\tilde c,\A \tilde b).
\end{align*}
Using \eqref{Eq:HybridLowHighLinfty} with $m=0$ and $m=2$,
we obtain
\begin{align*}
    \|T_{\tilde c}(\A \tilde b)\|_{\dot B^{s-2}_{p,1}}\leq C\|\tilde c\|_{\mathcal X}\|\tilde b\|_{\dot B^s_{p,1}},
\end{align*}
and
\begin{align*}
    \|T_{\A \tilde b}\tilde c\|_{\dot B^{s-2}_{p,1}}\leq C\|\tilde b\|_{\mathcal X}\|\tilde c\|_{\dot B^s_{p,1}}.
\end{align*}

For the remainder, the same argument as above yields
\begin{align*}
\|\dot\Delta_j\bigl(\dot\Delta_k \tilde c\,\A\widetilde{\dot\Delta}_k \tilde b\bigr)\|_{L^p}
\leq C
2^{3j/2}2^{2k}\|\dot\Delta_k \tilde c\|_{L^p}\|\widetilde{\dot\Delta}_k \tilde b\|_{L^2},
\end{align*}
and hence
\begin{align*}
\|R(\tilde c,\A \tilde b)\|_{\dot B^{s-2}_{p,1}}
&\leq C
\sum_{k\in\mathbb Z}
2^{k(s+\frac32)}
\|\dot\Delta_k\tilde c\|_{L^p}
\|\widetilde{\dot\Delta}_k\tilde b\|_{L^2}.
\end{align*}
If $k\leq0$, then $2^k\leq 1$ and thus
\begin{align*}
2^{k(s+\frac32)}\|\dot\Delta_k\tilde c\|_{L^p}\|\widetilde{\dot\Delta}_k\tilde b\|_{L^2}
=2^k\big(2^{ks}\|\dot\Delta_k\tilde c\|_{L^p}\big)\big(2^{k/2}\|\widetilde{\dot\Delta}_k \tilde b\|_{L^2}\big)
\leq C2^{ks}\|\dot\Delta_k\tilde c\|_{L^p}\|\tilde b\|_{\dot B^{1/2}_{2,\infty}}.
\end{align*}
If $k\geq1$, then
\begin{align*}
    2^{k(s+\frac32)}\|\dot\Delta_k\tilde c\|_{L^p}\|\widetilde{\dot\Delta}_k\tilde b\|_{L^2}
=\big(2^{ks}\|\dot\Delta_k\tilde c\|_{L^p}\big)\big(2^{3k/2}\|\widetilde{\dot\Delta}_k\tilde b\|_{L^2}\big).
\end{align*}
Therefore,
\begin{align*}
    \|R(\tilde c,\A \tilde b)\|_{\dot B^{s-2}_{p,1}}\lesssim\|\tilde c\|_{\dot B^s_{p,1}}\|\tilde b\|_{\mathcal X}.
\end{align*}
Combining the paraproduct and remainder estimates proves
\eqref{Eq:HybridViscosity}.
\end{proof}

\subsection{The stationary incompressible Navier--Stokes problem}

We conclude this section by identifying the stationary incompressible flow
that arises as the limiting profile in the stationary low Mach number
analysis. More precisely, we consider
\begin{equation}\label{SNS}
\left\{
\begin{aligned}
&\rho_\infty \operatorname{div}(u^*\otimes u^*)
 =\mu\Delta u^*-\nabla\Pi^*+\rho_\infty F(x),\\
&\operatorname{div}u^*=0.
\end{aligned}
\right.
\end{equation}
The following result provides the existence of the unique small stationary solution in
Besov--Sobolev spaces.

\begin{thm}\label{SNSthm}
Assume that \(k\geq3\). There exists a constant \(c_1>0\) such that,
if
\[
F\in\dot B^{-3/2}_{2,\infty}\cap\dot H^k
\qquad\text{and}\qquad
\|F\|_{\dot B^{-3/2}_{2,\infty}\cap\dot H^k}\leq c_1,
\]
then the problem \eqref{SNS} admits a unique solution
\[
u^*\in\dot B^{1/2}_{2,\infty}\cap\dot H^{k+2}
\]
satisfying
\begin{equation}\label{SNSest}
\|u^*\|_{\dot B^{1/2}_{2,\infty}\cap\dot H^{k+2}}
\leq
C\|F\|_{\dot B^{-3/2}_{2,\infty}\cap\dot H^k}.
\end{equation}
\end{thm}

\begin{proof}
By \cite[Theorem~3.6]{De-2025}, if
\(\|F\|_{\dot B^{-3/2}_{2,\infty}}\) is sufficiently small, then
the problem \eqref{SNS} admits a unique solution
\(u^*\in\dot B^{1/2}_{2,\infty}\) such that
\begin{equation}\label{SNSBesovEstimate}
\|u^*\|_{\dot B^{1/2}_{2,\infty}}
\leq
C\|F\|_{\dot B^{-3/2}_{2,\infty}}.
\end{equation}
It remains to establish the higher-order Sobolev estimate. A standard
Friedrichs regularization argument allows us to perform the following
calculation for smooth approximate solutions and then pass to the limit.

Since \(\mathbb Pu^*=u^*\), applying the Helmholtz projection
\(\mathbb P\) to the momentum equation in \eqref{SNS} gives
\[
\mu\Delta u^*
=
\rho_\infty\mathbb P\operatorname{div}(u^*\otimes u^*)
-\rho_\infty\mathbb PF.
\]
Consequently,
\begin{equation}\label{SNSHighEstimate}
\|u^*\|_{\dot H^{k+2}}
\leq
C\|u^*\otimes u^*\|_{\dot H^{k+1}}
+C\|F\|_{\dot H^k}.
\end{equation}
By the tame product estimate in Proposition~\ref{BesovP},
\[
\|u^*\otimes u^*\|_{\dot H^{k+1}}
\leq
C\|u^*\|_{\dot B^{1/2}_{2,\infty}}
  \|u^*\|_{\dot H^{k+2}}.
\]
Substituting this estimate into \eqref{SNSHighEstimate} and using
\eqref{SNSBesovEstimate}, we obtain
\begin{align}\label{2.12.1}
\|u^*\|_{\dot H^{k+2}}
\leq
C\|F\|_{\dot B^{-3/2}_{2,\infty}}
  \|u^*\|_{\dot H^{k+2}}
+C\|F\|_{\dot H^k}.
\end{align}
Choosing \(c_1>0\) sufficiently small, the first term on the right-hand
side of \eqref{2.12.1} can be absorbed by the term on the left-hand side. Hence,
\[
\|u^*\|_{\dot H^{k+2}}
\leq C\|F\|_{\dot H^k}.
\]
Combining this inequality with \eqref{SNSBesovEstimate} yields
\eqref{SNSest}. The uniqueness in
\(\dot B^{1/2}_{2,\infty}\cap\dot H^{k+2}\) follows directly from the
uniqueness in \(\dot B^{1/2}_{2,\infty}\).
\end{proof}

Theorem~\ref{SNSthm} determines the stationary incompressible profile
\(u^*\) appearing in Theorem~\ref{LMSNSKThm-intro}. In
Section~\ref{Sec:SNSK}, we construct the corresponding stationary NSK
solution \((\rho_\eps^*,u_\eps^*)\) uniformly with respect to the Mach
number and establish its second-order convergence toward
\((\rho_\infty,u^*)\).

\section{ The stationary NSK equations}
\label{Sec:SNSK}

This section is devoted to the proof of Theorem~\ref{LMSNSKThm-intro}. We first reformulate the stationary NSK equations as an elliptic
fixed-point problem and establish the nonlinear estimates required for
the contraction argument.
We then construct the stationary solution
uniformly with respect to $\epsilon$ and prove its second-order convergence toward the stationary incompressible flow.

\subsection{Reformulation and nonlinear estimates}
Let
\(\rho_\epsilon^*=\rho_\infty+\epsilon^2\sigma_\epsilon^*\). Substituting
this expression into \eqref{LMSNKS-intro}, we obtain
\begin{equation}\label{PLMSNKS}
\left\{
\begin{aligned}
&\rho_\infty\operatorname{div}u_\epsilon^*
 +\epsilon^2\operatorname{div}(\sigma_\epsilon^* u_\epsilon^*)=0,\\
&p'(\rho_\infty)\nabla\sigma_\epsilon^*
 -\A u_\epsilon^*
 -\kappa\rho_\infty\nabla\Delta\sigma_\epsilon^*
 =g_\epsilon(\sigma_\epsilon^*,u_\epsilon^*),
\end{aligned}
\right.
\end{equation}
where
\begin{align*}
g_\epsilon(\sigma_\epsilon^*,u_\epsilon^*)
=&\,-\operatorname{div}\big(
(\rho_\infty+\epsilon^2\sigma_\epsilon^*)
u_\epsilon^*\otimes u_\epsilon^*
\big) -\big(
p'(\rho_\infty+\epsilon^2\sigma_\epsilon^*)
-p'(\rho_\infty)
\big)\nabla\sigma_\epsilon^*\\
&\,+\kappa\epsilon^2\sigma_\epsilon^*
\nabla\Delta\sigma_\epsilon^*
+(\rho_\infty+\epsilon^2\sigma_\epsilon^*)F.
\end{align*}

Applying the operator \(\operatorname{div}\) to the momentum equation in
\eqref{PLMSNKS}, we have
\[
p'(\rho_\infty)\Delta\sigma_\epsilon^*
-(\mu+\nu)\Delta\operatorname{div}u_\epsilon^*
-\kappa\rho_\infty\Delta^2\sigma_\epsilon^*
=
\operatorname{div}g_\epsilon(\sigma_\epsilon^*,u_\epsilon^*).
\]
which, together with the density equation in \eqref{PLMSNKS}, yields
\begin{align*}
p'(\rho_\infty)\sigma_\epsilon^*
-\kappa\rho_\infty\Delta\sigma_\epsilon^*
&=
\Delta^{-1}\operatorname{div}
g_\epsilon(\sigma_\epsilon^*,u_\epsilon^*)
+(\mu+\nu)\operatorname{div}u_\epsilon^*\\
&=
\Delta^{-1}\operatorname{div}
g_\epsilon(\sigma_\epsilon^*,u_\epsilon^*)
-\epsilon^2\frac{\mu+\nu}{\rho_\infty}
\operatorname{div}(\sigma_\epsilon^* u_\epsilon^*).
\end{align*}
On the other hand, applying the Helmholtz projection \(\mathbb P\) to the momentum equation gives
\[
-\mu\Delta\mathbb Pu_\epsilon^*
=
\mathbb Pg_\epsilon(\sigma_\epsilon^*,u_\epsilon^*).
\]

We take the elliptic operator $\mathbf B$, defined in \eqref{2.1.1}, as the following specific elliptic operator:
\begin{align}\label{NJKG3.2}
\mathbf B=p'(\rho_\infty)-\kappa\rho_\infty\Delta. 
\end{align}
Therefore, the preceding identities are equivalent to
\begin{equation}\label{siguepseq}
\left\{
\begin{aligned}
\sigma_\epsilon^*
=&\,\mathbf B^{-1}\left\{
\Delta^{-1}\operatorname{div}
g_\epsilon(\sigma_\epsilon^*,u_\epsilon^*)
-\epsilon^2\frac{\mu+\nu}{\rho_\infty}
\operatorname{div}(\sigma_\epsilon^* u_\epsilon^*)
\right\},\\
u_\epsilon^*
=&\,-\frac1\mu\Delta^{-1}\mathbb P
g_\epsilon(\sigma_\epsilon^*,u_\epsilon^*)
-\frac{\epsilon^2}{\rho_\infty}
\nabla\Delta^{-1}
\operatorname{div}(\sigma_\epsilon^* u_\epsilon^*).
\end{aligned}
\right.
\end{equation}

To solve \eqref{siguepseq}, we introduce the mapping
\begin{align*}
    \mathcal J_\epsilon: (\sigma,u) \longmapsto (\bar\sigma,\bar u)
\end{align*}
where
\begin{align*}
\bar\sigma
:=&\,
\mathbf B^{-1}\left\{
\Delta^{-1}\operatorname{div}g_\epsilon(\sigma,u)
-\epsilon^2\frac{\mu+\nu}{\rho_\infty}
\operatorname{div}(\sigma u)
\right\},\\
\bar u
:=&\,
-\frac1\mu\Delta^{-1}\mathbb P g_\epsilon(\sigma,u)
-\frac{\epsilon^2}{\rho_\infty}
\nabla\Delta^{-1}\operatorname{div}(\sigma u).
\end{align*}

We work in the space
\[
X:=
\left(
\dot B^{-1/2}_{2,\infty}\cap\dot H^{k+3}
\right)
\times
\left(
\dot B^{1/2}_{2,\infty}\cap\dot H^{k+2}
\right)^3,
\]
endowed with the norm
\[
\|(\sigma,u)\|_X
:=
\|\sigma\|_{\dot B^{-1/2}_{2,\infty}\cap\dot H^{k+3}}
+
\|u\|_{\dot B^{1/2}_{2,\infty}\cap\dot H^{k+2}}.
\]
For \(\delta>0\), define the closed ball
\[
X_\delta
:=
\left\{
(\sigma,u)\in X:
\|(\sigma,u)\|_X\leq\delta
\right\}.
\]
To apply the contraction mapping theorem, it remains to control the
two types of nonlinear expressions that appear in \(\mathcal J_\epsilon\). The
first estimate concerns the momentum source
\(g_\epsilon(\sigma,u)\), whereas the second controls the correction
\(\operatorname{div}(\sigma u)\) generated by the stationary density equation. Both estimates are uniform for \(0<\epsilon\leq1\).

\begin{lem}\label{gepslem}
Assume that \(0<\delta<1\). 
There exists a constant $C>0$,  independent of \(\epsilon\in(0,1]\) and \(\delta\), such that
for any \((\sigma,u)\in X_\delta\)
\begin{equation}\label{gepsest}
\|g_\epsilon(\sigma,u)\|_{\dot B_{2,\infty}^{-\frac32}\cap\dot H^k}
\leq
C\Big(
\|(\sigma,u)\|_X^2
+\|F\|_{\dot B_{2,\infty}^{-\frac32}\cap\dot H^k}
\Big),
\end{equation}
and for any
\((\sigma_1,u_1),(\sigma_2,u_2)\in X_\delta\)
\begin{equation}\label{geps12est}
\begin{aligned}
\|g_\epsilon(\sigma_1,u_1)-g_\epsilon(\sigma_2,u_2)\|
_{\dot B_{2,\infty}^{-\frac32}\cap\dot H^k}
 \leq
C\Big(
\delta+\|F\|_{\dot B_{2,\infty}^{-\frac32}\cap\dot H^k}
\Big)
\|(\sigma_1-\sigma_2,u_1-u_2)\|_X.
\end{aligned}
\end{equation}
\end{lem}

\begin{proof}
Recall that
\begin{align*}
g_\epsilon(\sigma,u)
&=-\operatorname{div}
\bigl((\rho_\infty+\epsilon^2\sigma)u\otimes u\bigr)
-\bigl(p'(\rho_\infty+\epsilon^2\sigma)-p'(\rho_\infty)\bigr)
\nabla\sigma
\\
&\quad\,+\kappa\epsilon^2\sigma\nabla\Delta\sigma
+(\rho_\infty+\epsilon^2\sigma)F\nonumber\\
&=: g_1+g_2+g_3+g_4.
\end{align*}
We estimate the terms \(g_i\) with \(1\leq i\leq4\) in turn. For the term \(g_1\),
Proposition~\ref{BesovP} (vi) gives
\begin{align*}
 \|u\otimes u\|_{\dot B_{2,\infty}^{-\frac12}}
\lesssim
\|u\|_{\dot B_{2,\infty}^{\frac12}}^2.   
\end{align*}
Moreover, Proposition~\ref{BesovP} (iii), (vi) provide
\begin{align*}
\|\sigma u\otimes u\|_{\dot B_{2,\infty}^{-\frac12}}
\lesssim
\|\sigma\|_{\dot B_{2,1}^{\frac32}}
\|u\otimes u\|_{\dot B_{2,\infty}^{-\frac12}}
\lesssim
\|\sigma\|_{\dot B_{2,1}^{\frac32}}
\|u\|_{\dot B_{2,\infty}^{\frac12}}^2
\lesssim
\delta\|(\sigma,u)\|_X^2.
\end{align*}
For its Sobolev norm, the standard product estimate yields
\begin{align*}
\|(\rho_\infty+\epsilon^2\sigma)u\otimes u\|_{\dot H^{k+1}}
\lesssim 
(1+\|\sigma\|_{L^\infty})
\|u\|_{L^\infty}\|u\|_{\dot H^{k+1}}
 +\|\sigma\|_{\dot H^{k+1}}\|u\|_{L^\infty}^2
\lesssim
\|(\sigma,u)\|_X^2.
\end{align*}
Consequently, Proposition~\ref{BesovP} (i) gives
\begin{equation}\label{Est:g1}
\|g_1\|_{\dot B_{2,\infty}^{-\frac32}\cap\dot H^k}
\lesssim
\|(\sigma,u)\|_X^2.
\end{equation}

For the term \(g_2\), Proposition~\ref{BesovP} (i), (iii), (v), (vi) imply
\begin{align*}
\|
\bigl(p'(\rho_\infty+\epsilon^2\sigma)-p'(\rho_\infty)\bigr)
\nabla\sigma
\|_{\dot B_{2,\infty}^{-\frac32}}
\lesssim&\, 
\|p'(\rho_\infty+\epsilon^2\sigma)-p'(\rho_\infty)\|
_{\dot B_{2,\infty}^{-\frac12}}
\|\nabla\sigma\|_{\dot B_{2,1}^{\frac12}}
\\
\lesssim&\,
\bigl(1+\|\sigma\|_{\dot B_{2,1}^{\frac32}}\bigr)
\|\sigma\|_{\dot B_{2,\infty}^{-\frac12}}
\|\sigma\|_{\dot B_{2,1}^{\frac32}}
\nonumber\\
\lesssim&\,
\|(\sigma,u)\|_X^2.
\end{align*}
In addition, it is easy to obtain that
\begin{align*}
\|&
\bigl(p'(\rho_\infty+\epsilon^2\sigma)-p'(\rho_\infty)\bigr)
\nabla\sigma
\|_{\dot H^k}
 \lesssim
\|\sigma\|_{L^\infty}\|\sigma\|_{\dot H^{k+1}}
+\|\nabla\sigma\|_{L^\infty}\|\sigma\|_{\dot H^k}
\lesssim
\|(\sigma,u)\|_X^2.
\end{align*}
Therefore, we have
\begin{equation}\label{Est:g2}
\|g_2\|_{\dot B_{2,\infty}^{-\frac32}\cap\dot H^k}
\lesssim
\|(\sigma,u)\|_X^2.
\end{equation}

For  the term  \(g_3\), Proposition~\ref{BesovP} (i), (iii), (vi) and the Sobolev product estimate
give
\begin{equation}\label{Est:g3}
\begin{aligned}
\|\sigma\nabla\Delta\sigma\|
_{\dot B_{2,\infty}^{-\frac32}\cap\dot H^k}
\lesssim&\,
\|\sigma\|_{\dot B_{2,\infty}^{-\frac12}}
\|\nabla\Delta\sigma\|_{\dot B_{2,1}^{\frac12}}
+\|\sigma\|_{L^\infty}\|\sigma\|_{\dot H^{k+3}}
+\|\nabla\Delta\sigma\|_{L^\infty}\|\sigma\|_{\dot H^k}
\\
\lesssim &\,
\|(\sigma,u)\|_X^2.
\end{aligned}
\end{equation}
Finally,
for the term \(g_4\), Proposition~\ref{BesovP} (iii), (vi) yield
\begin{align}\label{Est:g4}
\|(\rho_\infty+\epsilon^2\sigma)F\|
_{\dot B_{2,\infty}^{-\frac32}\cap\dot H^k}
\lesssim&\,
\|F\|_{\dot B_{2,\infty}^{-\frac32}\cap\dot H^k}
+\|\sigma\|_{\dot B_{2,1}^{\frac32}\cap\dot H^{k+1}}
\|F\|_{\dot B_{2,\infty}^{-\frac32}\cap\dot H^k}
\lesssim\,
\|F\|_{\dot B_{2,\infty}^{-\frac32}\cap\dot H^k}.
\end{align}
Then, from \eqref{Est:g1}--\eqref{Est:g4}, we get
\eqref{gepsest}. Applying Proposition~\ref{BesovP} (i), (iii), (v), (vi)
to the difference
\begin{align*}
g_\epsilon(\sigma_1,u_1)-g_\epsilon(\sigma_2,u_2)    
\end{align*}
yields \eqref{geps12est}. Hence, the proof of Lemma \ref{gepslem} is complete.
\end{proof}

The following lemma provides the corresponding estimate for the nonlinear term arising from the density equation.

\begin{lem}\label{sigulem}
Assume that \(0<\delta<1\). There exists a constant $C>0$,  independent of \(\epsilon\in(0,1]\) and \(\delta\), such that for  any \((\sigma,u)\in X_\delta\)
\begin{equation}\label{siguest}
\|\operatorname{div}(\sigma u)\|_{\dot B_{2,\infty}^{-\frac12}\cap\dot H^{k+1}}
\leq C\|(\sigma,u)\|_X^2,
\end{equation}
and for any \((\sigma_1,u_1),(\sigma_2,u_2)\in X_\delta\)
\begin{align}\label{sigu12est}
 \|\operatorname{div}(\sigma_1u_1)-\operatorname{div}(\sigma_2u_2)\|
_{\dot B_{2,\infty}^{-\frac12}\cap\dot H^{k+1}}
 \leq
C\delta\|(\sigma_1-\sigma_2,u_1-u_2)\|_X.
\end{align}
\end{lem}

\begin{proof}
By Proposition~\ref{BesovP} (i), (iii), (vi), we have
\begin{align*}
\|\operatorname{div}(\sigma u)\|_{\dot B_{2,\infty}^{-\frac12}\cap\dot H^{k+1}}
&\lesssim
\|\sigma u\|_{\dot B_{2,\infty}^{\frac12}\cap\dot H^{k+2}}
\\
&\lesssim
\|\sigma\|_{\dot B_{2,1}^{\frac32}}
\|u\|_{\dot B_{2,\infty}^{\frac12}}
+\|\sigma\|_{L^\infty}\|u\|_{\dot H^{k+2}}
+\|u\|_{L^\infty}\|\sigma\|_{\dot H^{k+2}}
\\
&\lesssim
\|(\sigma,u)\|_X^2,
\end{align*}
which implies \eqref{siguest}. Furthermore, using the identity
\begin{align*}
\sigma_1u_1-\sigma_2u_2
=
\sigma_1(u_1-u_2)+(\sigma_1-\sigma_2)u_2.
\end{align*}
and applying the preceding estimate to the two terms on the right-hand side, together with the bounds \(\|(\sigma_i,u_i)\|_X\leq\delta\) with \(i=1,2\),  we obtain
\eqref{sigu12est}.
\end{proof}

\subsection{Proof of Theorem \ref{LMSNSKThm-intro}}

The preceding two lemmas provide the mapping and Lipschitz estimates needed for the fixed-point argument. We now prove
Theorem~\ref{LMSNSKThm-intro}. 
The proof consists of two steps: we first construct the stationary
NSK solution uniformly in \(\epsilon\), and then compare its Helmholtz components with the stationary incompressible solution given by Theorem~\ref{SNSthm}.

\begin{proof}
We first construct the stationary NSK solution. For
\((\sigma,u)\in X_\delta\), we write
\[
(\bar\sigma,\bar u)
=
\mathcal J_\epsilon(\sigma,u).
\]
By the definition of \(\mathcal J_\epsilon\), Proposition~\ref{BesovP} (i), Proposition~\ref{Prop2.3}, and Lemmas~\ref{gepslem}--\ref{sigulem},
we have, uniformly for \(0<\epsilon\leq1\),
\begin{align}
\|\bar\sigma\|_{\dot B_{2,\infty}^{-\frac12}}
&\lesssim
\|g_\epsilon(\sigma,u)\|_{\dot B_{2,\infty}^{-\frac32}}
+\epsilon^2
\|\operatorname{div}(\sigma u)\|_{\dot B_{2,\infty}^{-\frac12}}
\lesssim
\delta^2+\|F\|_{\dot B_{2,\infty}^{-\frac32}\cap\dot H^k},\label{sigepslest}\\
\|\bar\sigma\|_{\dot H^{k+3}}
&\lesssim
\|g_\epsilon(\sigma,u)\|_{\dot H^k}
+\epsilon^2
\|\operatorname{div}(\sigma u)\|_{\dot H^{k+1}}
\lesssim
\delta^2+\|F\|_{\dot B_{2,\infty}^{-\frac32}\cap\dot H^k},\label{sigepshest}\\
\|\bar u\|_{\dot B_{2,\infty}^{\frac12}}
&\lesssim
\|g_\epsilon(\sigma,u)\|_{\dot B_{2,\infty}^{-\frac32}}
+\epsilon^2
\|\operatorname{div}(\sigma u)\|_{\dot B_{2,\infty}^{-\frac12}}
\lesssim
\delta^2+\|F\|_{\dot B_{2,\infty}^{-\frac32}\cap\dot H^k},\label{uepslest}
\end{align}
and
\begin{align}
\|\bar u\|_{\dot H^{k+2}}
&\lesssim
\|g_\epsilon(\sigma,u)\|_{\dot H^k}
+\epsilon^2
\|\operatorname{div}(\sigma u)\|_{\dot H^{k+1}}
\lesssim
\delta^2+\|F\|_{\dot B_{2,\infty}^{-\frac32}\cap\dot H^k}.\label{uepshest}
\end{align}
Consequently, it follows from \eqref{sigepslest}--\eqref{uepshest} that
\begin{equation}\label{Eq:JmapsBall}
\|\mathcal J_\epsilon(\sigma,u)\|_X
\leq C_0(\delta^2+\delta_0).
\end{equation}

We next verify that the mapping $\mathcal{J}_\epsilon$ is contractive. Let
\[
(\bar\sigma_i,\bar u_i)
=
\mathcal J_\epsilon(\sigma_i,u_i),
\qquad
(\sigma_i,u_i)\in X_\delta,
\qquad i=1,2.
\]
Using Proposition~\ref{Prop2.3} and
Lemmas~\ref{gepslem}--\ref{sigulem}, we obtain that
\begin{align}\label{Contractionest1}
\|\bar\sigma_1-\bar\sigma_2\|
_{\dot B_{2,\infty}^{-\frac12}\cap\dot H^{k+3}}
\lesssim&\,
\|g_\epsilon(\sigma_1,u_1)-g_\epsilon(\sigma_2,u_2)\|
_{\dot B_{2,\infty}^{-\frac32}\cap\dot H^k}\nonumber\\
&+\epsilon^2
\|\operatorname{div}(\sigma_1u_1)
-\operatorname{div}(\sigma_2u_2)\|
_{\dot B_{2,\infty}^{-\frac12}\cap\dot H^{k+1}}
\nonumber\\
\leq&\,C_1
\left(\delta+\|F\|_{\dot B_{2,\infty}^{-\frac32}\cap\dot H^k}\right)
\|(\sigma_1-\sigma_2,u_1-u_2)\|_X,
\end{align}
and
\begin{align}\label{Contractionest2}
\|\bar u_1-\bar u_2\|
_{\dot B_{2,\infty}^{\frac12}\cap\dot H^{k+2}}
 \lesssim&\,
\|g_\epsilon(\sigma_1,u_1)-g_\epsilon(\sigma_2,u_2)\|
_{\dot B_{2,\infty}^{-\frac32}\cap\dot H^k}\nonumber\\
&+\epsilon^2
\|\operatorname{div}(\sigma_1u_1)
-\operatorname{div}(\sigma_2u_2)\|
_{\dot B_{2,\infty}^{-\frac12}\cap\dot H^{k+1}}
\nonumber\\
\leq&\,
C_1\left(\delta+\|F\|_{\dot B_{2,\infty}^{-\frac32}\cap\dot H^k}\right)
\|(\sigma_1-\sigma_2,u_1-u_2)\|_X.
\end{align}
 
Choose \(\delta_0>0\) and \(\delta>0\) such that
\[
C_0(\delta^2+\delta_0)\leq\delta,
\qquad
C_1(\delta+\delta_0)\leq\frac14,
\]
Then \(\mathcal J_\epsilon\) maps \(X_\delta\) into itself and is a contraction on \(X_\delta\), uniformly for \(0<\epsilon\leq1\).
Thus,  based on the contraction mapping theorem, there is a fixed-point
$
(\sigma_\epsilon^*,u_\epsilon^*)\in X_\delta
$
for \(\mathcal J_\epsilon\), which, together with \eqref{siguepseq}, implies that this fixed point solves
the stationary perturbation system \eqref{PLMSNKS}.
The fixed-point construction guarantees the uniqueness of $(\sigma_\epsilon^*,u_\epsilon^*)$ in \(X_\delta\).

To obtain the estimate stated in the theorem, we apply
\eqref{gepsest} and \eqref{siguest} once more to the fixed point and obtain
\[
\|(\sigma_\epsilon^*,u_\epsilon^*)\|_X
\leq
C\left(
\|(\sigma_\epsilon^*,u_\epsilon^*)\|_X^2+\|F\|_{\dot B_{2,\infty}^{-\frac32}\cap\dot H^k}
\right).
\]
By reducing \(\delta\) if necessary, the quadratic term is absorbed,
and hence
\begin{equation}\label{Eq:StationaryUniformBound}
\|(\sigma_\epsilon^*,u_\epsilon^*)\|_X
\leq C\|F\|_{\dot B_{2,\infty}^{-\frac32}\cap\dot H^k}.
\end{equation}
Recalling the stationary notation,
\[
(\rho_\epsilon^*, u_\eps^*)
=
(\rho_\infty+\epsilon^2\sigma_\epsilon^*,u_\epsilon^*),
\]
and combining \eqref{Eq:StationaryUniformBound} with
Theorem~\ref{SNSthm}, we obtain the uniform estimate \eqref{siguepsest-intro}.

It remains to prove the second-order convergence in
\eqref{machest-intro}. For this purpose, we define
\[
v_{1,\epsilon}:=\mathbb P u_\epsilon^*-u^*,
\qquad
v_{2,\epsilon}:=\mathbb Q u_\epsilon^*,
\]
and compare the stationary NSK velocity with the stationary
incompressible velocity. Applying \(\mathbb P\) to the stationary
momentum equation, subtracting
the projected stationary incompressible equation of \eqref{SNS}, and using
the fact that the pressure contribution is a gradient, we obtain
\begin{equation}\label{SteadyMachEq}
\left\{
\begin{aligned}
v_{1,\epsilon}
=&
\frac{\rho_\infty}{\mu}\Delta^{-1}\mathbb P
\operatorname{div}\Bigl(
(u_\epsilon^*-u^*)\otimes u_\epsilon
+u^*\otimes(u_\epsilon^*-u^*)
\Bigr)
\\
&+\frac{\epsilon^2}{\mu}\Delta^{-1}\mathbb P
\Bigl(
\operatorname{div}(\sigma_\epsilon^* u_\epsilon^*\otimes u_\epsilon^*)
-\kappa\sigma_\epsilon^*\nabla\Delta\sigma_\epsilon^*
-\sigma_\epsilon^* F
\Bigr),
\\
v_{2,\epsilon}
=&
-\frac{\epsilon^2}{\rho_\infty}
\nabla\Delta^{-1}
\operatorname{div}(\sigma_\epsilon^* u_\epsilon^*),
\end{aligned}
\right.
\end{equation}
where the second identity follows directly from the density equation in
\eqref{PLMSNKS}.

By Lemma~\ref{sigulem} and \eqref{Eq:StationaryUniformBound}, we have
\begin{align}\label{v2est}
\|v_{2,\epsilon}\|_{\dot B_{2,\infty}^{\frac12}\cap\dot H^{k+2}}
 \lesssim
\epsilon^2
\|\operatorname{div}(\sigma_\epsilon^* u_\epsilon^*)\|
_{\dot B_{2,\infty}^{-\frac12}\cap\dot H^{k+1}}
 \lesssim
\epsilon^2\|(\sigma_\epsilon^*,u_\epsilon^*)\|_X^2
\lesssim
\epsilon^2\delta_0^2.
\end{align}
Moreover, Proposition~\ref{BesovP} (i), (vi),
Lemma~\ref{gepslem}, Theorem~\ref{SNSthm}, and
\eqref{Eq:StationaryUniformBound} give
\begin{equation}\label{v1est}
\begin{aligned}
\|v_{1,\epsilon}\|_{\dot B_{2,\infty}^{\frac12}\cap\dot H^{k+2}}
&\lesssim
\left(
\|u_\epsilon^*\|_{\dot B_{2,\infty}^{\frac12}\cap\dot H^{k+2}}
+\|u^*\|_{\dot B_{2,\infty}^{\frac12}\cap\dot H^{k+2}}
\right)
\|u_\epsilon^*-u^*\|_{\dot B_{2,\infty}^{\frac12}\cap\dot H^{k+2}}
+\epsilon^2\delta_0^2
\\
&\lesssim
\delta_0\|u_\epsilon^*-u^*\|_{\dot B_{2,\infty}^{\frac12}\cap\dot H^{k+2}}
+\epsilon^2\delta_0^2.
\end{aligned}
\end{equation}
Since
\[
u_\epsilon^*-u^*
=
v_{1,\epsilon}+v_{2,\epsilon},
\]
the estimates \eqref{v2est}--\eqref{v1est} imply
\[
\|u_\epsilon^*-u^*\|_{\dot B_{2,\infty}^{\frac12}\cap\dot H^{k+2}}
\leq
C\delta_0\|u_\epsilon^*-u^*\|_{\dot B_{2,\infty}^{\frac12}\cap\dot H^{k+2}}
+C\epsilon^2\delta_0^2.
\]
Choosing \(\delta_0>0\) sufficiently small, we conclude that
\begin{align*}
\|u_\epsilon^*-u^*\|_{\dot B_{2,\infty}^{\frac12}\cap\dot H^{k+2}}
\leq
C\epsilon^2\delta_0^2,
\end{align*} 
which, together with \eqref{v2est},  yields
\begin{align*}
\|(\mathbb Q u_\epsilon^*,
\mathbb P u_\epsilon^*-u^*)\|
_{\dot B_{2,\infty}^{\frac12}\cap\dot H^{k+2}}
\leq
C\epsilon^2 \delta_0^2. 
\end{align*}
Thus, we obtain \eqref{machest-intro}, thereby completing the proof of Theorem \ref{LMSNSKThm-intro}.
\end{proof}

\section{The non-stationary NSK equations}\label{Sec:NSNSK}
In this section, we investigate the non-stationary problem. We first derive the perturbation system, establish the high-order and low-frequency a priori estimates, and then prove global existence and decay. Finally, we analyze the acoustic-capillary semigroup and complete the convergence estimates between the compressible and incompressible components.

\subsection{Perturbation system and source estimates} \label{S4.1}
Recalling \eqref{Sigweps-intro} together with 
\begin{align*}
\rho_\epsilon^*=\rho_\infty+\eps^2 \sigma_\epsilon^*,\quad \A=\mu \Delta + \nu \nabla \text{div},
\end{align*}
we obtain the following perturbation equations on $(\sigma_\epsilon,w_\epsilon)$:
\begin{equation}\label{Eq:MainPerEq}
\left\{
\begin{aligned}
&\partial_t \sigma_\epsilon+\frac{\rho_\infty}{\epsilon}\text{div}w_\epsilon=-\text{div}H,\\
&\partial_t w_\epsilon-\Psi(\rho_\epsilon)\A w_\epsilon+\frac1\epsilon \nabla \K\sigma_\eps=G,
\end{aligned}
\right.
\end{equation}
where
\begin{align*}
H:=\sigma_\eps (u_\eps^*+w_\eps)+\eps \sigma_\eps^* w_\eps,
\end{align*}
and
\begin{align*}
G:=&-u_\eps^*\cdot\nabla w_\epsilon-w_\epsilon\cdot\nabla u_\eps^*
-w_\epsilon\cdot\nabla w_\epsilon+\bigl(\Psi(\rho_\eps)-\Psi(\rho_\eps^*)\bigr)\A u_\eps^*\\
&-\frac1\eps\bigl(\Phi(\rho_\eps)-\gamma_0\bigr)\nabla\sigma_\epsilon-\bigl(\Phi(\rho_\eps)-\Phi(\rho_\eps^*)\bigr)\nabla \sigma_\eps^* .
\end{align*}
To apply the method of constant-coefficient linear semigroup, we reformulate the system \eqref{Eq:MainPerEq} as follows:
\begin{equation}\label{Eq:MainConCoeffEq}
\left\{
\begin{aligned}
&\partial_t \sigma_\epsilon+\frac{\rho_\infty}{\epsilon}\text{div}w_\epsilon=-\text{div}H,\\
&\partial_t w_\eps-\rho_\infty^{-1}\A w_\eps+\frac1\eps \nabla \K\sigma_\eps=\mathcal{G},
\end{aligned}
\right.
\end{equation}
where
\begin{align*}
\mathcal{G}:=G+\bigl(\Psi(\rho_\eps)-\rho_\infty^{-1}\bigr)\A w_\eps.
\end{align*}

Having derived the perturbation systems \eqref{Eq:MainPerEq} and \eqref{Eq:MainConCoeffEq}, we introduce our bootstrap framework and turn to the estimates for the source terms. Denote
\begin{align*}
M_N(t):=\sup_{0\le \tau\le t} \mathcal{E}_N(\tau),\quad \delta_F=\|F\|_{\dot B_{2,\infty}^{-\frac{3}{2}} \cap \dot H^{N+1}},
\end{align*}
and assume the bootstrap assumption
\begin{equation}\label{bootstrap}
M_N(t)+\delta_F\le \delta_0,
\end{equation}
where $\delta_0>0$ is sufficiently small. The source terms of the systems \eqref{Eq:MainPerEq} and \eqref{Eq:MainConCoeffEq} can be estimated as follows.

\begin{lem}[Density source term]\label{Hest}
Assume that \eqref{bootstrap} holds, then we have
\begin{equation}\label{Hhest}
\|\K^{1/2}H\|_{H^N}\le C\bigl(\delta_F+M_N(t)\bigr) \mathcal{D}_N(t),
\end{equation}
and
\begin{equation}\label{Hlest}
\|\K^{1/2}H\|_{\dot B^{1/2}_{2,\infty}} \le C\bigl(\delta_F+M_N(t)\bigr)\|(\K^{1/2}\sigma_\eps,w_\eps)\|_{\dot B^{3/2}_{2,\infty}\cap H^N}.
\end{equation}
\end{lem}

\begin{proof}
It follows from \eqref{Khalfest-intro}, Theorem \ref{LMSNSKThm-intro} and Lemma \ref{LA.1} that
\begin{align*}
\|\K^{1/2}(\sigma_\eps u_\eps^*)\|_{H^N}\lesssim \|\sigma_\eps u_\eps^*\|_{H^{N+1}} \lesssim 
\|u_\eps^*\|_{H^{N+1}}\|\sigma_\eps\|_{L^\infty}+\|u_\eps^*\|_{L^\infty}\|\sigma_\eps\|_{H^{N+1}}
\lesssim \delta_F\|\K^{1/2}\sigma_\eps\|_{H^N}.
\end{align*}
Similarly, 
\begin{align*}
\|\K^{1/2}(\sigma_\eps w_\eps)\|_{H^N}\lesssim\|\sigma_\eps\|_{H^{N+1}}\|w_\eps\|_{H^{N+1}}
\lesssim M_N(t) \mathcal{D}_N(t),
\end{align*}
and
\begin{align*}
\eps\|\K^{1/2}(\sigma_\eps^* w_\eps)\|_{H^N} \lesssim \delta_F \|w_\eps\|_{H^{N+1}} \lesssim \delta_F \mathcal{D}_N(t).
\end{align*}
Then \eqref{Hhest} holds. 

For the Besov estimate, by virtue of Proposition \ref{BesovP} (vi), we have
\begin{align*}
\|\K^{1/2}(\sigma_\eps u_\eps^*)\|_{\dot B^{1/2}_{2,\infty}} 
&\lesssim \|\sigma_\eps u_\eps^*\|_{\dot B^{1/2}_{2,\infty}\cap H^{N}}
\lesssim \|\sigma_\eps\|_{B^{3/2}_{2,\infty}\cap H^{N}}\|u_\eps^*\|_{\dot B^{1/2}_{2,\infty}\cap H^{N}}\\
&\le C\bigl(\delta_F+M_N(t)\bigr)\|(\K^{1/2}\sigma_\eps,w_\eps)\|_{\dot B^{3/2}_{2,\infty}\cap H^N}.
\end{align*}
The remaining two terms in $H$ can be controlled similarly.  Thus, the proof of Lemma \ref{Hest} is complete.
\end{proof}

\begin{lem}[Momentum source term]\label{Gest}
Assume that \eqref{bootstrap} holds, then we have
\begin{equation}\label{GHighest}
\sum_{|\alpha|\le N}\bigl|\langle \partial^\alpha w_\eps,\partial^\alpha G\rangle\bigr|
\le C\bigl(\delta_F+M_N(t)\bigr)\mathcal{D}_N(t)^2,
\end{equation}
and
\begin{equation}\label{GLowest}
\| \mathcal{G}\|_{\dot B^{-1/2}_{2,\infty}}
\le C\bigl(\delta_F+M_N(t)\bigr)\, \|(\K^{1/2}\sigma_\eps,w_\eps)\|_{\dot B^{3/2}_{2,\infty}}.
\end{equation}
\end{lem}

\begin{proof}
We split $G$ and $\mathcal{G}$ into the following five classes.

\textbf{(1) Stationary convection terms:} $-u_\eps^*\cdot\nabla w_\epsilon - w_\eps\cdot\nabla u_\eps^*.$ By Lemma \ref{Lem:L6ab}, we have
\begin{equation*}
\sum_{|\alpha|\le N}
\bigl|
\langle \partial^\alpha w_\eps,\partial^\alpha(u_\eps^*\cdot\nabla w_\eps+w_\eps\cdot\nabla u_\eps^*)\rangle
\bigr| \lesssim\delta_F\|\nabla w_\eps\|_{H^N}^2.
\end{equation*}
Moreover, Proposition \ref{BesovP} (vi) yields
\begin{align*}
\|u_\eps^*\cdot\nabla w_\epsilon + w_\eps\cdot\nabla u_\eps^*\|_{\dot B^{-1/2}_{2,\infty}}
\lesssim & \|u_\eps^*\|_{\dot B^{1/2}_{2,\infty}}\|\nabla w_\epsilon \|_{\dot B^{1/2}_{2,\infty}}
+\|\nabla u_\eps^*\|_{\dot B^{-1/2}_{2,\infty}}\|w_\epsilon\|_{\dot B^{3/2}_{2,\infty}}\\
\lesssim & \bigl(\delta_F+M_N(t)\bigr)\, \|(\K^{1/2}\sigma_\eps,w_\eps)\|_{\dot B^{3/2}_{2,\infty}}.
\end{align*}

\textbf{(2) Quadratic convection term:} $-w_\eps\cdot\nabla w_\eps.$ Lemma \ref{LA.1} implies
\begin{align*}
\sum_{|\alpha|\le N}\bigl|\langle \partial^\alpha w_\eps,\partial^\alpha(w_\eps\cdot\nabla w_\eps)\rangle\bigr|=&\sum_{|\alpha|\le N}\bigl|\langle \partial^\alpha w_\eps,[\partial^\alpha,w_\eps\cdot\nabla] w_\eps\rangle\bigr|\\
\lesssim & \|w_\eps\|_{H^N}\|\nabla w_\eps\|_{H^N}^2
\lesssim  M_N(t)\mathcal{D}_N(t)^2.
\end{align*}
 Proposition \ref{BesovP} (vi) gives
\begin{align*}
\|w_\eps\cdot\nabla w_\eps\|_{\dot B^{-1/2}_{2,\infty}}
\lesssim  \|w_\eps\|_{\dot B^{1/2}_{2,\infty}}\|\nabla w_\epsilon \|_{\dot B^{1/2}_{2,\infty}}
\lesssim  \bigl(\delta_F+M_N(t)\bigr)\, \|(\K^{1/2}\sigma_\eps,w_\eps)\|_{\dot B^{3/2}_{2,\infty}}.
\end{align*}

\textbf{(3) Other terms involving stationary solution:} $\bigl(\Psi(\rho_\eps)-\Psi(\rho_\eps^*)\bigr)\A u_\eps^*-\big(\Phi(\rho_\eps)-\Phi(\rho_\eps^*)\big)\nabla \sigma_\eps^*.$
Using Lemma \ref{LA.1}, we obtain 
\begin{align*}
\sum_{|\alpha|\le N}\bigl|\big\langle \partial^\alpha w_\eps,\partial^\alpha\big(\bigl(\Psi(\rho_\eps)-\Psi(\rho_\eps^*)\bigr)\A u_\eps^*\big)\big\rangle\bigr|
\lesssim\delta_F\|\nabla w_\eps\|_{H^N}\|\sigma_\eps\|_{H^N},
\end{align*}
and
\begin{align*}
\sum_{|\alpha|\le N}\big|
\big\langle \partial^\alpha w_\eps,\partial^\alpha\bigl(\big(\Phi(\rho_\eps)-\Phi(\rho_\eps^*)\big)\nabla \sigma_\eps^*\bigr)\big\rangle\big|
\lesssim\delta_F \mathcal{D}_N(t)^2.
\end{align*}
For the Besov estimates, Proposition~\ref{BesovP} (v), (vi) give
\begin{align*}
\|\bigl(\Psi(\rho_\eps)-\Psi(\rho_\eps^*)\bigr)\A u_\eps^*\|_{\dot B^{-1/2}_{2,\infty}}
&\lesssim  \|\A u_\eps^*\|_{\dot B^{-1/2}_{2,\infty}}\|\Psi(\rho_\eps)-\Psi(\rho_\eps^*) \|_{\dot B^{3/2}_{2,\infty}}
\lesssim \|u_\eps^*\|_{\dot B^{3/2}_{2,\infty}} \|\epsilon\sigma_\eps\|_{\dot B^{3/2}_{2,\infty}}\\
&\lesssim  \bigl(\delta_F+M_N(t)\bigr)\, \|(\K^{1/2}\sigma_\eps,w_\eps)\|_{\dot B^{3/2}_{2,\infty}},
\end{align*}
Similarly, we also have
\begin{align*}
\|\big(\Phi(\rho_\eps)-\Phi(\rho_\eps^*)\big)\nabla \sigma_\eps^*\|_{\dot B^{-1/2}_{2,\infty}}
&\lesssim  \|\nabla \sigma_\eps^*\|_{\dot B^{-1/2}_{2,\infty}}\|\Phi(\rho_\eps)-\Phi(\rho_\eps^*) \|_{\dot B^{3/2}_{2,\infty}}
\lesssim \|\sigma_\eps^*\|_{\dot B^{1/2}_{2,\infty}} \|\epsilon\sigma_\eps\|_{\dot B^{3/2}_{2,\infty}}\\
&\lesssim  \bigl(\delta_F+M_N(t)\bigr)\, \|(\K^{1/2}\sigma_\eps,w_\eps)\|_{\dot B^{3/2}_{2,\infty}}.
\end{align*}

\textbf{(4) Pressure nonlinearity:} $-\frac1\eps\big(\Phi(\rho_\eps)-\gamma_0\big)\nabla\sigma_\eps.$ We easily get
\begin{align*}
\left\|\eps^{-1}\big(\Phi(\rho_\eps)-\gamma_0\big)\nabla\sigma_\eps\right\|_{H^N}
\lesssim\bigl(M_N(t)+\delta_F\bigr)\|\nabla\sigma_\eps\|_{H^N}
\lesssim\bigl(M_N(t)+\delta_F\bigr)\mathcal{D}_N(t).
\end{align*}
Consequently, we have
\begin{align*}
\left\|\big\langle\partial^\alpha w_\eps,\eps^{-1}\big(\Phi(\rho_\eps)-\gamma_0\big)\nabla\sigma_\eps\big\rangle\right\|_{H^N}
\lesssim\bigl(M_N(t)+\delta_F\bigr)\mathcal{D}_N(t)^2.
\end{align*}
Furthermore,  Proposition \ref{BesovP} (v), (vi) yield 
\begin{align*}
\big\|\eps^{-1}\big(\Phi(\rho_\eps)-\gamma_0\big)\nabla\sigma_\eps\big\|_{\dot B^{-1/2}_{2,\infty}}
&\lesssim  \big\|\eps^{-1}\big(\Phi(\rho_\eps)-\gamma_0\big)\big\|_{\dot B^{1/2}_{2,\infty}}\|\nabla\sigma_\eps \|_{\dot B^{1/2}_{2,\infty}}
\lesssim \|\sigma_\eps+\eps \sigma_\eps^*\|_{\dot B^{1/2}_{2,\infty}} \|\sigma_\eps\|_{\dot B^{3/2}_{2,\infty}}\\
&\lesssim  \bigl(\delta_F+M_N(t)\bigr)\, \|(\K^{1/2}\sigma_\eps,w_\eps)\|_{\dot B^{3/2}_{2,\infty}}.
\end{align*}

\textbf{(5) Additional term for $\mathcal{G}$:} $\bigl(\Psi(\rho_\eps)-\rho_\infty^{-1}\bigr)\A w_\eps.$ By Proposition \ref{BesovP} (v), (vi), we obtain
\begin{align*}
\big\|\bigl(\Psi(\rho_\eps)-\rho_\infty^{-1}\bigr)\A w_\eps\big\|_{\dot B^{-1/2}_{2,\infty}}
&\lesssim  \|\Psi(\rho_\eps)-\rho_\infty^{-1}\|_{\dot B^{3/2}_{2,\infty}}\|\A w_\eps \|_{\dot B^{-1/2}_{2,\infty}}
\lesssim \epsilon\|\sigma_\eps+\eps \sigma_\eps^*\|_{\dot B^{3/2}_{2,\infty}} \|w_\eps\|_{\dot B^{3/2}_{2,\infty}}\\
&\lesssim  \bigl(\delta_F+M_N(t)\bigr)\, \|(\K^{1/2}\sigma_\eps,w_\eps)\|_{\dot B^{3/2}_{2,\infty}}.
\end{align*}
Combining all the above estimates yields \eqref{GHighest} and \eqref{GLowest}. 
\end{proof}

\subsection{Uniform a priori estimates}\label{S4.2}
Now, based on the estimates of source terms established in the previous section, we derive uniform bootstrap a priori estimates for the solutions to the perturbation NSK system. Moreover, combining high-order dissipative energy inequalities and low-frequency dyadic estimates with interpolation arguments, we establish global-in-time existence and time‑decay properties.

We start with high-frequency energy estimates. For each multi-index $\alpha$ with $|\alpha|\le N$, define
\begin{align*}
E_\alpha^{h,0}(t):= \frac12\Bigl(\|\K^{1/2}\partial^\alpha \sigma_\eps\|_{L^2}^2 + \rho_\infty \|\partial^\alpha w_\eps\|_{L^2}^2\Bigr).
\end{align*}
Applying the operator $\partial^\alpha$ to \eqref{Eq:MainPerEq}, and then taking the inner products of the resulting two equations with $\K\partial^\alpha \sigma_\epsilon$ and $\rho_\infty \partial^\alpha w_\epsilon$, respectively,
we obtain
\begin{align*}
&\frac{\rm d}{{\rm d}t}E_\alpha^{h,0}+ \rho_\infty\mu \int \Psi(\rho_\eps)|\nabla\partial^\alpha w_\eps|^2\,{\rm d}x
+\rho_\infty\lambda \int \Psi(\rho_\eps)|\operatorname{div}\partial^\alpha w_\eps|^2\,{\rm d}x\nonumber\\
=& -\langle \K\partial^\alpha \sigma_\eps,\partial^\alpha \text{div} H\rangle
+
\rho_\infty\langle \partial^\alpha w_\eps,\partial^\alpha G\rangle
+
\rho_\infty\langle \partial^\alpha w_\eps,[\partial^\alpha,\Psi(\rho_\eps)]\A w_\eps\rangle,
\end{align*}
where we have used the cancellation of the singular terms
\begin{align*}
\frac{\rho_\infty}{\epsilon}\int (\partial^\alpha\text{div} w_\eps\, \K\partial^\alpha\sigma_\eps + \partial^\alpha w_\eps\cdot \partial^\alpha\nabla \K \sigma_\eps)\, {\rm d}x=0.
\end{align*}
Since $\rho_\eps\ge \rho_\infty/2$, there exists a constant $c_\mu>0$ such that the viscosity part satisfies
\begin{align*}
\rho_\infty\mu \int \Psi(\rho_\eps)|\nabla\partial^\alpha w_\eps|^2\,{\rm d}x
+\rho_\infty\lambda \int \Psi(\rho_\eps)|\operatorname{div}\partial^\alpha w_\eps|^2\,{\rm d}x
\ge c_\mu \|\nabla\partial^\alpha w\|_{L^2}^2.
\end{align*}
By Lemma \ref{LA.1}, we have
\begin{align*}
|\langle \partial^\alpha w_\eps,[\partial^\alpha,\Psi(\rho_\eps)]\A w_\eps\rangle|\le C\bigl(\delta_F+M_N(t)\bigr)\mathcal{D}_N(t)^2.
\end{align*} 
For the source terms, Lemmas \ref{Hest} and \ref{Gest} yields
\begin{align*}
\sum_{|\alpha|\le N} \bigl|\langle \K\partial^\alpha \sigma_\eps,\partial^\alpha f\rangle\bigr|
=\bigl|\langle \nabla \K^{1/2}\partial^\alpha \sigma_\eps,\K^{1/2}\partial^\alpha H\rangle \bigr|
\le C\bigl(\delta_F+M_N(t)\bigr)\mathcal{D}_N(t)^2,
\end{align*}
and
\begin{align*}
\sum_{|\alpha|\le N} \bigl|\langle \partial^\alpha w_\eps,\partial^\alpha G\rangle\bigr|
\le C\bigl(\delta_F+M_N(t)\bigr)\mathcal{D}_N(t)^2.
\end{align*}
Therefore,
\begin{equation}\label{Eest}
\frac{{\rm d}}{{\rm d}t}\sum_{|\alpha|\le N}E_\alpha^{h,0}+c_\mu \|\nabla w_\eps\|_{H^N}^2
\le C\bigl(\delta_F+M_N(t)\bigr)\mathcal{D}_N(t)^2.
\end{equation}

Let $0<\eta\ll 1$ and define
\begin{equation}\label{KawaJ}
J_\alpha^h(t):=\eta \eps \langle \partial^\alpha w_\eps,\nabla\partial^\alpha \sigma_\eps\rangle,
\qquad |\alpha|\le N.
\end{equation}
It follows from \eqref{Khalfest-intro} that
\begin{equation}\label{Jest}
|J_\alpha^h(t)|\le C\eta E_\alpha^{h,0}(t).
\end{equation}
Differentiating \eqref{KawaJ} with respect to time and using the system \eqref{Eq:MainPerEq}, we have
\begin{equation}\label{Kawaest}
\begin{aligned}
&\frac{{\rm d}}{{\rm d}t} J_\alpha^h 
+\eta \langle \nabla \mathcal{K}\partial^\alpha \sigma_\eps,\nabla\partial^\alpha \sigma_\eps\rangle
+\eta\rho_\infty
\langle \partial^\alpha w_\eps,\nabla(\operatorname{div}\partial^\alpha w_\eps)\rangle \\
=&-\eta\eps \langle \partial^\alpha w_\eps,\nabla \partial^\alpha \text{div}H\rangle
+\eta\eps \langle \nabla\partial^\alpha \sigma_\epsilon, \partial^\alpha G\rangle
+\eta\eps \langle \nabla\partial^\alpha \sigma_\epsilon, \partial^\alpha (\Psi(\rho_\eps)\A w_\eps)\rangle.
\end{aligned}
\end{equation}
For the terms on the left-hand side of \eqref{Kawaest}, we easily get
\begin{align*}
\eta \langle \nabla \mathcal{K}\partial^\alpha \sigma_\eps,\nabla\partial^\alpha \sigma_\eps\rangle
=\eta \|\nabla \mathcal{K}^{1/2}\partial^\alpha \sigma_\eps\|_{L^2}^2,
\end{align*}
and
\begin{align*}
\eta\rho_\infty
\langle \partial^\alpha w_\eps,\nabla(\operatorname{div}\partial^\alpha w_\eps)\rangle
=-\eta\rho_\infty \|\operatorname{div}\partial^\alpha w_\eps\|_{L^2}^2.
\end{align*}
For the terms on the right-hand side of \eqref{Kawaest}, 
using Lemmas \ref{Hest} and \ref{Gest}, we have
\begin{align*}
|\eta\eps \langle \partial^\alpha w_\eps,\nabla \partial^\alpha \text{div}H\rangle|
&=
|\eta\eps \langle \operatorname{div}\partial^\alpha w_\eps,\operatorname{div}\partial^\alpha H\rangle| \nonumber\\
&\le C\eps \|\nabla \partial^\alpha w_\eps\|_{L^2}\|\nabla\partial^\alpha H\|_{L^2}
\le C\bigl(\delta_F+M_N(t)\bigr)\mathcal{D}_N(t)^2,
\end{align*}
and
\begin{align*}
|\eta\eps \langle \nabla\partial^\alpha \sigma_\epsilon, \partial^\alpha G\rangle|
 {
\le \eta\eps (\|\nabla \partial^\alpha \sigma_\eps\|_{L^2}^{2}+\|\partial^\alpha G\|_{L^2}^{2})
\le \eta\eps \|\nabla \partial^\alpha \sigma_\eps\|_{L^2}^{2}+C\bigl(\delta_F+M_N(t)\bigr)\mathcal{D}_N(t)^2.}
\end{align*}
For the last term of \eqref{Kawaest}, it is easy to obtain 
\begin{align}
|\eta\eps \langle \nabla\partial^\alpha \sigma_\epsilon, \partial^\alpha (\Psi(\rho_\eps)\A w_\eps)\rangle|
\le \eta\eps (\| \nabla^2\partial^{\alpha} \sigma_\epsilon\|_{L^2}^2
+ \|\nabla^{|\alpha|-1} (\Psi(\rho_\eps)\A w_\eps)\|_{L^2}^2).\label{4.12.1}
\end{align}
Applying Lemma \ref{LA.1} to the last term in \eqref{4.12.1}, we  derive
\begin{align*}
\|\nabla^{|\alpha|-1} (\Psi(\rho_\eps)\A w_\eps)\|_{L^2}^2\lesssim & \,\|\nabla^{|\alpha|+1}w_\eps\|_{L^2}^2+ \|[\nabla^{|\alpha|-1},\Phi(\rho_\eps)]\A w_\eps\|_{L^2}^2\notag\\
\lesssim & \,\|\nabla^{|\alpha|+1}w_\eps\|_{L^2}^2+\bigl(\delta_F+M_N(t)\bigr)\mathcal{D}_N(t)^2.
\end{align*}
Thus, combining the above estimates and summing over $|\alpha|\leq N$, we arrive at 
\begin{equation}\label{Jst}
\frac{{\rm d}}{{\rm d}t}\sum_{|\alpha|\le N}J_\alpha^h+c_1\eta \|\nabla K^{1/2}\sigma_\eps\|_{H^N}^2
\le c_2\eta \|\nabla w_\eps\|_{H^N}^2+ C\bigl(\delta_F+M_N(t)\bigr)\mathcal{D}_N(t)^2.
\end{equation}

Define
\begin{align*}
E_N^h(t):=\sum_{|\alpha|\le N}\bigl(E_\alpha^{h,0}(t)+J_\alpha^h(t)\bigr).
\end{align*}
It follows from \eqref{Jest} that for sufficiently small $\eta$,
\begin{align}
E_N^h(t)\sim \|(\K^{1/2}\sigma_\eps,w_\eps)(t)\|_{H^N}^2.\label{4.14.1}
\end{align}
Combining \eqref{Eest} with \eqref{Jst} yields
\begin{align*}
\frac{{\rm d}}{{\rm d}t}E_N^h(t)+c_0 \mathcal{D}_N(t)^2 \le C\bigl(\delta_F+M_N(t)\bigr)\mathcal{D}_N(t)^2.
\end{align*}
Furthermore,  choosing $\delta_0$ sufficiently small, it holds that
\begin{equation}\label{Eng}
\frac{{\rm d}}{{\rm d}t}E_N^h(t)+\frac{c_0}{2}\mathcal{D}_N(t)^2\le 0.
\end{equation}
Integrating \eqref{Eng} over time and using \eqref{4.14.1}, we conclude 
\begin{equation}\label{HighFeqest}
\|(\K^{1/2}\sigma_\eps,w_\eps)(t)\|_{H^N}^2+\int_0^t \mathcal{D}_N(\tau)^2\,{\rm d}\tau
\le C\|(\K^{1/2}\sigma_0,w_0)\|_{H^N}^2.
\end{equation}

Next, we establish low-frequency dyadic estimates. Let
\begin{align*}
\sigma_j:=\dot\Delta_j \sigma_\eps,\quad
w_j:=\dot\Delta_j w_\eps,\quad
H_j:=\dot\Delta_j H,\quad
G_j:=\dot\Delta_j G.
\end{align*}
Define the block energy
\begin{align*}
E_j(t):=E_j^0(t)+J_j(t):=\frac12\Bigl(\|\K^{1/2}\sigma_j\|_{L^2}^2+\rho_\infty\|w_j\|_{L^2}^2\Bigr)
+\eta\eps \langle w_j,\nabla\sigma_j\rangle.
\end{align*}
As in \eqref{Jest}, for $\eta$ small enough,
\begin{align*}
E_j(t)\sim \|\K^{1/2}\sigma_j\|_{L^2}^2+\|w_j\|_{L^2}^2.
\end{align*}

We now provide the estimate for $E_j(t)$ in the following lemma.
\begin{lem}\label{Lem:dyaest}
Assume that \eqref{bootstrap} holds. For every $j\in\mathbb Z$,
\begin{equation}\label{dyaest}
\frac{\rm d}{{\rm d}t}E_j + c\,2^{2j}E_j\le C2^j\|\K^{1/2}H_j\|_{L^2}E_j^{1/2}
+C\|\mathcal{G}_j\|_{L^2}E_j^{1/2},
\end{equation}
where $\{c_j\}_{j\in\mathbb Z}\in \ell^1$ and $\|c_j\|_{\ell^1}\le 1$.
\end{lem}

\begin{proof}
Applying \(\dot\Delta_j\) to the system \eqref{Eq:MainConCoeffEq} yields
\begin{equation}\label{Eq:dya}
\left\{
\begin{aligned}
&\partial_t\sigma_j+\frac{\rho_\infty}{\eps}\text{div} w_j=-\text{div} H_j,\\
&\partial_t w_j-\rho_\infty^{-1}\A w_j+\frac1\eps\nabla\mathcal K\sigma_j=\mathcal G_j.
\end{aligned}
\right.
\end{equation}
Taking the inner products of the first equation in \eqref{Eq:dya} with \(\mathcal K\sigma_j\) and the second equation in \eqref{Eq:dya} with \(\rho_\infty w_j\), 
we have
\begin{align*}
\frac12\frac{{\rm d}}{{\rm d}t}\left(\|\mathcal K^{1/2}\sigma_j\|_{L^2}^2+\rho_\infty\|w_j\|_{L^2}^2\right)
-\langle \A w_j,w_j\rangle
=-\langle \text{div} H_j,\mathcal K\sigma_j\rangle+\rho_\infty\langle \mathcal G_j,w_j\rangle,
\end{align*}
where we have used the cancellation of the singular terms
\begin{align*}
\langle \nabla\mathcal K\sigma_j,w_j\rangle=-\langle \mathcal K\sigma_j,\text{div} w_j\rangle.
\end{align*}
Moreover, by the ellipticity of \(\A\) and H\"{o}lder inequality, we arrive at
\begin{align}\label{Ej}
\frac{{\rm d}}{{\rm d}t}E_j^{0}+c_0\|\nabla w_j\|_{L^2}^2 \le C2^j\|\mathcal K^{1/2}H_j\|_{L^2}
\|\mathcal K^{1/2}\sigma_j\|_{L^2}+C\|\mathcal G_j\|_{L^2}\|w_j\|_{L^2}.
\end{align}

Differentiating $J_j(t)$ with respect to time and using the system \eqref{Eq:dya}, we have
\begin{align*}
\frac{\rm d}{{\rm d}t}J_j=& \,\eta\epsilon\rho_\infty^{-1}\langle \A w_j,\nabla\sigma_j\rangle
-\eta\langle \nabla\mathcal K\sigma_j,\nabla\sigma_j\rangle
+\eta\eps\langle \mathcal G_j,\nabla\sigma_j\rangle\\
&\,-\eta\rho_\infty\langle w_j,\nabla\text{div} w_j\rangle-\eta\eps\langle w_j,\nabla\text{div} H_j\rangle.
\end{align*}
Furthermore, integration by parts and H\"{o}lder's inequality yield
\begin{align}\label{Jj}
&\frac{\rm d}{{\rm d}t}J_j+\eta\|\mathcal K^{1/2}\nabla\sigma_j\|_{L^2}^2
-\eta\rho_\infty\|\text{div} w_j\|_{L^2}^2\nonumber\\
=&\,\eta\eps
\rho_\infty^{-1}\langle \A w_j,\nabla\sigma_j\rangle
+\eta\eps\langle \mathcal G_j,\nabla\sigma_j\rangle
-\eta\eps\langle w_j,\nabla\text{div} H_j\rangle \notag\\
\le &\, C\eta\epsilon2^{2j}\|w_j\|_{L^2}\|\mathcal K^{1/2}\sigma_j\|_{L^2}+C\eta\eps 2^j
\|\mathcal G_j\|_{L^2}\|\sigma_j\|_{L^2}+C2^j\|\mathcal K^{1/2}H_j\|_{L^2}\|w_j\|_{L^2}.
\end{align}
Thus, we get from \eqref{Ej} and \eqref{Jj} that for sufficiently small $\eta$, 
\begin{align*}
\frac{{\rm d}}{{\rm d}t}E_j+c\left(\|\nabla w_j\|_{L^2}^2+\|\mathcal K^{1/2}\nabla\sigma_j\|_{L^2}^2\right)
\le C2^j\|\mathcal K^{1/2}H_j\|_{L^2}E_j^{1/2}+C\|\mathcal G_j\|_{L^2}E_j^{1/2}.
\end{align*}
This completes the proof of \eqref{dyaest}.
\end{proof}

It follows from Lemmas \ref{Hest}, \ref{Gest} and \ref{Lem:dyaest} that
\begin{equation}\label{LowFeqest}
\sup_{t\ge 0}\|(\K^{1/2}\sigma_\eps,w_\eps)(t)\|_{\dot B^{1/2}_{2,\infty}}
\le
C\|(\K^{1/2}\sigma_0,w_0)\|_{\dot B^{1/2}_{2,\infty}}.
\end{equation}
Then, combining \eqref{HighFeqest} with \eqref{LowFeqest}, we get the uniform estimate \eqref{MainEng-intro} in Theorem \ref{NNSKThm-intro}.

With the high-order dissipative energy inequalities and low-frequency dyadic estimates in hand, we are ready to derive the decay estimate stated in Theorem \ref{NNSKThm-intro}.
Let
\begin{equation*}
U_\eps:=(\K^{1/2}\sigma_\eps,w_\eps).
\end{equation*}
For any integer \(1\leq m\leq N\), 
the homogeneous energy of order $m$ satisfies
\begin{align*}
{
\mathscr{E}_m(t) := \sum_{|\alpha|=m} E_\alpha^{h,0}(t)+J_\alpha^h(t)
   \sim\|U_\eps(t)\|_{\dot H^m}^2.}
\end{align*}
Repeating the high-order energy calculation at the homogeneous level \(\dot H^m\) gives
\begin{equation}\label{Eq:HomEnergyLevel}
\frac{{\rm d}}{{\rm d}t}\mathscr E_m(t)
+c\|\nabla U_\eps(t)\|_{\dot H^m}^2\leq0,
\end{equation}
where the constant $c$ is independent of $\eps$.
Let
\begin{align*}
    A:=\sup_{t\geq0}\|U_\eps(t)\|_{\dot B^{1/2}_{2,\infty}}.
\end{align*}
For any integer $J$, splitting the dyadic sum gives
\begin{align}
    \|U_\eps\|^2_{\dot H^m}
    &=\sum_{j\leq J} 2^{2jm}\|\dot\Delta_j U_\eps\|_{L^2}^2
    +\sum_{j> J} 2^{2jm}\|\dot\Delta_j U_\eps\|_{L^2}^2\nonumber\\
    &\leq C\bigg( \sum_{j\leq J} A^2 2^{j(2m-1)}+2^{-2J}\sum_{j> J} 2^{2jm}\|\dot\Delta_j \nabla U_\eps\|_{L^2}^2\bigg)\nonumber\\
    &\leq C\left(A^2 2^{J(2m-1)}+2^{-2J}\|\nabla U_\eps\|^2_{\dot H^m}\right).\label{4.23.1}
\end{align}
Choosing $J$ such that the last two terms in \eqref{4.23.1} are equal, we obtain
\begin{equation}
\|U_\eps\|_{\dot H^m}
\leq C A^{1/(m+1/2)}
\|\nabla U_\eps\|_{\dot H^m}^{(m-1/2)/(m+1/2)}.\label{4.24.1}
\end{equation}
Inserting \eqref{4.24.1} into \eqref{Eq:HomEnergyLevel} yields a Bernoulli-type ordinary differential inequality, which implies
\begin{align*}
    \mathscr E_m(t)\leq C_m(1+t)^{-(m-1/2)}\delta_0^2,
\end{align*}
and therefore,
\begin{equation}\label{Eq:HmDecay}
\|U_\eps(t)\|_{\dot H^m}
\leq C_m(1+t)^{-(m-1/2)/2}\delta_0.
\end{equation}
For \(1/2<s<m\), 
Proposition \ref{BesovP} (iii) gives
\begin{equation}\label{Eq:LowHighInterpolation}
\|z\|_{\dot B^s_{2,1}}
\leq C_{s,m}
\|z\|_{\dot B^{1/2}_{2,\infty}}^{\theta}
\|z\|_{\dot H^m}^{1-\theta},
\qquad
\theta=\frac{m-s}{m-1/2}.
\end{equation}
For $1/2<s<N-2$, taking an integer \(m\in(s,N-2]\), and using \eqref{LowFeqest} and \eqref{Eq:HmDecay}, we deduce
\begin{equation}\label{TimeDecay}
\|U_\eps(t)\|_{\dot B^s_{2,1}}
\leq C_s(1+t)^{-(s-1/2)/2}\delta_0.
\end{equation}

Next, we justify the global existence. Let $J_n$ denote the Fourier truncation onto the frequencies $|\xi|\leq 2^n$.
Applying $J_n$ to the system \eqref{Eq:MainPerEq}, we obtain the truncated system, which is a finite-dimensional ordinary differential equation. The estimates \eqref{HighFeqest} and \eqref{LowFeqest} hold uniformly with respect to the truncation parameter $n$ for the approximate solutions. Since $N\geq 5$, the Sobolev embedding $H^N\hookrightarrow L^\infty$ ensures the density lower bound
\begin{align*}
    \rho_\eps^*+\eps\sigma_\eps\geq\rho_\infty/2
\end{align*}
as long as the bootstrap norm is small. Thanks to the uniform a priori bounds, standard compactness arguments on each finite time interval allow us to extract a convergent subsequence and pass $n\to\infty$, yielding a strong solution on some maximal time interval $[0,T^*)$. We claim $T^*=\infty$. Suppose, for contradiction, $T^*<\infty$, then \eqref{HighFeqest} and \eqref{LowFeqest} improve the bootstrap bound from \(2C\delta_0\) to \(C\delta_0\). This improved bound enables us to extend the solution beyond $T^*$, contradicting the maximality of $T^*$. Therefore, the solution exists globally in time. 
To establish uniqueness, we apply identical energy estimates to the difference of two strong solutions, which forces the difference to vanish identically.

Finally, we conclude this subsection by establishing some estimates for the incompressible perturbation, needed for the subsequent low-Mach-number convergence analysis.
The perturbation $\widetilde{u}:=u-u^*$ satisfies the following perturbed Stokes system
\begin{equation}\label{Eq:IncompPerturbation}
\partial_t\widetilde u-\frac{\mu}{\rho_\infty}\Delta\widetilde u
=-\mathbb P\bigl(u^*\cdot\nabla\widetilde u
+\widetilde u\cdot\nabla u^*
+\widetilde u\cdot\nabla\widetilde u\bigr),
\quad \operatorname{div}\widetilde u=0.
\end{equation}
The initial data can be decomposed as
\begin{align*}
    \widetilde u(0)=\mathbb P w_{\eps,0}+(\mathbb P u_\eps^*-u^*),
\end{align*}
which is small in $\dot B^{1/2}_{2,\infty}\cap H^N$ by our assumptions.
The linear terms with stationary coefficients $u^*$ are handled by Lemma \ref{Lem:L6ab}, and the quadratic term $\widetilde u\cdot\nabla\widetilde u$  {is estimated by the same product and commutator estimates as those applied to
\(w_\eps\cdot\nabla w_\eps\) in the proof of Lemma~\ref{Gest} }. Combining the dyadic heat semigroup estimates and the low-high frequency Littlewood--Paley interpolation, we obtain the uniform bound
\begin{align}
&\sup_{t\geq0}\|\widetilde u(t)\|_{\dot B^{1/2}_{2,\infty}\cap H^N}
+\|\nabla\widetilde u\|_{L^2_tH^N}
\leq C\delta_0,
\label{Eq:IncompUniform}\\
&\|\widetilde u(t)\|_{\dot B^m_{2,1}}
\leq C_m(1+t)^{-(m-1/2)/2}\delta_0,
\quad 1/2<m<N-2.
\label{Eq:IncompDecay}
\end{align}

\subsection{Linear semigroup and dispersive Strichartz theory}\label{S4.3}
To handle the oscillatory acoustic part in the Duhamel formulation, we analyze the linearized acoustic-capillary subsystem via symmetrized variables and derive a set of dispersive Strichartz estimates required for the low-Mach-number convergence analysis.

To isolate the compressible component of the velocity, we set
\begin{equation}\label{Eq:qdDef}
    q:=\mathbb{Q}w, \quad d:=\Lambda^{-1}\dive q.
\end{equation}
Then the linear compressible system reads
\begin{equation}\label{4.32.1}
\begin{cases}
    \partial_t \sigma+\frac{\rho_\infty}{\epsilon}\Lambda d=F_1,\\
    \partial_t d-\nu_0 \Delta d-\frac{1}{\epsilon}\Lambda \mathcal{K}\sigma=F_2,
\end{cases}
\end{equation}
where $\nu_0=\frac{(\mu+\nu)}{\rho_\infty}$. Taking the Fourier transform of \eqref{4.32.1} and denoting 
\begin{align*}
    r:=|\xi|,\quad k(r):=\gamma_0+\kappa r^2,
\end{align*}
we deduce the characteristic equation for the homogeneous system of \eqref{4.32.1}:
\begin{equation*}
    \lambda^2+\nu_0 r^2\lambda +\frac{\rho_\infty r^2}{\eps^2}(\gamma_0+\kappa r^2)=0.
\end{equation*}
Take $\eps_0>0$ small enough, such that $l_{\eps_0}\geq \frac{1}{2}\rho_\infty\kappa$, where
\begin{align}\label{l-eps}
    l_\eps:=\rho_\infty\kappa-\frac{\eps^2 \nu_0^2}{4}. 
\end{align}
When $0<\eps\leq \eps_0$, the characteristic roots are
\begin{equation}\label{characteristic roots}
    \lambda_{\pm}(r)=-\frac{\nu_0 r^2}{2}\pm i\Omega_\eps(r),
\end{equation}
where
\begin{align*}
   \Omega_\eps(r)=\frac{r}{\eps}h_\eps(r),\quad h_\eps(r)=(\rho_\infty \gamma_0+l_\eps r^2)^{\frac{1}{2}}.
\end{align*}
Moreover,
\begin{align*}
    h_\eps(r)\sim  \langle r\rangle,\quad \Omega_\eps(r)\sim \left\{
    \begin{array}{cc}
         \frac{r}{\eps},&  0<r\lesssim 1,\\
         \frac{r^2}{\eps},& r\gtrsim 1.
    \end{array}\right.
\end{align*}
Since
\begin{align*}
    \left|\operatorname{Im}\frac{\eps\lambda_\pm(r)}r\right|
=h_\eps(r)\sim\langle r\rangle,
\end{align*}
the spectral projections are not zero-order multipliers in the original variables $(\sigma,d)$. To overcome this difficulty, we introduce the symmetrized variable
\begin{align*}
    V:=\left(\begin{array}{c}
         \mathcal K^{\frac{1}{2}}\sigma \\
         \sqrt{\rho_\infty}d 
    \end{array}\right),
\end{align*}
then the homogeneous system becomes
\begin{align}\label{Eq:SymmetricMatrix}
    \partial_t \hat{V}=A_\eps(\xi)\hat{V}, \quad A_\eps(\xi)=
    \left(\begin{array}{cc}         
    0              & - \frac{r\sqrt{\rho_\infty k(r)}}{\eps}\\         
    \frac{r\sqrt{\rho_\infty k(r)}}{\eps} & -\nu_0 r^2    
    \end{array}\right).
\end{align}
The eigenvalues of $A_\eps(\xi)$ are $\lambda_{\pm}(r)$ given by \eqref{characteristic roots}.

Let $\Pi_\pm(\xi)$ be the spectral projections of $A_\eps(\xi)$. Then $\Pi_\pm(\xi)$ satisfy the following uniform Mikhlin bound.
\begin{lem}\label{Lem:UniformProjection}
For any multi-index $\alpha$, there exists a constant $C_\alpha>0$ independent of $\eps\in(0,\eps_0]$, such that
    \begin{align}
        |\partial_\xi^\alpha \Pi_{\pm}(\xi)|\leq C_\alpha |\xi|^{-|\alpha|}, \quad \xi\neq 0.\label{4.36.1}
    \end{align}
\end{lem}
\begin{proof}
Since $\lambda_+(r)-\lambda_-(r)=2i\Omega_\eps$, we have
\begin{equation*}
\Pi_\pm(r)=\frac{\pm(\lambda_{\pm}(r){\mathbb{I}_2}-A_\eps)}{\lambda_+(r)-\lambda_-(r)}=\frac12{\mathbb{I}_2}\pm\frac{A_\eps+\frac12\nu_0r^2 {\mathbb{I}_2}}{2i\Omega_\eps}
=\frac12{\mathbb{I}_2}\pm\frac1{2i}
\begin{pmatrix}
\alpha_\eps(r)&-\beta_\eps(r)\\
\beta_\eps(r)&-\alpha_\eps(r)
\end{pmatrix},
\end{equation*}
where $\mathbb{I}_2$ denotes the second-order identity matrix, and 
\begin{align*}
\alpha_\eps(r)=\frac{\eps\nu_0 r}{2h_\eps(r)},\quad
\beta_\eps(r)=\frac{\sqrt{\rho_\infty k(r)}}{h_\eps(r)}.
\end{align*}
With the help of \eqref{l-eps}, we have
\begin{align*}
    \rho_\infty \kappa\geq l_\eps\geq \frac{1}{2}\rho_\infty \kappa>0.
\end{align*}
Direct computations yield
\begin{align*}
    r\partial_r \alpha_\eps(r)=\frac{l_\eps r^2}{\rho_\infty\gamma_0+l_\eps r^2}\alpha_\eps(r),\quad
    r\partial_r \beta_\eps(r)=\left(\frac{\rho_\infty\gamma_0}{\rho_\infty\gamma_0+l_\eps r^2}-\frac{\gamma_0}{\gamma_0+\kappa r^2}\right)\beta_\eps(r),
\end{align*}
which, together with the principle of mathematical induction, implies that
for every integer \(k\geq0\),
\begin{align*}
    \sup_{r>0}|(r\partial_r)^k\alpha_\eps(r)|
    +\sup_{r>0}|(r\partial_r)^k\beta_\eps(r)|\leq C_k.
\end{align*}
Then, applying the derivative formula for radial functions, we obtain 
\begin{align*}
    |\partial_\xi^\alpha \Pi_{\pm}(\xi)|\leq |\xi|^{-|\alpha|}\sum_{k=0}^{|\alpha|}|c_{\alpha,k}|(|(r\partial_r)^k\alpha_\eps(r)|+|(r\partial_r)^k\beta_\eps(r)|) \leq C_\alpha |\xi|^{-|\alpha|},
\end{align*}
which is the desired estimate \eqref{4.36.1}. Thus, we complete the proof of this lemma.
\end{proof}
Let $S_\eps(t)$ denote the semigroup generated by \eqref{Eq:SymmetricMatrix}. The explicit expression of $S_\eps(t)$ is
\begin{equation*}
S_\eps(t)=e^{-\nu_0 r^2t/2}
\bigl(e^{it\Omega_\eps(r)}\Pi_+(r)
+e^{-it\Omega_\eps(r)}\Pi_-(r)\bigr).
\end{equation*} 
We fix a dividing point between low- and high-frequency. Since changing finitely many blocks only alters the constants and does not affect the essential nature of estimates, we write \(j\leq0\) for the low-frequency blocks and \(j\geq1\) for the high-frequency blocks. Under this classification, we obtain the following dispersive estimates.
\begin{lem}[Dyadic dispersive estimates]\label{Lem:DyadicDispersion}
Let \(2\leq p\leq\infty\), $\frac{1}{p}+\frac{1}{p^\prime}=1$, and
\(\delta_p:=\frac12-\frac1p\).  
Then, it holds that for any function $Z\in L^{p^\prime}$, 
\begin{align}
\|\dot\Delta_jS_\eps(t)Z\|_{L^p}
&\leq Ce^{-c2^{2j}t}2^{4j\delta_p}
\left(\frac\eps t\right)^{2\delta_p}
\|\dot\Delta_jZ\|_{L^{p'}},
&&j\leq0,
\label{Eq:LowDispersion}\\
\|\dot\Delta_jS_\eps(t)Z\|_{L^p}
&\leq Ce^{-c2^{2j}t}
\left(\frac\eps t\right)^{3\delta_p}
\|\dot\Delta_jZ\|_{L^{p'}},\label{Eq:HighDispersion}
&&j\geq1.
\end{align}
 Moreover, for $j\in\mathbb Z$ and any function $Z\in L^{2}$, we have
\begin{align}
\|\dot\Delta_jS_\eps(t)Z\|_{L^p}
\leq Ce^{-c2^{2j}t}2^{3j\delta_p}
\|\dot\Delta_jZ\|_{L^2}.
\label{Eq:EnergyBernstein}
\end{align}
\end{lem}
\begin{proof}
 {By Lemma  \ref{Lem:UniformProjection}, the spectral projectors are uniformly bounded
zero-order Fourier multipliers.}
Utilizing spherical coordinates $\xi=r\omega$, the kernel is reduced to
    \begin{align*}
        K(t,x)&=\int_{\mathbb{R}^3} e^{ix\cdot \xi}e^{-\frac{1}{2}\nu_0r^2t} e^{\pm it\Omega_\eps(r)}\Pi_{\pm}(r) {\rm d}\xi\\
        &=\int_0^\infty r^2 e^{-\frac{1}{2}\nu_0r^2t} e^{\pm it\Omega_\eps(r)}\Pi_{\pm}(r) \left(\iint_{\mathbb S^2}e^{ir\omega\cdot x}{\rm d}\sigma(\omega)\right) {\rm d}r\\
        &=\int_0^\infty r^2 e^{-\frac{1}{2}\nu_0r^2t} e^{\pm it\Omega_\eps(r)}\Pi_{\pm}(r) \frac{4\pi\sin(r|x|)}{r|x|} {\rm d}r\\
        &=\frac{4\pi}{|x|}\int_0^\infty r e^{-\frac{1}{2}\nu_0r^2t} e^{\pm it\Omega_\eps(r)}\Pi_{\pm}(r) \frac{e^{ir|x|}-e^{-ir|x|}}{2i} {\rm d}r,
    \end{align*}
    which is a sum of one-dimensional integrals with phases $t\Omega_\eps(r)\pm r|x|$ and a prefactor $|x|^{-1}$. 
    
    For $j\leq 0$ and $2^{j-1}\leq r\leq 2^{j+1}$, we have
    \begin{align*}
        \partial_r(t\Omega_\eps)=\frac{t}{\eps}(h_\eps(r)+rh_\eps'(r))\sim \frac{t}{\eps}.
    \end{align*}
     If $|x|$ is not comparable with ${t}/{\eps}$, the phase derivative is bounded from below and integration by parts gives rapid decay.  In the remaining region \(|x|\sim t/\eps\), the prefactor is \(O(\eps/t)\), while the radial amplitude has integral \(O(2^{2j})\). Thus
\begin{equation}\label{Eq:LowKernel}
\|K_j^\ell(t,\cdot)\|_{L^\infty}
\leq Ce^{-c2^{2j}t}2^{2j}\frac\eps t.
\end{equation}
When \(t2^j/\eps>1\), the right-hand side of \eqref{Eq:LowKernel} is bounded by $Ce^{-c2^{2j}t} 2^{3j}$, improving the trivial kernel bound $C2^{3j}$. When \(t2^j/\eps\leq1\), we have $2^{2j}\eps/t\geq 2^{3j}$, so the trivial kernel bound $C2^{3j}$ is sharper.

For $j\geq 1$, rescale $\xi=2^j \eta$, then $\frac{1}{2}\leq|\eta|\leq 2$, and
\begin{align*}
    t\Omega_\eps(|\xi|)=t\Omega_\eps(2^j|\eta|)
    =\frac{t 2^{2j}}{\eps}|\eta|(l_\eps|\eta|^2+\rho_\infty \gamma_0 2^{-2j})^{\frac{1}{2}}.
\end{align*}
Define the oscillatory parameter $\tau=t2^{2j}/\eps$ and the rescaled phase
\begin{align*}
    \phi_{\eps,j}(\eta)=|\eta|\bigl(l_\eps|\eta|^2+\rho_\infty\gamma_0 2^{-2j}\bigr)^{\frac{1}{2}}.
\end{align*}
The Hessian of $\phi_{\eps,j}(\eta)$ is uniformly nondegenerate on the fixed annulus $\frac{1}{2}\leq|\eta|\leq 2$. Stationary phase in three dimensions gives
\begin{equation}\label{Eq:HighKernel}
\|K_j^h(t,\cdot)\|_{L^\infty}
\leq C2^{3j}e^{-c2^{2j}t} \left(\frac{t 2^{2j}}{\eps}\right)^{-\frac{3}{2}}
\leq Ce^{-c2^{2j}t}\left(\frac{\eps}{t}\right)^{\frac{3}{2}}.
\end{equation}
The same dichotomy at $\tau=t 2^{2j}/\eps=1$ applies: 
for $\tau> 1$, \eqref{Eq:HighKernel} is the sharper bound; while for \(\tau\leq1\), we use the trivial estimate $C2^{3j}$.

Consequently, interpolating the \(L^1\to L^\infty\) estimates \eqref{Eq:LowKernel}--\eqref{Eq:HighKernel} with the \(L^2\to L^2\) estimate gives \eqref{Eq:LowDispersion}--\eqref{Eq:HighDispersion}.  
Furthermore, by virtue of Bernstein's inequality and  \eqref{Eq:LowDispersion}--\eqref{Eq:HighDispersion} with $p=2$, we have
\begin{align*}
    \|\dot\Delta_jS_\eps(t)Z\|_{L^p}
    \leq C2^{3j\delta_p} \|\dot\Delta_jS_\eps(t)Z\|_{L^2}
    \leq Ce^{-c2^{2j}t}2^{3j\delta_p}\|\dot\Delta_jZ\|_{L^2},
\end{align*}
which implies \eqref{Eq:EnergyBernstein}.  Thus, we complete the proof of this lemma.
\end{proof}

Define
\begin{equation*}
\beta(p,r):=\min\left\{\frac1r,\delta_p\right\},
\quad
a_{p,r}^{\rm w}:=3\delta_p-\frac1r-
\left(\frac1r-\delta_p\right)_+,
\quad
a_{p,r}^{\rm s}:=3\delta_p-\frac2r.
\end{equation*}
Then we have the following space-time estimates.
\begin{lem}[Dyadic space-time estimates]\label{Lem:DyadicStrichartz}
Let \(2<p<6\) and \(2<r<\infty\).  Then, it holds that for any function $Z\in L^2$,
\begin{align}
\|\dot\Delta_jS_\eps(t)Z\|_{L^r_tL^p}
&\leq C\eps^{\beta(p,r)}2^{ja_{p,r}^{\rm w}}
\|\dot\Delta_jZ\|_{L^2},
&&j\leq0,
\label{Eq:LowBlockStrichartz}\\
\|\dot\Delta_jS_\eps(t)Z\|_{L^r_tL^p}
&\leq C\eps^{\beta(p,r)}2^{ja_{p,r}^{\rm s}}
\|\dot\Delta_jZ\|_{L^2},
&&j\geq1.
\label{Eq:HighBlockStrichartz}
\end{align}
\end{lem}
\begin{proof}
We first treat the low-frequency blocks $j\leq 0$. 
Define the linear propagator on a single dyadic block: for $Z\in L^2(\mathbb{R}^3)$, define
\begin{align*}
    T_j Z:=\dot\Delta_jS_\eps(t)Z.
\end{align*}
We shall prove that $T_j:L^2\to L_t^{r_{\rm w}}L^p$ is bounded at the wave endpoint $r_{\rm w}=1/\delta_p$.  
Accordingly, its adjoint operator $T_j^*:L_t^{r'_{\rm w}}L^{p'}\to L^2$. By \eqref{Eq:LowDispersion}, the composition $T_jT_j^*$ has a time-convolution kernel satisfying
 \begin{align*}
     \|K_j(t-s)\|_{L^{p'}\to L^p}
\leq C\eps^{2\delta_p}2^{4j\delta_p}|t-s|^{-2\delta_p}.
 \end{align*}
Extending functions by zero outside \((0,\infty)\), the
one-dimensional Hardy--Littlewood--Sobolev inequality gives
\begin{align*}
    \|T_jT_j^*\|_{L^{r_w'}_tL^{p'}\to L^{r_w}_tL^p}
\leq C\eps^{2\delta_p}2^{4j\delta_p}.
\end{align*}
The identity
\begin{align*}
    \|T_j\|^2_{L^2\to L_t^{r_{\rm w}}L^{p}}=
    \|T_j T_j^*\|_{L_t^{r'_{\rm w}}L^{p'}\to L_t^{r_{\rm w}}L^{p}}
\end{align*}
yields
\begin{equation}\label{Eq:WaveEndpoint}
\|\dot\Delta_jS_\eps(t)Z\|_{L_t^{r_{\rm w}}L^p}
\leq C\eps^{\delta_p}2^{2j\delta_p}
\|\dot\Delta_jZ\|_{L^2}.
\end{equation}
If \(r\geq r_{\rm w}\), interpolation with \eqref{Eq:EnergyBernstein} at \(L_t^\infty L^p\) gives
\begin{align}
    \|\dot\Delta_jS_\eps(t)Z\|_{L_t^{r}L^p}\leq C\eps^{\frac 1 r}2^{j(3\delta_p-\frac{1}{r})}\|\dot\Delta_jZ\|_{L^2}.\label{4.45.1}
\end{align}
If \(r<r_{\rm w}\), let $1/q=1/r-1/r_{\rm w}>0$, by \eqref{Eq:WaveEndpoint}, it holds that
\begin{align}    
\|\dot\Delta_jS_\eps(t)Z\|_{L_t^{r}L^p}
&=\Big\|e^{-\frac{c}{2} 2^{2j}t}  e^{-\frac{c}{2} 2^{2j}t}\frac{{S_\eps(t)}\dot\Delta_jZ}{e^{-c 2^{2j}t}}\Big\|_{L_t^{r}L^p}
\leq \|e^{-\frac{c}{2} 2^{2j}t}\|_{L_t^q}
\|{{S_\eps(t)}\dot\Delta_jZ}\|_{L_t^{r_{\rm w}}L^p}\nonumber\\
&\leq C 2^{-\frac{2j}{q}} \|S_\eps(t)\dot\Delta_jZ\|_{L_t^{r_{\rm w}}L^p}
\leq C 2^{-\frac{2j}{q}} \eps^{\delta_p}2^{2j\delta_p}
\|\dot\Delta_jZ\|_{L^2}\nonumber\\
&\leq C \eps^{\delta_p}2^{j(4\delta_p-\frac{2}{r})}\|\dot\Delta_jZ\|_{L^2}.\label{4.46.1}
\end{align}
It follows from \eqref{4.45.1}--\eqref{4.46.1} that \eqref{Eq:LowBlockStrichartz} holds.

For high-frequency region, the Schr\"odinger endpoint \(r_{\rm s}\) is determined by
\(2/r_{\rm s}=3\delta_p\).  The same \(TT^*\) argument applied to \eqref{Eq:HighDispersion} gives
\begin{align}\label{Eq:SchrodingerEndpoint}
    \|\dot\Delta_jS_\eps(t)Z\|_{L_t^{r_{\rm s}}L^p}
    \leq C\eps^{\frac{1}{r_{\rm s}}}\|\dot\Delta_jZ\|_{L^2}.
\end{align}
If \(r\geq r_{\rm s}\), interpolation with \eqref{Eq:EnergyBernstein} at \(L_t^\infty L^p\) yields
\begin{align*}    
\|\dot\Delta_jS_\eps(t)Z\|_{L_t^{r}L^p}
\leq C\eps^{\frac 1 r}2^{j(3\delta_p-\frac{2}{r})}\|\dot\Delta_jZ\|_{L^2}.
\end{align*}
If \(r< r_{\rm s}\), let $1/\tilde q=1/r-1/r_{\rm s}>0$, by virtue of \eqref{Eq:SchrodingerEndpoint} and a similar heat-factor argument as above, we obtain
\begin{align*}    
\|\dot\Delta_jS_\eps(t)Z\|_{L_t^{r}L^p}
\leq C \eps^{\frac{3\delta_p}{2}}2^{j(3\delta_p-\frac{2}{r})}\|\dot\Delta_jZ\|_{L^2}.
\end{align*}
Replacing any stronger power of \(\eps\) by
\(\eps^{\beta(p,r)}\) proves \eqref{Eq:HighBlockStrichartz}. Thus, we complete the proof of this lemma.
\end{proof}

With the low-high frequency splitting in hand, we are now ready to establish  Strichartz-type bounds for the linear semigroup $S_\eps(t)$. The homogeneous and inhomogeneous Strichartz estimates are given in the following proposition.

\begin{prop}[Homogeneous and inhomogeneous Strichartz estimates]\label{Prop:AcousticStrichartz}
Assume
\begin{equation}\label{Eq:IndexRange}
2<p<6,
\qquad 2<r<\infty,
\qquad \frac12+\frac2r<s<\frac3p.
\end{equation}
If \(M>s+a_{p,r}^{\rm s}\), then
\begin{equation}\label{Eq:HomogeneousStrichartz}
\|S_\eps(t)Z_0\|_{L^r_t\dot B^s_{p,1}}
\leq C\eps^{\beta(p,r)}
\bigl(\|Z_0^\ell\|_{\dot B^{1/2}_{2,\infty}}
+\|Z_0^h\|_{H^M}\bigr),
\end{equation}
 for any function $Z_0\in \dot B^{\frac{1}{2}}_{2,\infty}\cap H^M$, and
\begin{align}
\left\|\int_0^tS_\eps(t-\tau)Z(\tau)\,{\rm d}\tau\right\|_{L^r_t\dot B^s_{p,1}}
\leq C\eps^{\beta(p,r)}
\bigl(\|Z^\ell\|_{L^1_t\dot B^{1/2}_{2,\infty}}
+\|Z^h\|_{L^1_tH^M}\bigr),
\label{Eq:InhomogeneousStrichartz}
\end{align}
for any function $Z\in L^1(0,\infty;\dot B^{\frac{1}{2}}_{2,\infty}\cap H^M)$.
\end{prop}
\begin{proof}
For the low-frequency part of $Z_0$, by \eqref{Eq:LowBlockStrichartz}, we have
\begin{align*}
        \|S_\eps(t)Z_0^\ell\|_{L^r_t\dot B_{p,1}^s}
        &\leq \sum_{j\leq 0} \|S_\eps(t) \dot\Delta_jZ_0^\ell\|_{L^r_t\dot B_{p,1}^s}\\
        &\leq C\eps^{\beta(p,r)} \sum_{j\leq 0}2^{j(s+a_{p,r}^{\rm w})}\|\dot\Delta_jZ_0^\ell\|_{L^2}\\
        &\leq C\eps^{\beta(p,r)} \sum_{j\leq 0}2^{j(s+a_{p,r}^{\rm w}-1/2)}\|Z_0^\ell\|_{\dot B^{1/2}_{2,\infty}}.
\end{align*}
By virtue of \eqref{Eq:IndexRange}, we have
\[
s+a_{p,r}^{\rm w}-\frac{1}{2} >
\begin{cases}
\dfrac{1}{r}+3\delta_p, & \text{if } \dfrac{1}{r} \le \delta_p,\\[6pt]
4\delta_p, & \text{if } \dfrac{1}{r} > \delta_p.
\end{cases}
\]
Consequently, \(s+a_{p,r}^{\rm w}-1/2 > 0\) in both scenarios, and thus the geometric series
\[
\sum_{j\le 0} 2^{j\left(s+a_{p,r}^{\rm w}-\frac{1}{2}\right)}
\]
converges.

For the high-frequency part of $Z_0$, by virtue of \eqref{Eq:HighBlockStrichartz} and Cauchy--Schwarz inequality, we have
    \begin{align*}        
    \|S_\eps(t)Z_0^h\|_{L^r_t\dot B_{p,1}^s}        
    &\leq \sum_{j> 0} \|S_\eps(t) \dot\Delta_jZ_0^h\|_{L^r_t\dot B_{p,1}^s}\\        
    &\leq C\eps^{\beta(p,r)} \sum_{j\leq 0}2^{j(s+a_{p,r}^{\rm s})}\|\dot\Delta_jZ_0^h\|_{L^2}\\        
    &\leq C\eps^{\beta(p,r)} (\sum_{j\leq 0}2^{2j(s+a_{p,r}^{\rm s}-M)})^{\frac{1}{2}}\|Z_0^h\|_{H^M}.    
    \end{align*}
Obviously, since \(M>s+a_{p,r}^{\rm s}\), the geometric series $$\sum_{j\leq 0}2^{2j(s+a_{p,r}^{\rm s}-M)}$$ 
converges. Based on the above analysis for the low- and high-frequency parts of $Z_0$, we obtain \eqref{Eq:HomogeneousStrichartz}.  
    
In addition, Minkowski's inequality and time-translation invariance give  
    \begin{align*}
        \left\|\int_0^tS_\eps(t-\tau)Z(\tau)\,{\rm d}\tau\right\|_{L^r_t\dot B^s_{p,1}}
        &\leq \int_0^\infty \|S_\eps(t-\tau)Z(\tau)\|_{L^r_t\dot B^s_{p,1}}{\rm d}\tau\\
        &\leq C\eps^{\beta(p,r)}\bigl(\|Z^\ell\|_{L^1_t\dot B^{1/2}_{2,\infty}}+\|Z^h\|_{L^1_tH^M}\bigr),
    \end{align*}
    which yields \eqref{Eq:InhomogeneousStrichartz}. Thus, we complete the proof of this proposition.
\end{proof}

Since the high-frequency stationary-coefficient source does not belong to $L^1_t H^M$, we employ the following damped estimate instead. Although the constant \(C\) is independent of \(\eps\), the estimate contains the explicit loss \(\eps^{-3\delta_p}\). This loss is
essential for treating the high-frequency stationary-coefficient
source, which is not integrable in time.
\begin{lem}[High-frequency damped estimate]\label{Lem:HighDampedEstimate}
Let \(2\leq p<\infty\), \(1\leq r\leq\infty\), and \(s\in\mathbb R\).  Then, for any function $Z\in L^r_t\dot B^{s-2}_{p,1}$, we have
\begin{equation}\label{Eq:HighDampedEstimate}
\left\|\int_0^tS_\eps(t-\tau)Z^h(\tau)\,d\tau\right\|_{L^r_t\dot B^s_{p,1}}
\leq C\eps^{-3\delta_p}
\|Z^h\|_{L^r_t\dot B^{s-2}_{p,1}}.
\end{equation}
\end{lem}
\begin{proof}
For $j\geq 1$, we proceed analogously to the proof of Lemma \ref{Lem:DyadicDispersion}. Under the scaling $\xi=2^j\eta$, we obtain an oscillatory parameter \(\tau=t2^{2j}/\eps\) and a family of compactly supported phases
\begin{align*}    
\phi_{\eps,j}(\eta)=|\eta|\bigl(l_\eps|\eta|^2+\rho_\infty\gamma_0 2^{-2j}\bigr)^{\frac{1}{2}}
\end{align*}
with uniformly nondegenerate Hessians. Moreover,
    \begin{align*}
        K_j^h(t,x)&=\int_{\mathbb{R}^3} e^{ix\cdot \xi} e^{-\frac{1}{2}\nu_0 |\xi|^2 t}\big(e^{it\Omega_\eps(|\xi|)}\Pi_{+} +e^{-it\Omega_\eps(|\xi|)}\Pi_{-}\big)\varphi(2^{-j}\xi) {\rm d}\xi\\
        &=2^{3j}\int_{\mathbb{R}^3} e^{i2^jx\cdot \eta} e^{-\frac{1}{2}\nu_0 |\eta|^2 \eps\tau}\big(e^{i\tau\phi_{\eps,j}(\eta)}\Pi_{+} +e^{-i\tau\phi_{\eps,j}(\eta)}\Pi_{-}\big)\varphi(\eta) {\rm d}\eta\\
        &=2^{3j}e^{-\frac{1}{2}\nu_0 |\eta|^2 \eps\tau}\big(\mathcal{K}_j^+(\tau,2^j x)+\mathcal{K}_j^-(\tau,2^j x)\big),
    \end{align*}
    where 
    \begin{align*}
        \mathcal{K}_j^{\pm}(\tau,y):=\int_{\mathbb{R}^3} e^{i(y\cdot \eta\pm \tau\phi_{\eps,j}(\eta))}a^\pm_{\eps,j}(\eta)d\eta.
    \end{align*}
    The amplitudes $a^\pm_{\eps,j}(\eta)$ are supported in the annulus $\frac{1}{2}\leq |\eta|\leq 2$ and are uniformly bounded in $C^M$  {with respect to both $\eps$ and $j$} for any fixed $M\geq 0$.
    
    We first estimate the spatial $L^1$ norm of the normalized kernel $\mathcal{K}_j^{\pm}(\tau,\cdot)$. 
    If $0\leq\tau\leq1$, then integration by parts gives
    \begin{align}
        (1+|y|^2)^M \mathcal{K}_j^{\pm}(\tau,y)=\int_{\mathbb{R}^3}(1-\Delta_\eta)^M(e^{\pm i\tau\phi_{\eps,j}(\eta)}a^\pm_{\eps,j}(\eta)){\rm d}\eta,\label{4.52.1}
    \end{align}
    for any integer $M\geq 1$. Since $0\leq \tau\leq 1$, the integrand on the right-hand side of \eqref{4.52.1} is uniformly bounded in $L^1_\eta$  {with respect to both $\eps$ and $j$}, which implies
    \begin{align*}
        |\mathcal{K}_j^{\pm}(\tau,y)|\leq C_M(1+|y|^2)^{-M}.
    \end{align*}
    Taking $M>3/2$, we obtain
    \begin{align*}
        \sup_{0\leq \tau\leq 1}\|\mathcal{K}_j^{\pm}(\tau,\cdot)\|_{L^1_y}\leq C.
    \end{align*}
    
    For $\tau\geq 1$, since $\nabla_\eta (y\cdot \eta\pm \tau\phi_{\eps,j}(\eta))=0$ is equivalent to $y=\mp \tau\nabla_\eta\phi_{\eps,j}(\eta)$, the stationary phase gives a pointwise size \(C\tau^{-3/2}\) in a region of volume \(C\tau^3\). Outside this region, repeated integration by parts yields rapid decay.  Hence, the normalized kernel has \(L^1\) norm at most \(C\tau^{3/2}\).
  Thus, for any $\tau\geq 0$, we have
    \begin{align}\label{4.53.1}
        \|\mathcal{K}_j^{\pm}(\tau,\cdot)\|_{L^1_y}\leq C(1+\tau)^{\frac{3}{2}}.
    \end{align}

Applying \eqref{4.53.1} and Young's convolution inequality in space, we get
    \begin{align*}
        \|\dot\Delta_jS_\eps(t)Z\|_{L^\infty}    \leq Ce^{-c2^{2j}t}    \left(1+\frac{t2^{2j}}\eps\right)^{\frac{3}{2}}    \|\dot\Delta_jZ\|_{L^\infty}.
    \end{align*}
Interpolating with the \(L^2\) operator estimate yields
    \begin{equation*}
    \|\dot\Delta_jS_\eps(t)Z\|_{L^p}
    \leq Ce^{-c2^{2j}t}
    \left(1+\frac{t2^{2j}}\eps\right)^{3\delta_p}
    \|\dot\Delta_jZ\|_{L^p}.
    \end{equation*}
In addition, observe that
    \begin{align*}
        \int_0^\infty e^{-c2^{2j}t}    \left(1+\frac{t2^{2j}}\eps\right)^{3\delta_p} {\rm d}t
        =2^{-2j}\int_0^\infty e^{-cs}    \left(1+\frac{s}\eps\right)^{3\delta_p} {\rm d}s
        \leq C2^{-2j} \eps^{-3\delta_p}.
    \end{align*}
Combining the above three estimates with Young's convolution inequality in time and summing over \(j\geq1\),  we arrive at    \begin{align*}
        \left\|\int_0^t S_\eps(t-\tau)Z^h(\tau){\rm d}\tau\right\|_{L^r_t\dot B^s_{p,1}}
        \leq C \eps^{-3\delta_p}\sum_{j\geq 1}2^{j(s-2)}\|\dot \Delta_j Z^h\|_{L^r_t L^p_x}
        =C \eps^{-3\delta_p}\|Z^h\|_{L^r_t\dot B^{s-2}_{p,1}},
    \end{align*}
    which implies \eqref{Eq:HighDampedEstimate}. The proof of this lemma is complete.  
\end{proof}

\subsection{Low Mach number  convergence analysis}\label{S4.4}
{With the uniform-in-$\epsilon$ estimates and the dispersive Strichartz estimates for the acoustic-capillary semigroup in hand, we now pay attention to the asymptotic behaviour of the solution $(\rho_\epsilon,u_\epsilon)$ as the Mach number \(\epsilon\) tends to zero.
We establish the convergence rates for both the compressible acoustic component and the incompressible projected part, which constitute the main conclusions of Theorem~\ref{NNSKThm-intro}.}

We first analyze the convergence of the compressible acoustic component.
Applying \(\mathbb Q\) to \eqref{Eq:MainConCoeffEq} and replacing $(w,d)$ in \eqref{Eq:qdDef} by $(w_\epsilon,d_\epsilon)$,
we obtain
\begin{equation*}
\left\{
\begin{aligned}
&\partial_t\sigma_\eps+\frac{\rho_\infty}{\eps}\Lambda d_\eps=f_\eps,\\
&\partial_td_\eps-\nu_0\Delta d_\eps
-\frac1\eps\Lambda\K\sigma_\eps
=\Lambda^{-1}\operatorname{div}\mathbb Q\mathcal G,
\end{aligned}
\right.
\end{equation*}
where
\begin{align*}
f_\eps:=-\operatorname{div}H.    
\end{align*}
Set
\begin{equation}\label{Eq:AcousticSymmetricUnknown}
V_\eps:=\begin{pmatrix}\K^{1/2}\sigma_\eps\\[1mm]
\sqrt{\rho_\infty}\,d_\eps\end{pmatrix},
\qquad
\Phi_\eps:=\begin{pmatrix}\K^{1/2}f_\eps\\[1mm]
\sqrt{\rho_\infty}\Lambda^{-1}\operatorname{div}\mathbb Q\mathcal G
\end{pmatrix}.
\end{equation}
Then
\begin{equation}\label{Eq:AcousticDuhamel}
V_\eps(t)=S_\eps(t)V_{\eps,0}
+\int_0^tS_\eps(t-\tau)\Phi_\eps(\tau)\,{\rm d}\tau.
\end{equation}
Based on \eqref{MainEng-intro}, \eqref{Eq:HmDecay}, and \eqref{TimeDecay}, we first establish the following time-integrability results.
\begin{lem}\label{Lem:TimeIntegrability}
For \(N\geq 5\), we have
\begin{equation}\label{Eq:TimeIntegrability1}
\int_0^\infty
\bigl(\|U_\eps(t)\|_{\dot H^N}+\|U_\eps(t)\|_{L^\infty}\bigr)\,{\rm d}t
\leq C\delta_0,
\end{equation}
and
\begin{equation}\label{Eq:TimeIntegrability2}
\int_0^\infty\|U_\eps(t)\|_{L^\infty}
\|\nabla U_\eps(t)\|_{H^N}\,{\rm d}t
\leq C\delta_0^2.
\end{equation}
\end{lem}

\begin{proof}
Choosing \(s\in(5/2,N-2)\) in \eqref{TimeDecay} and using the Besov embedding theorem, we obtain
\begin{align}\label{U_eps_L_infty}
    \|U_\eps(t)\|_{L^\infty}\leq C_s\|U_\eps(t)\|_{\dot B^s_{2,1}}
    \leq C_s(1+t)^{-(s-1/2)/2}\delta_0,
\end{align}
which implies \(U_\eps\in L^1_tL^\infty_x\).  
Taking $m=N$ in \eqref{Eq:HmDecay} gives
\begin{align}
    \|U_\eps(t)\|_{\dot H^N}
\leq C(1+t)^{-(N-1/2)/2}\delta_0,
\end{align}
which yields \(U_\eps\in L^1_tH^N\).  Thus, \eqref{Eq:TimeIntegrability1} holds.

Meanwhile, the estimate \eqref{U_eps_L_infty} also implies
\(U_\eps\in L^2_tL^\infty_x\). By virtue of the dissipative estimate \eqref{MainEng-intro} and Cauchy--Schwarz inequality, we obtain
\begin{align*}
    \int_0^\infty\|U_\eps(t)\|_{L^\infty}\|\nabla U_\eps(t)\|_{H^N}\,{\rm d}t
    \leq \|U_\eps\|_{L^2_tL^\infty} \|\nabla U_\eps\|_{L^2_t H^N}
    \leq C\delta_0^2,
\end{align*}
which is \eqref{Eq:TimeIntegrability2}. Thus, we complete the proof of this lemma.
\end{proof}

Recall the definitions of $\Psi$, $\Phi$, and $\gamma_0$ in \eqref{def-Psi-Phi-intro}:
\begin{align*}
    \Psi(\rho)=\rho^{-1},\quad \Phi(\rho)=\frac{p'(\rho)}{\rho}, \quad \gamma_0=\Phi(\rho_\infty)>0.
\end{align*}
Define \(\bar\Psi:=\rho_\infty^{-1}\), and split $f_\epsilon$ 
and $\mathcal{G}$ into
\begin{equation*}
f_\eps=f_\eps^{\rm d}+f_\eps^{\rm s},
\qquad \mathcal G=\mathcal G^{\rm d}+\mathcal G^{\rm s},
\end{equation*}
where
\begin{align*}
f_\eps^{\rm d}:=&-\operatorname{div}\bigl(\sigma_\eps(u_\eps^*+w_\eps)\bigr),\quad
f_\eps^{\rm s}:=-\eps\operatorname{div}(\sigma_\eps^*w_\eps),\\
\mathcal G^{\rm d}:=&-u_\eps^*\cdot\nabla w_\eps
-w_\eps\cdot\nabla u_\eps^*-w_\eps\cdot\nabla w_\eps
+\bigl(\Psi(\rho_\eps)-\Psi(\rho_\eps^*)\bigr)
\A(u_\eps^*+w_\eps)
\\
&-\frac{\Phi(\rho_\eps)-\Phi(\rho_\eps^*)}{\eps}
\nabla\sigma_\eps
-\bigl(\Phi(\rho_\eps)-\Phi(\rho_\eps^*)\bigr)
\nabla \sigma_\eps^*,
\\
\mathcal G^{\rm s}:=&
\bigl(\Psi(\rho_\eps^*)-\bar\Psi\bigr)\A w_\eps
-\frac{\Phi(\rho_\eps^*)-\gamma_0}{\eps}\nabla\sigma_\eps.
\end{align*}
Since $p\in C^\infty$, $\rho_\infty>0$, and $\rho_\eps^*=\rho_\infty+\eps^2\sigma_\eps^*$, we have
\begin{equation}\label{Eq:StationaryCoefficientSize}
\Psi(\rho_\eps^*)-\bar\Psi=O(\eps^2\sigma_\eps^*),
\qquad
\frac{\Phi(\rho_\eps^*)-\gamma_0}{\eps}=O(\eps \sigma_\eps^*).
\end{equation}
Define \(\Phi_\eps^{\rm d}\) and \(\Phi_\eps^{\rm s}\) by replacing \((f_\eps,\mathcal G)\) in \eqref{Eq:AcousticSymmetricUnknown} with the corresponding split pair:
\begin{align*}
    \Phi_\eps^{\rm d}:=\begin{pmatrix}\K^{1/2}f_\eps^{\rm d}\\[1mm]\sqrt{\rho_\infty}\Lambda^{-1}\operatorname{div}\mathbb Q\mathcal G^{\rm d}\end{pmatrix},
    \qquad
    \Phi_\eps^{\rm s}:=\begin{pmatrix}\K^{1/2}f_\eps^{\rm s}\\[1mm]\sqrt{\rho_\infty}\Lambda^{-1}\operatorname{div}\mathbb Q\mathcal G^{\rm s}\end{pmatrix}.
\end{align*}
Then, for these acoustic source terms, we have the following estimates.

\begin{lem}[Acoustic source estimates]\label{Lem:AcousticSource}
Assume that \eqref{Eq:IndexRange} holds and set
\(X_\eps:=\|V_\eps\|_{L^r_t\dot B^s_{p,1}}\).  Then we have
\begin{align}
&\|(\Phi_\eps^{\rm d})^\ell\|_{L^1_t\dot B^{1/2}_{2,\infty}}
+\|(\Phi_\eps^{\rm d})^h\|_{L^1_tH^{N-1}}
\leq C\delta_0^2,
\label{Eq:DecayingSource}\\
&\|(\Phi_\eps^{\rm s})^\ell\|_{L^1_t\dot B^{1/2}_{2,\infty}}
\leq C\eps\delta_0^2,
\label{Eq:StationaryLowSource}\\
&\|(\Phi_\eps^{\rm s})^h\|_{L^r_t\dot B^{s-2}_{p,1}}
\leq C\eps\delta_0^2+C\eps\delta_0X_\eps.
\label{Eq:StationaryHighSource}
\end{align}
\end{lem}

\begin{proof}
At the high-frequency region, \(\K^{1/2}\operatorname{div}\) has order two. Then, by Theorem \ref{LMSNSKThm-intro}, we have
\begin{align}
\|(\K^{1/2}f_\eps^{\rm d})^h\|_{H^{N-1}}
&=\|-(\K^{1/2}\dive (\sigma_\eps(u_\eps^*+w_\eps)))^h\|_{H^{N-1}}\nonumber\\
&\leq C\|\sigma_\eps(u_\eps^*+w_\eps)\|_{H^{N+1}}\nonumber\\
&\leq C(\delta_0\|\K^{1/2}\sigma_\eps\|_{H^N}
+\|\sigma_\eps\|_{L^\infty}\|w_\eps\|_{H^{N+1}}
+\|w_\eps\|_{L^\infty}\|\K^{1/2}\sigma_\eps\|_{H^N}).\label{4.64.1}
\end{align} 
Thanks to the estimate \eqref{Eq:TimeIntegrability1}, 
the first and third terms on the right-hand side of \eqref{4.64.1} are integrable over time. 
Applying Cauchy--Schwarz inequality in time and the estimates \eqref{MainEng-intro} and \eqref{Eq:TimeIntegrability2}, we obtain that the second term on the right-hand side of \eqref{4.64.1} belongs to $L^1_t$.

For \(\mathcal G^{\rm d}\), the transport terms have the bound
\begin{align*}
    \|-u_\eps^*\cdot\nabla w_\eps
    -w_\eps\cdot\nabla u_\eps^*-w_\eps\cdot\nabla w_\eps\|_{H^N}
    \leq C\delta_0(\|w_\eps\|_{H^N}+\|w_\eps\|_{L^\infty})+C\|w_\eps\|_{L^\infty}\|w_\eps\|_{H^N}.
\end{align*}
The remaining terms rely on
\begin{align*}
    \Psi(\rho_\eps)-\Psi(\rho_\eps^*)=O(\eps\sigma_\eps),\qquad
    \frac{\Phi(\rho_\eps)-\Phi(\rho_\eps^*)}{\eps}=O(\sigma_\eps),
\end{align*}
and their time integrability follows from the estimates
\eqref{Eq:TimeIntegrability1}--\eqref{Eq:TimeIntegrability2}.
Using \eqref{TimeDecay} with some
\(s\in(5/2,N-2)\), the same Bony decomposition yields the required low-frequency estimate for
\((\Phi_\eps^{\rm d})^\ell\), which proves \eqref{Eq:DecayingSource}.

By \eqref{Eq:StationaryCoefficientSize}, the low-frequency component of the stationary source part carries an explicit factor \(\eps\), and
\eqref{Eq:TimeIntegrability1} gives \eqref{Eq:StationaryLowSource}.  At the high-frequency region, we have
\begin{align*}
\|(\K^{1/2}f_\eps^{\rm s})^h\|_{\dot B^{s-2}_{p,1}}
&\leq C\eps \delta_0\|w_\eps\|_{\dot B^s_{p,1}},\\
\|(\mathcal G^{\rm s})^h\|_{\dot B^{s-2}_{p,1}}
&\leq C\eps^2\delta_0\|w_\eps\|_{\dot B^s_{p,1}}
+C\eps\delta_0\|\K^{1/2}\sigma_\eps\|_{\dot B^s_{p,1}}.
\end{align*}
Combining the Besov embedding
\(\dot B^{s+3\delta_p}_{2,1}\hookrightarrow\dot B^s_{p,1}\) and the time-decay estimate \eqref{TimeDecay}, we obtain
\begin{equation}\label{Eq:wLrBs}
\|w_\eps\|_{L^r_t\dot B^s_{p,1}}
\leq C\delta_0,
\end{equation}
where we have used \eqref{Eq:IndexRange} to ensure \(s+3\delta_p<N-2\), and the decay exponent is time-integrable because
\(s>1/2+2/r\).  Thus, we obtain \eqref{Eq:StationaryHighSource} and complete the proof of this lemma.
\end{proof}

With the aid of Lemmas \ref{Lem:TimeIntegrability}--\ref{Lem:AcousticSource}, we proceed to establish the convergence result.
\begin{prop}[Convergence of the acoustic component]\label{Prop:AcousticConvergence}
Under the assumption \eqref{Eq:IndexRange}, we have
\begin{equation}\label{Eq:StrongAcousticConvergence}
\|\K^{1/2}\sigma_\eps\|_{L^r_t\dot B^s_{p,1}}
+\|\mathbb Qw_\eps\|_{L^r_t\dot B^s_{p,1}}
\leq C\eps^{\beta(p,r)}\delta_0,
\end{equation}
and 
\begin{equation}\label{Eq:AcousticConvergence}
\|\sigma_\eps\|_{L^r_t\dot B^s_{p,1}}
+\|\mathbb Qw_\eps\|_{L^r_t\dot B^s_{p,1}}
\leq C\eps^{\beta(p,r)}\delta_0.
\end{equation}
\end{prop}

\begin{proof}
    By \eqref{Eq:AcousticDuhamel}, we directly get
    \begin{align}
        X_\eps&=\left\|S_\eps(t)V_{\eps,0}+\int_0^tS_\eps(t-\tau)\Phi_\eps(\tau)\,{\rm d}\tau\right\|_{L^r_t\dot B^s_{p,1}}\nonumber\\
        &\leq \|S_\eps(t)V_{\eps,0}\|_{L^r_t\dot B^s_{p,1}} +\left\|\int_0^tS_\eps(t-\tau)\Phi_\eps^{\rm d} (\tau)\,{\rm d}\tau\right\|_{L^r_t\dot B^s_{p,1}}
        +\left\|\int_0^tS_\eps(t-\tau)(\Phi_\eps^{\rm s})^{\ell}(\tau)\,{\rm d}\tau\right\|_{L^r_t\dot B^s_{p,1}}\nonumber\\
        &\qquad+\left\|\int_0^tS_\eps(t-\tau)(\Phi_\eps^{\rm s})^{h}(\tau)\,{\rm d}\tau\right\|_{L^r_t\dot B^s_{p,1}}.\label{4.68.1}
    \end{align}
    For the first three terms on the right-hand side of \eqref{4.68.1},  by Proposition \ref{Prop:AcousticStrichartz},  we deduce
    \begin{align*}
        &\|S_\eps(t)V_{\eps,0}\|_{L^r_t\dot B^s_{p,1}} +\left\|\int_0^tS_\eps(t-\tau)\Phi_\eps^{\rm d} (\tau)\,{\rm d}\tau\right\|_{L^r_t\dot B^s_{p,1}}        +\left\|\int_0^tS_\eps(t-\tau)(\Phi_\eps^{\rm s})^{\ell}(\tau)\,{\rm d}\tau\right\|_{L^r_t\dot B^s_{p,1}}\\
        \leq&\, C\eps^{\beta(p,r)}\big(\|V_{\eps,0}^{\ell}\|_{\dot B^{1/2}_{2,\infty}}+\|V_{\eps,0}^h\|_{H^M}+\|(\Phi_\eps^{\rm d})^{\ell}\|_{L^1_t\dot B^{1/2}_{2,\infty}}+\|(\Phi_\eps^{\rm d})^h\|_{L^1_tH^M}+\|(\Phi_\eps^{\rm s})^{\ell}\|_{L^1_t\dot B^{1/2}_{2,\infty}}\big).
    \end{align*}
    Applying Lemma \ref{Lem:HighDampedEstimate} to the high-frequency part of \(\Phi_\eps^{\rm s}\), we obtain
    \begin{align*}
        \left\|\int_0^tS_\eps(t-\tau)(\Phi_\eps^{\rm s})^{h}(\tau)\,{\rm d}\tau\right\|_{L^r_t\dot B^s_{p,1}}
        \leq C\eps^{-3\delta_p}\|(\Phi_\eps^{\rm s})^{h}\|_{L^r_t\dot B^{s-2}_{p,1}}.
    \end{align*}
    Let
\begin{equation*}
\theta_p:=1-3\delta_p=\frac3p-\frac12.
\end{equation*}
The index condition \eqref{Eq:IndexRange} implies
\begin{equation}\label{Eq:ThetaDominatesBeta}
\theta_p>\frac2r>\beta(p,r).
\end{equation}
Using the initial condition \eqref{initialcondition-intro} and Lemma \ref{Lem:AcousticSource}, we obtain
\begin{align*}
    X_\eps
\leq C\eps^{\beta(p,r)}\delta_0
+C\eps^{\theta_p}\delta_0^2
+C\eps^{\theta_p}\delta_0X_\eps.
\end{align*}
Choose \(\eps_0\) and \(\delta_0\) sufficiently small such that the last term can be absorbed.  By
\eqref{Eq:ThetaDominatesBeta}, the second term is bounded by
\(C\eps^{\beta(p,r)}\delta_0\).  This gives the estimate for
\((\K^{1/2}\sigma_\eps,d_\eps)\). Since
\(\mathbb Qw_\eps=-\nabla\Lambda^{-1}d_\eps\) and
\(\K^{-1/2}\) is a zero-order multiplier at the low-frequency region and a smoothing multiplier at the high-frequency region, \eqref{Eq:StrongAcousticConvergence} and \eqref{Eq:AcousticConvergence} follow.
\end{proof}

Next, we analyze the incompressible projection error. Set
\begin{equation}\label{Eq:zDef}
z_\eps:=\mathbb Pw_\eps-\widetilde u,
\qquad
a_\eps:=u_\eps^*-u^*,
\qquad
b_\eps:=w_\eps-\widetilde u=z_\eps+\mathbb Qw_\eps.
\end{equation}
The stationary estimate \eqref{machest-intro} implies
\begin{equation}\label{Eq:aStationaryError}
\|a_\eps\|_{\dot B^{1/2}_{2,\infty}\cap H^{N+3}}
\leq C\eps^2\delta_0^2,
\end{equation}
and
\begin{equation}\label{Eq:zInitial}
z_\eps(0)=u^*-\mathbb Pu_\eps^*.
\end{equation}

Substituting the continuity equation into the compressible momentum equation in \eqref{NNSK-intro}, and dividing the resulting identity by $\rho_\eps$, we obtain 
\begin{align}
\partial_tu_\eps +u_\eps\cdot\nabla u_\eps +\frac{\nabla p(\rho_\eps)}{\eps^2\rho_\eps}-\frac{\mu}{\rho_\eps}\Delta u_\eps-\frac{\nu}{\rho_\eps}\nabla\dive u_\eps-\frac{\kappa}{\eps^2}\nabla\Delta\rho_\eps
=F.\label{4.73.1}
\end{align}
Denote
\(\bar\mu:=\mu/\rho_\infty\), \(\bar\Psi:=\rho_\infty^{-1}\), and
\begin{align*}
\nabla \Theta_\eps:=-\frac{\Psi(\rho_\eps)}{\eps^2}\nabla p(\rho_\eps)+\frac{\kappa}{\eps^2}\nabla\Delta\rho_\eps+\bar\Psi\nu\nabla\dive u_\eps,
\end{align*}
then the equation \eqref{4.73.1} can be written as
\begin{equation}
\partial_tu_\eps-\bar\mu\Delta u_\eps
=-u_\eps\cdot\nabla u_\eps
+(\Psi(\rho_\eps)-\bar\Psi)\A u_\eps
+F+\nabla\Theta_\eps.\label{4.74.1}
\end{equation}
Applying the Helmholtz projection \(\mathbb{P}\) to \eqref{4.74.1} and
subtracting the stationary projected equation, we deduce
\begin{align*}
    \partial_t \mathbb P w_\eps-\bar\mu\Delta\mathbb Pw_\eps
    &=-\mathbb P(u_\eps^*\cdot\nabla w_\eps+w_\eps\cdot\nabla u_\eps^*
    +w_\eps\cdot\nabla w_\eps)\\
    &\qquad+\bigl(\Psi(\rho_\eps)-\Psi(\rho_\eps^*)\bigr)\A u_\eps^*+\bigl(\Psi(\rho_\eps)-\bar\Psi\bigr)\A w_\eps,
\end{align*}
which, compared with \eqref{Eq:IncompPerturbation}, gives
\begin{equation*}
\partial_tz_\eps-\bar\mu\Delta z_\eps
=-\mathbb P\mathcal C_\eps+\mathbb P\mathcal V_\eps,
\end{equation*}
where
\begin{align}
\mathcal C_\eps:={}&
\bigl(u_\eps^*\cdot\nabla w_\eps
+w_\eps\cdot\nabla u_\eps^*
+w_\eps\cdot\nabla w_\eps\bigr)
\notag-\bigl(u^*\cdot\nabla\widetilde u
+\widetilde u\cdot\nabla u^*
+\widetilde u\cdot\nabla\widetilde u\bigr),
\notag\\
\mathcal V_\eps:={}&
\bigl(\Psi(\rho_\eps)-\Psi(\rho_\eps^*)\bigr)\A u_\eps^*
+\bigl(\Psi(\rho_\eps)-\bar\Psi\bigr)\A w_\eps.
\label{Eq:Verror}
\end{align}
By virtue of the expansions
\(u_\eps^*=u^*+a_\eps\) and \(w_\eps=\widetilde u+b_\eps\), we obtain
\begin{align}
\mathcal C_\eps={}&u^*\cdot\nabla b_\eps+b_\eps\cdot\nabla u^*
+\widetilde u\cdot\nabla b_\eps+b_\eps\cdot\nabla\widetilde u
+b_\eps\cdot\nabla b_\eps
\notag\\
&+a_\eps\cdot\nabla\widetilde u
+\widetilde u\cdot\nabla a_\eps
+a_\eps\cdot\nabla b_\eps+b_\eps\cdot\nabla a_\eps.
\label{Eq:CerrorExpanded}
\end{align}
Thus, every term in \eqref{Eq:CerrorExpanded} contains \(b_\eps=z_\eps+\mathbb Qw_\eps\) or the stationary error \(a_\eps=O(\eps^2)\).

Recall that
\begin{equation*}
\mathcal X=\dot B^{1/2}_{2,\infty}\cap\dot B^{3/2}_{2,1}.
\end{equation*}
The estimates already proved yield
\begin{equation}\label{Eq:AuxiliaryXBounds}
\sup_{t\geq0}
\|(w_\eps,\widetilde u,u_\eps^*,u^*)\|_{\mathcal X}
+\|(w_\eps,\widetilde u)\|_{L^r_t\dot B^s_{p,1}}
\leq C\delta_0.
\end{equation}
{Indeed, the estimates \eqref{TimeDecay}
and \eqref{Eq:IncompDecay}} provide the global uniform-in-$\epsilon$ a priori bound
\begin{align*}
    \sup_{t\geq0}\|(w_\eps,\widetilde u,u_\eps^*,u^*)\|_{\dot B^{3/2}_{2,1}}\leq C\delta_0,
\end{align*}
and the \(L^r_t\dot B^s_{p,1}\) bound follows from the decay estimates and the embedding $\dot B^{s+3\delta_p}_{2,1}\hookrightarrow\dot B^s_{p,1}$.

With the estimate \eqref{Eq:AuxiliaryXBounds} in hand, we are in a position to bound the error source.

\begin{lem}[Error source estimates]\label{Lem:ErrorSources}
Under the assumption\eqref{Eq:IndexRange}, we have
\begin{align}
\|\mathcal C_\eps\|_{L^r_t\dot B^{s-2}_{p,1}}
&\leq C\delta_0\bigl(
\|z_\eps\|_{L^r_t\dot B^s_{p,1}}
+\|\mathbb Qw_\eps\|_{L^r_t\dot B^s_{p,1}}\bigr)
+C\eps^2\delta_0,
\label{Eq:CerrorEstimate}\\
\|\mathcal V_\eps\|_{L^r_t\dot B^{s-2}_{p,1}}
&\leq C\eps\delta_0.
\label{Eq:VerrorEstimate}
\end{align}
\end{lem}

\begin{proof}
Notice the decomposition of $\mathcal C_\eps$ in \eqref{Eq:CerrorExpanded}. By Lemma \ref{Lem:HybridProduct} and \eqref{Eq:AuxiliaryXBounds}, each term containing $b_\eps$ is bounded by $C\delta_0\bigl(\|z_\eps\|_{L^r_t\dot B^s_{p,1}}+\|\mathbb Qw_\eps\|_{L^r_t\dot B^s_{p,1}}\bigr)$. Since every remaining term contains $a_\eps$, it follows from \eqref{Eq:aStationaryError} that these remaining terms can be bounded by $C\eps^2\delta_0$. Thus, we have \eqref{Eq:CerrorEstimate}.

 For $\mathcal V_\eps$, we shall make use of
    \begin{align*}
        \Psi(\rho_\eps)-\Psi(\rho_\eps^*)=O(\eps\sigma_\eps),
        \qquad\Psi(\rho_\eps)-\bar\Psi=O(\eps\sigma_\eps+\eps^2\sigma_\eps^*).
    \end{align*}
The first term in \eqref{Eq:Verror} is estimated by multiplying the smooth stationary factor $\A u_\eps^*$ with $\eps\sigma_\eps$.
The second term follows from
\eqref{Eq:HybridViscosity}. Then the inequality \eqref{Eq:AuxiliaryXBounds} gives \eqref{Eq:VerrorEstimate}.
\end{proof}

\begin{prop}[Convergence of the incompressible projection]\label{Prop:IncompressibleError}
Under the assumption \eqref{Eq:IndexRange}, we have
\begin{equation}\label{Eq:IncompressibleErrorConvergence}
\|\mathbb Pw_\eps-\widetilde u\|_{L^r_t\dot B^s_{p,1}}
\leq C\eps^{\beta(p,r)}\delta_0.
\end{equation}
\end{prop}

\begin{proof}
The dyadic heat estimate gives the maximal-regularity bound
\begin{align}
\|z_\eps\|_{L^r_t\dot B^s_{p,1}}
\leq{}&C\|z_\eps(0)\|_{\dot B^{s-2/r}_{p,1}}
+C\|\mathcal C_\eps\|_{L^r_t\dot B^{s-2}_{p,1}}
+C\|\mathcal V_\eps\|_{L^r_t\dot B^{s-2}_{p,1}}.
\label{Eq:HeatMaximalRegularity}
\end{align}
Indeed, on the $j$-th dyadic block, the heat kernel has \(L^1_t\) norm \(C2^{-2j}\), and the initial heat evolution has
\(L^r_t\) norm \(C2^{-2j/r}\). 
Combining \eqref{machest-intro}, \eqref{Eq:zInitial}, low-high interpolation, and the embedding
\(\dot B^{s-2/r+3\delta_p}_{2,1}\hookrightarrow\dot B^{s-2/r}_{p,1}\), we obtain
\begin{equation}\label{Eq:zInitialEstimate}
\|z_\eps(0)\|_{\dot B^{s-2/r}_{p,1}}
\leq C\eps^2\delta_0.
\end{equation}
Substituting \eqref{Eq:zInitialEstimate} and the estimes in Lemma \ref{Lem:ErrorSources} into
\eqref{Eq:HeatMaximalRegularity}, choosing \(\delta_0\) small enough to absorb the term involving \(z_\eps\), and using Proposition \ref{Prop:AcousticConvergence}, we obtain
\begin{align*}
    \|z_\eps\|_{L^r_t\dot B^s_{p,1}}
\leq C(\eps^2+\eps+\eps^{\beta(p,r)})\delta_0
\leq C\eps^{\beta(p,r)}\delta_0,
\end{align*}
which, together with \eqref{Eq:zDef}, gives \eqref{Eq:IncompressibleErrorConvergence}.
\end{proof}

\subsection{Proof of Theorem \ref{NNSKThm-intro}}
Finally, based on the conclusions obtained in Subsections \ref{S4.2}  and the convergence bounds in Subsection \ref{S4.4}, we complete the proof of Theorem \ref{NNSKThm-intro}.

\begin{proof}
For the global existence and uniqueness of $(\rho_\epsilon,u_\epsilon)$, the estimate \eqref{MainEng-intro}, and the time-decay estimate \eqref{1.12.1}, we have shown their detailed proofs in Subsection \ref{S4.2}.
For the low Mach number limit, Proposition \ref{Prop:AcousticConvergence} gives the estimates for \(\sigma_\eps\) and \(\mathbb Qw_\eps\), while  Proposition \ref{Prop:IncompressibleError} gives the estimate for
\(\mathbb Pw_\eps-\widetilde u\).  Their sum yields \eqref{NNSKest-intro}. Thus, we complete the proof of Theorem \ref{NNSKThm-intro}. 
\end{proof}

\bigskip 
{\bf Acknowledgements:} 
J. Ni would like to express his gratitude to Prof. Renjun Duan for some communications on the topic of stationary solutions this year.
L. Wang is partially supported by NSFC (Grant No. 12601409), Basic Research Program of Jiangsu (Grant No. BK20240058),  China Postdoctoral Science Foundation (Grant No. 2024M751365) and Jiangsu Funding Program for Excellent Postdoctoral Talent (Grant No. 2023ZB071).
Z. Zhang  is partially supported by NSFC (Grant Nos. 12471215 and 12331007) and Taishan Scholars Program (tsqn202507101).

\vspace{2mm}

\textbf{Conflict of interest.}  The authors do not have any possible conflicts of interest.

\vspace{2mm}

\textbf{Data availability statement.}
 Data sharing is not applicable to this article as no data sets were generated or analyzed during the current study.

\bibliographystyle{plain}

\begin{thebibliography}{aaa}

 

	










\bibitem{Al-2006}
T. Alazard,
  Low Mach number limit of the full Navier--Stokes equations,
  \emph{Arch. Ration. Mech. Anal.} \textbf{180 (1)} (2006), 1--73.


\bibitem{An-Mc-Wh-1998}
D. M. Anderson, G. B. McFadden, A. A. Wheeler,
 Diffuse-interface methods in fluid mechanics,
 \emph{Annu. Rev. Fluid Mech.} \textbf{30} (1998), 139--165.



\bibitem{BCD-Book-2011}
H. Bahouri, J.-Y. Chemin, R. Danchin,
\newblock \emph{Fourier Analysis and Nonlinear Partial Differential Equations},
\newblock Grundlehren der Mathematischen Wissenschaften, vol. 343, Springer, Heidelberg, 2011.


\bibitem{BYZ-2014}
D. Bian, L. Yao and C. Zhu, Vanishing capillarity limit of the compressible fluid models of Korteweg type to the Navier--Stokes equations. \emph{SIAM J. Math. Anal.} \textbf{46} (2014), 1633--1650.



\bibitem{Ch-2014}
F. Charve,
 Local in time results for local and non-local capillary Navier--Stokes systems with large data,
 \emph{J. Differential Equations} \textbf{256 (7)} (2014), 2152--2193.

\bibitem{Ch-Da-Xu-2021}
F. Charve, R. Danchin, J. Xu,
 Gevrey analyticity and decay for the compressible Navier--Stokes system with capillarity,
 \emph{Indiana Univ. Math. J.} \textbf{70 (5)} (2021), 1903--1944.

\bibitem{Ch-Ha-2011}
F. Charve, B. Haspot,
  Convergence of capillary fluid models: from the non-local to the local Korteweg model,
  \emph{Indiana Univ. Math. J.} \textbf{60 (6)} (2011), 2021--2059.

\bibitem{Ch-Zh-2014}
Z. Chen, H. Zhao,
\newblock Existence and nonlinear stability of stationary solutions to the full compressible Navier--Stokes--Korteweg system,
\newblock \emph{J. Math. Pures Appl. (9)} \textbf{101 (3)} (2014), 330--371.

\bibitem{Ch-Ko-2019}
N. Chikami, T. Kobayashi,
Global well-posedness and time-decay estimates of the compressible Navier--Stokes--Korteweg system in critical Besov spaces,
  \emph{J. Math. Fluid Mech.} \textbf{21 (2)} (2019), Paper No. 31.

 

\bibitem{Cu-Ok-Ts-2022}
J. Cunanan, T. Okabe, Y. Tsutsui,
  Asymptotic stability of stationary Navier--Stokes flow in Besov spaces,
  \emph{Asymptot. Anal.} \textbf{129 (1)} (2022), 29--50.

\bibitem{Da-2002}
R. Danchin, Zero Mach number limit in critical spaces for compressible Navier--Stokes equations,
\emph{Ann. Sci. \`Ec. Norm. Sup\'er. (4)} \textbf{35 (1)} (2002), 27--75.

\bibitem{Da-De-2001}
R. Danchin, B. Desjardins,
Existence of solutions for compressible fluid models of Korteweg type,
\emph{Ann. Inst. H. Poincar\'e Anal. Non Lin\'eaire} \textbf{18 (1)} (2001), 97--133.

\bibitem{Da-He-2016}
R. Danchin, L. He,
  The incompressible limit in $L^p$ type critical spaces,
  \emph{Math. Ann.} \textbf{366 (3--4)} (2016), 1365--1402.

\bibitem{Da-Mu-2017}
R. Danchin, P. B. Mucha,  Compressible Navier--Stokes system: large solutions and incompressible limit,
  \emph{Adv. Math.} \textbf{320} (2017), 904--925.

\bibitem{De-2024}
N. Deguchi,
  On the stability of stationary compressible Navier--Stokes flows in 3D,
  \emph{Math. Ann.} \textbf{390 (3)} (2024), 4361--4404.

\bibitem{De-2025}
N. Deguchi,
  Low Mach number limit for the compressible Navier--Stokes equation with a stationary force,
  \emph{J. Math. Pures Appl. (9)} \textbf{214} (2026), Paper No. 103951.

\bibitem{De-Gr-1999}
B. Desjardins, E. Grenier,
  Low Mach number limit of viscous compressible flows in the whole space,
  \emph{Proc. Roy. Soc. London Ser. A} \textbf{455} (1986) (1999), 2271--2279.

\bibitem{Du-Se-1985}
J. E. Dunn, J. Serrin,
  On the thermomechanics of interstitial working,
  \emph{Arch. Ration. Mech. Anal.} \textbf{88 (2)} (1985), 95--133.

\bibitem{Fe-No-2009}
E. Feireisl, A. Novotn\'y,
  \emph{Singular Limits in Thermodynamics of Viscous Fluids},
  Birkh\"auser, Basel, 2009.

\bibitem{Fu-2024}
M. Fujii,
  Low Mach number limit of the global solution to the compressible Navier--Stokes system for large data in the critical Besov space,
  \emph{Math. Ann.} \textbf{388 (4)} (2024), 4083--4134.

\bibitem{Fu-Li-2025}
M. Fujii, Y. Li,
Low Mach number limit for the global large solutions to the 2D Navier--Stokes--Korteweg system in the critical $\widehat{L^p}$ framework,
  \emph{Calc. Var. Partial Differential Equations} \textbf{64 (1)} (2025), Paper No. 29.

\bibitem{HLY-2024}
 K. Hao, Y. Li, R. Yin, Low Mach number limit of the full compressible Navier--Stokes--Korteweg equations with general initial data. \emph{Dyn. Partial Differ. Equ.} \textbf{21 (3)} (2024),  281--304.


\bibitem{Ha-Li-1994}
H. Hattori, D. Li,
  Solutions for two-dimensional system for materials of Korteweg type,  \emph{SIAM J. Math. Anal.} \textbf{25 (1)} (1994), 85--98.

\bibitem{Ha-Li-1996}
H. Hattori, D. Li,
 Global solutions of a high-dimensional system for Korteweg materials, \emph{J. Math. Anal. Appl.} \textbf{198 (1)} (1996), 84--97.

\bibitem{Ho-1998}
D. Hoff, The zero-Mach limit of compressible flows,
  \emph{Comm. Math. Phys.} \textbf{192 (3)} (1998), 543--554.

\bibitem{Is-1987}
H. Isozaki, Singular limits for the compressible Euler equation in an exterior domain,
\emph{J. Reine Angew. Math.} \textbf{381} (1987), 1--36.



\bibitem{JX-2022}
Q. Ju and J. Xu, Zero-Mach limit of the compressible Navier--Stokes--Korteweg equations. \emph{J. Math. Phys.} \textbf{63} (2022), Paper No. 111503, 21 pp.



\bibitem{Ka-Ko-Sh-2019}
K. Kaneko, H. Kozono, S. Shimizu,
  Stationary solution to the Navier--Stokes equations in the scaling invariant Besov space and its regularity,
  \emph{Indiana Univ. Math. J.} \textbf{68 (3)} (2019), 857--880.


\bibitem {commutator1} T. Kato, G. Ponce,  Commutator estimates and the Euler and Navier--Stokes equations,  {\it Commun. Pure Appl. Math.} {\bf 41} (1988), 891--907.

 

\bibitem{Ka-Sh-Xu-2021}
S. Kawashima, Y. Shibata, J. Xu,
  The $L^p$ energy methods and decay for the compressible Navier--Stokes equations with capillarity,
  \emph{J. Math. Pures Appl. (9)} \textbf{154} (2021), 146--184.


\bibitem {commutator2} C.-E. Kenig, G. Ponce, L. Vega, Well-posedness of the initial value problem for the Korteweg-de Vries equation, {\it J. Amer. Math. Soc.} {\bf 4} (1991), 323--347.

\bibitem{Kl-Ma-1981}
S. Klainerman, A. Majda,
 Singular limits of quasilinear hyperbolic systems with large parameters and the incompressible limit of compressible fluids,
  \emph{Comm. Pure Appl. Math.} \textbf{34 (4)} (1981), 481--524.

\bibitem{Kl-Ma-1982}
S. Klainerman, A. Majda,
  Compressible and incompressible fluids,
  \emph{Comm. Pure Appl. Math.} \textbf{35 (5)} (1982), 629--651.

\bibitem{Ko-Sv-2011}
A. Korolev, V. {S}ver\'ak,
 On the large-distance asymptotics of steady state solutions of the Navier--Stokes equations in 3D exterior domains,
 \emph{Ann. Inst. H. Poincar\'e Anal. Non Lin\'eaire} \textbf{28 (2)} (2011), 303--313.

\bibitem{Ko-1901}
D. J. Korteweg,
 Sur la forme que prennent les \`equations du mouvement des fluides si l'on tient compte des forces capillaires caus\'ees par des variations de densit\'e,
\emph{Arch. N\'eerlandaises Sci. Exactes Nat. S\'er. II} \textbf{6} (1901), 1--24.

\bibitem{Ko-2008}
M. Kotschote,
Strong solutions for a compressible fluid model of Korteweg type,
  \emph{Ann. Inst. H. Poincar\'e Anal. Non Lin\'eaire} \textbf{25 (4)} (2008), 679--696.

\bibitem{Ko-Sh-2023}
H. Kozono, S. Shimizu,
  Stability of stationary solutions to the Navier--Stokes equations in the Besov space,
  \emph{Math. Nachr.} \textbf{296 (5)} (2023), 1964--1982.



\bibitem{Li-2012}
Y. Li,
Global existence and optimal decay rate of the compressible Navier--Stokes--Korteweg equations with external force,
  \emph{J. Math. Anal. Appl.} \textbf{388 (2)} (2012), 1218--1232.

\bibitem{LY-2026}
Y. Li, R. Yin, Low mach number limit of the one-dimensional full compressible Navier--Stokes--Korteweg equations. \emph{Math. Ann.} \textbf{394 (2)} (2026), Paper No. 32, 49 pp.


\bibitem{Li-Yo-2016}
Y. Li, W.-A. Yong,
  Zero Mach number limit of the compressible Navier--Stokes--Korteweg equations,
  \emph{Commun. Math. Sci.} \textbf{14 (1)} (2016), 233--247.

\bibitem{Li-Ma-1998}
P.-L. Lions, N. Masmoudi,
  Incompressible limit for a viscous compressible fluid,
  \emph{J. Math. Pures Appl. (9)} \textbf{77 (6)} (1998), 585--627.

\bibitem{SL-2019}
 K. Sha, Y. Li, Low Mach number limit of the three-dimensional full compressible Navier--Stokes--Korteweg equations. 
 \emph{Z. Angew. Math. Phys.} \textbf{70 (6)} (2019), Paper No. 169, 16 pp.




\bibitem{Sh-Ta-2003}
Y. Shibata, K. Tanaka,
  On the steady flow of compressible viscous fluid and its stability with respect to initial disturbance,
  \emph{J. Math. Soc. Japan} \textbf{55 (3)} (2003), 797--826.

\bibitem{Sh-Ta-2007}
Y. Shibata, K. Tanaka,
  Rate of convergence of non-stationary flow to the steady flow of compressible viscous fluid,
  \emph{Comput. Math. Appl.} \textbf{53 (3--4)} (2007), 605--623.

\bibitem{Ta-Zh-2014}
Z. Tan, R. Zhang,
  Optimal decay rates of the compressible fluid models of Korteweg type,
  \emph{Z. Angew. Math. Phys.} \textbf{65 (2)} (2014), 279--300.

\bibitem{Ts-2016}
K. Tsuda,
\newblock Existence and stability of time periodic solution to the compressible Navier--Stokes--Korteweg system on $\mathbb R^3$,
\newblock \emph{J. Math. Fluid Mech.} \textbf{18 (1)} (2016), 157--185.





\bibitem{Uk-1986}
S. Ukai,
  The incompressible limit and the initial layer of the compressible Euler equation,
  \emph{J. Math. Kyoto Univ.} \textbf{26 (2)} (1986), 323--331.


\bibitem{Wa-Wa-2015}
W. Wang, W. Wang,
 Decay rates of the compressible Navier--Stokes--Korteweg equations with potential forces,
 \emph{Discrete Contin. Dyn. Syst.} \textbf{35 (1)} (2015), 513--536.

















\end{thebibliography}

\end{document}